\documentclass[letterpaper,12pt]{article}
\usepackage{tabularx} % extra features for tabular environment
\usepackage{amsthm}
\usepackage{amsmath, amssymb, bm, mathtools, mathdots}
\usepackage{floatrow}
\usepackage{bbm}
\usepackage{tikz-cd}
\usepackage{graphicx} % takes care of graphic including machinery
\usepackage[margin=1in,letterpaper]{geometry} % decreases margins
\usepackage{cite} % takes care of citations
\usepackage[all]{xy}
\usepackage{tikz}
\usetikzlibrary{arrows.meta}
\usepackage[final]{hyperref} % adds hyper links inside the generated pdf file
\hypersetup{
	colorlinks=true,       % false: boxed links; true: colored links
	linkcolor=blue,        % color of internal links
	citecolor=blue,        % color of links to bibliography
	filecolor=magenta,     % color of file links
	urlcolor=blue         
}
\usepackage{xcolor}
\newcommand{\R}{\mathbb{R}} %\R for reelle tal
\newcommand{\z}{\langle}
\newcommand{\y}{\rangle}
\newcommand{\ud}{\mathrm{d}}

\newtheorem{example}{Example}             % 整体编号

\newtheorem{theorem}{Theorem}[section]
\newtheorem{definition}{Definition}[section]

\newtheorem{proposition}{Proposition}
\newtheorem{lemma}{Lemma}
\newtheorem{corollary}{Corollary}
\newtheorem{remark}{Remark}

\newtheorem{assumption}{Assumption}

\begin{document}
\title{Uniform-in-time propagation of chaos for second order interacting particle systems}

%\author{Yun Gong, Zhenfu Wang and Pengzhi Xie}

\author{Yun Gong \footnote{Beijing International Center for Mathematical Research, Peking University, Beijing 100871, China. 1901110013@pku.edu.cn.} 
	\quad
	Zhenfu Wang\footnote{Beijing International Center for Mathematical Research, Peking University, Beijing 100871,  China. zwang@bicmr.pku.edu.cn.} \quad 
	Pengzhi Xie\footnote{Department of Finance and Control Sciences, School  of Mathematical Sciences, Fudan University, Shanghai 200433, China. 22110180046@m.fudan.edu.cn.}
}

\date{\today}
\maketitle

\abstract{We study the long time behavior of second order particle systems interacting through global Lipschitz kernels. Combining hypocoercivity method in \cite{Hypocoercivity} and relative entropy method in \cite{JFAJW}, we are able to overcome the degeneracy of diffusion in position direction by controlling the relative entropy and relative Fisher information together. This implies the uniform-in-time propagation of chaos through the strong convergence of all marginals. Our method works at the level of Liouville equation and relies on the log Sobolev inequality of equilibrium of Vlasov-Fokker-Planck equation.}

\tableofcontents

\section{Introduction}
\numberwithin{equation}{section}

\subsection{Framework}

In this article, we consider the stochastic second order particle systems for $N$ indistinguishable point-particles, subject to a confining external force $- \nabla V$ and an interacting kernel $K$,

\begin{equation}\label{Eq:particle system}
\left \{
\begin{aligned}
\ud x_i(t) &= v_i(t) \ud t, \\      \ud v_i(t) & = - \nabla V(x_i(t)) \ud t + \frac{1}{N} \sum_{j \neq i}K(x_i(t) - x_j(t)) \ud t - \gamma v_i(t) \ud t + \sqrt{2 \sigma } \ud B_t^i,
\end{aligned}
\right.
\end{equation}
where $i = 1,2,...,N$. We take the position $x_i$ in $\Omega$, which may be the whole space $\R^d$ or the periodic torus $\mathbb{T}^d$, while each velocity  $v_i$ lies in $\R^d$. Note that  $\{ (B_\cdot^i ) \}_{i=1}^N$ denotes  $N$ independent copies of Wiener processes on $\R^d$. We take the diffusion coefficient before the Brownian motions $\sqrt{2 \sigma}$ as a constant for simplicity. Usually one can take  $\sigma = \gamma \beta^{-1}$, where the constant $\gamma > 0$ denotes the friction parameter, and $\beta$ is the inverse temperature. In more general models, those $\sigma$  may depend on the number of particles $N$, the position of particles $x_i$, etc. 

%we emphasize that if we select $\sigma = \sigma_N \rightarrow 0$, as $N \rightarrow 0$, we approximate the deterministic system. 

In many important models, the kernel is given  by $K = -  \nabla W$, where $W$ is an interacting potential. A well-known example for $W$ is  the Coulomb potential. We recall that  the Hamiltonian energy $H: (\Omega \times \R^d)^{N} \rightarrow \R$ associated to the particle system  \eqref{Eq:particle system} reads as 
\begin{equation}\label{Eq:Hamiltonian energy}
H(x_1,v_1,...,x_N,v_N) = \frac{1}{2}\sum_{i=1}^{N} v_{i}^2 + \sum_{i=1}^N V(x_i) + \frac{1}{2N} \sum_{i=1}^N \sum_{j \neq i } W(x_i - x_j),
\end{equation}
which is the sum of the kinetic energy $\frac{1}{2} \sum_{i=1}^N v_i^2$ and the potential energy $U(x_1,...,x_N)$ defined by
\begin{equation}\label{Eq:potential}
U(x_1,...,x_N):= \sum_{i=1}^N V(x_i) + \frac{1}{2N} \sum_{i=1}^N \sum_{j \neq i } W(x_i - x_j). 
\end{equation}

We recall the Liouville equation as  in \cite{JFAJW}, which describes the joint distribution $f_N$ of the particle system  \eqref{Eq:particle system} on $(\Omega \times \R^d)^N$, 
\begin{equation}\label{Eq:Liouville equation}
\begin{split}
\partial_t f_N + 
\sum_{i=1}^{N} v_i \cdot \nabla_{x_i} f_N & + \sum_{i=1}^N \Big( - \nabla_{x_i} V(x_i) + \frac{1}{N}\sum_{j \neq i} K(x_i - x_j) \Big) \cdot \nabla_{v_i} f_N \\ & = \sigma \sum_{i=1}^N \Delta_{v_i} f_N + \gamma \sum_{i=1}^{N} \nabla_{v_i} \cdot (v_i f_N).
\end{split}
\end{equation}
We define the associated Liouville operator as
\begin{equation}\label{Liouville operator}
\begin{split}
L_N f_N = & \sum_{i=1}^N v_i \cdot \nabla_{x_i}f_N + \sum_{i=1}^N \Big( -\nabla_{x_i} V(x_i) + \frac{1}{N}\sum_{j \neq i} K(x_i - x_j) \Big) \cdot \nabla_{v_i}f_N \\ & - \sigma \sum_{i=1}^N \Delta_{v_i} f_N - \gamma \sum_{i=1}^{N} \nabla_{v_i} \cdot (v_i f_N).
\end{split}
\end{equation}
As the mean field theory indicates,  when  $N \rightarrow \infty$, the limiting behavior of any  particle in \eqref{Eq:particle system} is described by the McKean-Vlasov SDE 
\begin{equation}\label{Eq:MV equation}
\left \{
\begin{aligned}
\ud  x(t) &= v(t) \ud  t, \\        \ud v(t) & = - \nabla V(x(t)) + K \ast \rho(x(t)) \ud t - \gamma v(t) \ud t + \sqrt{2\sigma} \ud B_t, 
\end{aligned}
\right.
\end{equation}
where $(x,v) \in \Omega \times \R^d$, $(B_\cdot)$ denotes a standard Brownian motion on $\R^d$. We then denote its phase space density by $f_t = \mbox{Law}(x(t), v(t))$  and the spatial density by 
\begin{equation*}
\rho(t,x) = \int_{\R^d} f(t,x,v)  \ud v. 
\end{equation*}
The law  $f_t$ further satisfies the mean-field equation or the nonlinear Fokker-Planck equation
\begin{equation}\label{Eq:limiting equation}
\partial_t f + v \cdot \nabla_x f - \nabla V \cdot \nabla_v f + (K \ast \rho) \cdot \nabla_vf = \sigma \Delta_vf + \gamma \nabla_v \cdot (vf).
\end{equation}
If $\sigma = \gamma \beta^{-1}$, one can check  that its nonlinear equilibrium satisfies the following equation
\begin{equation}\label{Eq:nonlinear equilibrium}
f_{\infty} = \frac{1}{Z} e^{- \beta (V(x) + \frac{1}{2}v^2 + W \ast \rho_{\infty})}, 
\end{equation}
where $Z$ is the constant to make $f_\infty$ a probability density. 
In the literature, one may use the ``formal equilibrium" of Eq.\eqref{Eq:limiting equation} at time $t$ to refer 
\begin{equation}\label{Eq:foraml nonlinear equilibrium}
\hat{f}_t = \frac{1}{\hat Z} e^{- \beta (V(x) + \frac{1}{2}v^2 + W \ast \rho_t)}, 
\end{equation}
again where $ \hat Z$ is the constant to make $\hat f_t$ a probability density.

It is well-known that the large $N$ limit of particle system \eqref{Eq:particle system} is mathematically formalized by the notion of $propagation \ of \  chaos$ originating in \cite{1956Foundations}. See also the surveys for instance  \cite{Alain1991Topics,golse2016dynamics,jabin2014review,2017Mean}.  Recently, many important works have been done in quantifying propagation of chaos for different kinds of first order interacting particle systems with singualr kernels. See for instance  \cite{2017Mean,Chaintron_2022-2} for detailed descriptions  of recent development. However, for 2nd order systems or Newton's dynamcis for interacting particles, the natural question is to consider the mean field limit of a purely deterministic problem, that is $\sigma=0$ as in Eq. \eqref{Eq:particle system} or our setting with the diffusion only acts on the velocity variables.  Thus the limiting equation \eqref{Eq:limiting equation} has the Laplacian term only in its velocity variable. Consequently, we cannot treat more general kernels $K$ and long time propagation of chaos for 2nd order systems by exploiting entropy dissipation as easy as those in the first-order setting for instance as in \cite{InventionJW}.    

The main purposes of this article are two-fold. Firstly, we study the long time convergence of  the solution to the N-particle Liouville equation \eqref{Eq:Liouville equation} toward its unique equilibrium (i.e. the Gibbs measure) uniformly in $N$. Secondly,  we establish propagation of chaos of  \eqref{Eq:particle system} (or \eqref{Eq:Liouville equation} )toward \eqref{Eq:MV equation} (or \eqref{Eq:limiting equation} ) uniformly in time. We combine the entropy method developed in \cite{JFAJW} with the hypocoercivity method refined in \cite{Hypocoercivity} to overcome the degeneracy of diffusion in position direction. Here we deal with the cases where the interaction kernel is globally Lipschitz and the strength of diffusion $\sigma$ is large enough. Uniform-in-time propagation of chaos cannot hold in full generality, one of the most critical issues is that the non-linear equilibrium of \eqref{Eq:limiting equation} might not be unique. Furthermore, for systems  \eqref{Eq:Liouville equation} with singular interactions $K$, for instance the Coulomb interactions, even the mean-field limit/propagation of chaos for general initial data on a fixed finite time horizon, for instance $t \in [0, T]$, is still widely open. The recent progress can be found for instance in \cite{hauray2007n,jabin2015particles,JFAJW,lazarovici2017mean,carrillo2018mean,BJJ,bresch2024duality}.   Uniform-in-$N$ convergence of \eqref{Eq:Liouville equation} toward its equilibrium is seldom investigated for systems even with very mild singularities. Those will be the objects of our further study. 

Let us briefly describe how we combine the relative entropy with hypocoercivity to establish propagation of chaos.  To illustrate the ideas, we pretend  all solutions are classical  and there is no regularity issue when we take derivatives. We first recall the evolution of relative entropy in \cite{JFAJW} for systems with bounded kernel $K$, 
\begin{equation}\label{IEq:evolution of RE1}
\begin{aligned}
\frac{\ud}{ \ud t} H_N(f^t_N|f_t^{\otimes N}) \leq & - \frac{1}{N} \int_{(\Omega \times \R^d)^N} f^t_N \overline{R}_N \ud Z \\ & - \frac{\sigma}{N} \int_{(\Omega \times \R^d)^N} f^t_N \bigg| \nabla_V \log \frac{f^t_N}{f_t^{\otimes N}} \bigg|^2 \ud Z,
\end{aligned}
\end{equation}
where $H_N(f^t_N|f_t^{\otimes N})$ is  the normalized relative entropy defined as 
\begin{equation*}
H_N(f^t_N|f_t^{\otimes N}) = \frac{1}{N} \int_{(\Omega \times \R^d)^N} f^t_N \log \frac{f^t_N}{f_t^{\otimes N}} \ud Z,
\end{equation*}
and $\overline{R}_N$ is the error term defined as
\begin{equation*}
\overline{R}_N = \sum_{i=1}^N \nabla_{v_i} \log f_t(x_i,v_i) \cdot \bigg\{ \frac{1}{N} \sum_{j,j\neq i} K(x_i - x_j) - K \ast \rho_t(x_i) \bigg\}.
\end{equation*}
 Hereinafter we write  that $Z= (x_1, v_1, x_2, v_2, \cdots, x_N, v_N) \in (\Omega \times \R^d)^N$, $X= (x_1, x_2, \cdots, x_N) \in \Omega^N$ and $V= (v_1, v_2, \cdots, v_N) \in (\mathbb{R}^d)^N$.   Then by integration by parts, one can rewrite \eqref{IEq:evolution of RE1} as
\begin{equation}\label{Eq:evolution of RE2}
\begin{aligned}
\frac{\ud}{ \ud t} H_N(f^t_N|f_t^{\otimes N}) = & - \frac{1}{N} \int_{(\Omega \times \R^d)^N} \nabla_V \log \frac{f^t_N}{f_t^{\otimes N}} \cdot R^0_N \ud Z \\ & - \frac{\sigma}{N} \int_{(\Omega \times \R^d)^N} f^t_N \left| \nabla_V \log \frac{f^t_N}{f_t^{\otimes N}} \right|^2 \ud Z,
\end{aligned}
\end{equation}
where $R^0_N$ is a $dN$-dimensional vector defined as 
\begin{equation*}
R^0_N = \bigg\{ \frac{1}{N} \sum_{j=1,j\neq i}^N K(x_i - x_j) - K \ast \rho_t(x_i) \bigg\}_{i=1}^N.
\end{equation*}
Inspired by the strategy used in \cite{Guillin2021UniformIT} to deal with the first order dynamics on torus, we expect to find the term
\begin{equation*}
- \frac{1}{N} \int_{(\Omega \times \R^d)^N} f^t_N \left| \nabla_X \log \frac{f^t_N}{f_t^{\otimes N}} \right|^2 \ud Z
\end{equation*}
in some way to use log-Sobolev inequality for the second order system \eqref{Eq:particle system}. Fortunately, hypocoercivity in entropy sense in \cite{Hypocoercivity} motivates us to take the derivative  of normalized relative Fisher information $I^M_N(f^t_N|f_t^{\otimes N})$, which  is defined as
\begin{equation*}
I_N^M(f^t_N|f_t^{\otimes N}) = \frac{1}{N} \int_{(\Omega \times \R^d)^N} f^t_N \left| \sqrt{M} \nabla \log \frac{f^t_N}{f_t^{\otimes N}} \right|^2 \ud Z,
\end{equation*}
where $M$ is a $2Nd \times 2Nd$ positive defined matrix. If we choose the matrix $M$ as the form
\begin{equation*}
M = 
\left(
\begin{array}{cc}
E  & F \\
F  & G
\end{array}
\right),
\end{equation*}
where $E,F,G$ are $Nd \times Nd$ positive definite matrices to be specified later, then we obtain
\begin{equation}\label{IEq:evolution of RFi}
\begin{aligned}
\frac{\ud}{\ud t} I_N^M(f^t_N|f_t^{\otimes N}) \leq & - \frac{c_1}{N} \int_{(\Omega \times \R^d)^N} f^t_N \left|\sqrt{E} \nabla_X \log \frac{f^t_N}{f_t^{\otimes N}} \right|^2 \ud Z \\ & +  
\frac{c_2}{N} \int_{(\Omega \times \R^d)^N} f^t_N \left| \sqrt{G} \nabla_V \log \frac{f^t_N}{f_t^{\otimes N}} \right|^2 \ud Z \\ & + \mbox{some  error terms},
\end{aligned}
\end{equation}
where $c_1, c_2$ are two positive constants that depend on $M, \sigma, \gamma$, $K$ and $V$. The concrete form of error terms will be given in Lemma \ref{lemma:Keylemma2} and  Lemma \ref{Lemma:CS-error}. Roughly speaking, time derivative  of  relative entropy only provides us dissipation in the direction of velocity and time derivative of relative Fisher information provides us further dissipation in the direction of position. The reference measure of \eqref{IEq:evolution of RE1} and \eqref{IEq:evolution of RFi} could be more general. For insatnce, if we replace $f_t^{\otimes N}$ by $f_{N,\infty}$, we can show the convergence from $f^t_N$ towards $f_{N,\infty}$ as $t \rightarrow \infty$. 

Now we combine these two quantities, relative entropy and Fisher information, to form a new ``modulated energy" as 
\begin{equation*}
\mathcal{E}^M_N(f^t_N|\bar{f_N}) = H_N(f^t_N|\bar{f_N}) + I^M_N(f^t_N|\bar{f_N}),
\end{equation*}
where $\bar{f_N}$ may take $f_t^{\otimes N}$ or $f_{N,\infty}$. The remaining argument is to appropriately control error terms in \eqref{IEq:evolution of RFi}. When $\bar{f_N} = f_{N,\infty}$, linear structure of the Liouville equation \eqref{Eq:Liouville equation} makes error terms vanish. When $\bar{f_N} = f_t^{\otimes N}$, error terms in \eqref{IEq:evolution of RFi} can be written as
\begin{equation} \label{IEq:error terms}
-2 \int_{(\Omega \times \R^d)^N} f^t_N \z \nabla \overline{R}_N, M \nabla \log \frac{f^t_N}{f_t^{\otimes N}}\y \ud Z
\end{equation}
for suitable selection of $M$ (See Corollary \ref{Cor:block matrix}), here $\nabla = (\nabla_X, \nabla_V)$. However, the terms $\nabla_x \nabla_v \log f_t$ and $\nabla_v \nabla_v \log f_t$ are too difficult to control  to obtain some uniform-in-time estimates for \eqref{IEq:error terms}  (See Remark \ref{remark: nabla^2-log-f}). The key observation is that $\nabla_x \nabla_v \log f_t$ and $\nabla_v \nabla_v \log f_t$ are trivial terms if we replace $f_t$ by its non-linear equilibrium, then we can show
\begin{equation}\label{IEq:evolution of E^M_N}
\begin{aligned}
\frac{\ud}{ \ud t} \mathcal{E}^M_N(f^t_N|f_{\infty}^{\otimes N}) \leq & - \frac{c'_1}{N} \int_{(\Omega \times \R^d)^N} f^t_N \left| \nabla_X \log \frac{f^t_N}{f_{\infty}^{\otimes N}} \right|^2 \ud Z \\ & -  
\frac{c'_2}{N} \int_{(\Omega \times \R^d)^N} f^t_N \left| \nabla_V \log \frac{f^t_N}{f_{\infty}^{\otimes N}} \right|^2 \ud Z \\ & + \frac{C}{N}, 
\end{aligned}
\end{equation}
where $c'_1, c'_2 > 0, C > 0$ are some constants that only depend on $M, \sigma, \gamma, K$ and $V$ (See Theorem \ref{Thm:theorem1.4}). To control the error between $\mathcal{E}^M_N(f^t_N|f_t^{\otimes N})$ and $\mathcal{E}^M_N(f^t_N|f_{\infty}^{\otimes N})$, we use the convergence result from $f_t$ to its equilibrium $f_{\infty}$. There exist some results studying this type of convergence . In this article,  we adapt the arguments as  in \cite{UPIandULSI} for $K$ with small $\| \nabla K \|_{\infty}$ and \cite{chen2023uniform} for the case when $K = - \nabla W$ and the  interaction energy is convex to establish the following estimate
\begin{equation}\label{IEq:convergence of limiting equation}
H(f_t|f_{\infty}) \leq Ce^{-c'_3t}, 
\end{equation}
where $c'_3 > 0$ depends on $\sigma, \gamma, K$ and $C > 0$ depends on initial data $f_0$.
We give the sketch of proof of \eqref{IEq:convergence of limiting equation} in Theorem \ref{Thm: theorem-1.1} and Theorem \ref{Thm:theorem-1.2} in the appendix. In the end, we combine \eqref{IEq:evolution of E^M_N} with \eqref{IEq:convergence of limiting equation} through 2-Wasserstein distance to conclude.    

Finally, we give some comments about terms $\nabla_x \nabla_v \log f_t$ and $\nabla_v \nabla_v \log f_t$. Unfortunately we cannot directly obtain  uniform-in-time estimates about these terms, otherwise we would  have 
\begin{equation*}
\begin{aligned}
\frac{\ud}{\ud t} \mathcal{E}^M_N(f^t_N|f_t^{\otimes N}) \leq & - \frac{c'_1}{N} \int_{(\Omega \times \R^d)^N} f^t_N \left| \nabla_X \log \frac{f^t_N}{f_t^{\otimes N}} \right|^2  \ud Z \\ & -  
\frac{c'_2}{N} \int_{(\Omega \times \R^d)^N} f^t_N \left| \nabla_V \log \frac{f^t_N}{f_t^{\otimes N}} \right|^2 \ud Z \\ & + \frac{C}{N}, 
\end{aligned}
\end{equation*}
for some $c_1', c_2', C > 0$ and then finish the proof by Gronwall's  inequality. Uniform-in-time propagation of chaos can only be expected when the limiting PDE \eqref{Eq:limiting equation} has a unique equilibrium. Indeed, we obtain Theorem \ref{Thm:theorem1.5} under the assumptions in  Theorem \ref{Thm:theorem1.4} and some assumption on $W$ which prevents the existence of multiple equilibria of Eq.\eqref{Eq:limiting equation}. Similar arguments/estimates have also appeared in the  first-order setting. For instance, the authors of  \cite{feng2023quantitative}  obtained Li-Yau type growth estimates of $\nabla^2 \log \rho_t$ with respect to the position $x$ for any fixed time horizon in particular in the case of 2d Navier-Stokes equation, with a universal constant growing with time $t$. Those type estimates are very useful when one treats the mean field limit problem set on the whole space \cite{rosenzweig2024relative,carrillo2024relative}.

\subsection{Main results and examples}\label{Main results and examples}
Let us fix some notations first. We denote $z = (x,v) \in \Omega \times \R^d$ as a phase space configuration of one particle, and $X = (x_1,...,x_N)$, $V = (v_1,...,v_N)$ and $Z= (X, V)$ for configurations in position, velocity  and phase space of $N$ particles, respectively. We also use that $z_i = (x_i,v_i)$ and $Z = (X,V) = (z_1,...,z_N)$. As the same spirit, we write  operator $\nabla_X = (\nabla_{x_1},..., \nabla_{x_N})$, $\nabla_V = (\nabla_{v_1},..., \nabla_{v_N})$ and $\Delta_V = \sum_{i=1}^{N} \Delta_{v_i}$. Those are operators on $\mathbb{R}^{dN}$. 

For two probability measures $\mu, \nu$ on $\Omega \times \R^d$,  we denote by $\Pi(\mu,\nu)$ theset of all  couplings between $\mu$ and $\nu$. We define $L^2$ Wasserstein distance as 
\begin{equation*}
\mathcal{W}_2(\mu,\nu) =  \Big( \inf_{\pi \in \Pi(\mu,\nu)} \int_{(\Omega \times \R^d)^2} |z_1 - z_2|^2 \ud \pi(z_1,z_2) \Big)^{1/2}. 
\end{equation*}
For a measure $\mu$ and a positive integrable function $f$ such that $\int_{\Omega \times \R^d} f|\log f| \ud \mu < \infty$, we define the  entropy of $f$ with respect to  $\mu$ as
\begin{equation*}
Ent_{\mu}(f) = \int_{\Omega \times \R^d} f \log f \ud \mu - \bigg( \int_{\Omega \times \R^d} f \ud \mu \bigg) \log \left( \int_{\Omega \times \R^d} f  \ud \mu \right). 
\end{equation*}
For two probability measures $\mu$ and $\nu$, we define the relative entropy between $\mu$ and $\nu$ as 
\begin{equation*}
H(\mu|\nu) = \left\{ 
\begin{aligned}
& \int_{\Omega \times \R^d} h \log h  \ud \nu,  & \mbox{ if \, } \mu << \nu  \mbox{ and }  h = \frac{\ud \mu }{\ud \nu }, \\ 
& + \infty,  &  \mbox{otherwise}, 
\end{aligned}
\right.
\end{equation*}
and the normalized version of relative entropy for probability measures $\mu' $ and $\nu'$ on $(\Omega \times \R^d)^N$ as 
\begin{equation}\label{Quantity:N-RE} 
H_N(\mu'|\nu') = \frac{1}{N}H(\mu'|\nu').
\end{equation}
We also define relative Fisher information between $\mu$ and $\nu$ as
\begin{equation*}
I^M(\mu|\nu) = \left\{
\begin{aligned}
&\int_{\Omega \times \R^d} \frac{\z M \nabla h, \nabla h\y}{h} \ud \nu, & \mbox{ if \, } \mu << \nu  \mbox{ and }  h = \frac{\ud \mu }{\ud \nu }, \\
&+ \infty & \ \ otherwise,
\end{aligned}
\right.
\end{equation*}
where $M(z)$ is a positive definite matrix-valued function such that $M(z) \geq \kappa I$ for some $\kappa > 0$ independent on $z \in \Omega \times \R^d$. Similarly, the normalized version of relative Fisher information between two probability measures $\mu' $ and $\nu'$ on $(\Omega \times \R^d)^N$ reads as
\begin{equation}\label{Quantity:N-RFI}
I^M_N(\mu'|\nu') = \frac{1}{N} I^M(\mu'|\nu').
\end{equation}
We abuse the notation of probability measures $f_N$ and $f$ as probability densities if they have densities. Finally, we give definitions of some functional inequalities that we will use in the following.

\begin{definition}(log-Sobolev inequality) We call that  the probability measure 
$\mu$ on $\Omega \times \R^d$ satisfies the log-Sobolev inequality  if there exists some constant $\rho_{ls}(\mu) > 0$ such that  for all smooth function $g$ with $\int g^2 \ud \mu = 1$, it holds that 
\begin{equation}
Ent_{\mu}(g^2) \leq \rho_{ls}(\mu) \int_{\Omega \times \R^d} (|\nabla_x g|^2 + |\nabla_v g|^2) \ud \mu.
\end{equation}
\end{definition}

\begin{definition}(Weighted log-Sobolev inequality) Let $H(z)$ be a smooth function on $\R^d \times \R^d$. We call that  the probability measure $\mu$ on $\R^d \times \R^d$ satisfies the weighted log-Sobolev inequality with weight $H$ if there exists some constant $\theta > 0, \rho_{wls}(\mu) > 0$  such that for all smooth function $g$ with $\int g^2 \ud \mu = 1$, it holds that 
\begin{equation}\label{}
Ent_{\mu}(g^2) \leq \rho_{wls}(\mu) \int_{\R^d \times \R^d}  (H^{-2\theta}|\nabla_x g|^2 + |\nabla_v g|^2)  \ud \mu.
\end{equation} 
\end{definition}
In the following, we use $\rho_{ls}$ ($\rho_{wls}$) to denote the (weighted) log-Sobolev inequality constant of the equilibrium  measure  $f_{\infty}$ of the limiting PDE \eqref{Eq:limiting equation} and  $\rho_{LS}$ to denote the uniform-in-$N$ log-Sobolev inequality constant of the stationary measure  $f_{N,\infty}$ of the Liouville equation \eqref{Eq:Liouville equation}.

Now we detail assumptions about interaction potential $W$ and confining potential $V$.

\begin{assumption}\label{Assumption:V}
Suppose that  $V(x) \in C^2(\Omega)$ and there exist  $\lambda > 0, M > 0$ such that $V(x) \geq \lambda |x|^2 - M$.
\end{assumption}
The first condition means that $V$ goes to infinity at infinity and is bounded below. It can be implied by
\begin{equation*}
\frac{1}{2} \nabla V (x) \cdot x \geq 6\lambda (V(x) + \frac{x^2}{2}) - A, \ \ x \in \Omega
\end{equation*}
for some $A > 0$. This expression implies that the force $-\nabla V$ drags particles back to some compact set. Detailed proof can be found in \cite{coupling2}.

The second assumption implies  that  the potential $V$ grows at most quadratically  on $\Omega$. 

\begin{assumption}\label{Assumption:V1}
Suppose that $V(x) \in C^2(\Omega)$ and there exists $C_V > 0$ such that $\| \nabla^2 V \|_{L^{\infty}} \leq C_V < \infty$.
\end{assumption}

We also treat more general confining potentials when $\Omega = \R^d$.

\begin{assumption}\label{Assumption:V2}
Suppose that $V(x) \in C^2(\R^d)$ and there exist $\theta > 0$ and $C_V^{\theta}$ such that $\|V^{-2\theta} \nabla^2 V\|_{L^{\infty}} \leq C_V^{\theta} < \infty$. Moreover, outside a compact domain on $\R^d$, we assume that $V$ satisfies  

$(i)$ $\Delta V \leq \kappa_1 |\nabla V|^2$ for some $\kappa_1 \in (0,1)$; 

$(ii)$ $|\nabla V|^2 \geq \kappa_2 V^{2\theta + 1}$ for some positive constant $\kappa_2 > 0$.
\end{assumption}

Those conditions in {\bf Assumption 3}  have been explored in \cite{EntropyMutiplier} for kinetic Langevin process with confining potentials greater than quadratic growth at infinity. The boundedness of  $\|V^{-2\theta} \nabla^2 V\|_{L^{\infty}}$ extend the quadratic growth condition and the other two conditions guarantee the weighted log-Sobolev inequality of $f_{\infty}$ (See Section \ref{section:WLSI}). We also use multiplier method developed in \cite{EntropyMutiplier} to deal with this type of confining potentials.

\begin{example}
Some important examples have been provided in \cite{EntropyMutiplier}. The first kind of examples is $V(x) = |x|^k, k \geq 2$. Then we have $\Delta_x V = (dk + k^2 - 2k)|x|^{k-2}$, $|\nabla_x V|^2 = k^2|x|^{2k-2}$ and $\|V^{-2\theta} \nabla^2 V \|_{L^{\infty}} \sim |x|^{k-2k\theta - 2}$. Finally,  we take $\theta = \frac{1}{2} - \frac{1}{k}$. Then all conditions above can be satisfied. 
\end{example}

\begin{example}
Another kind of examples is $V(x) = e^{a|x|^k}$ provided in \cite{EntropyMutiplier}, which shows that the limit growth of $V$ must be below the exponential growth. Observing that $\Delta_x V \sim a k^2|x|^{2(k-1)}e^{a|x|^k}$, $|\nabla_x V|^2 = a^2k^2e^{2a|x|^k}$ and $\|V^{-2\theta} \nabla^2 V \|_{L^{\infty}} \sim e^{a(1-2\theta)|x|^k}$, the conditions above imply $k < 1$ and $\theta = \frac{1}{2}$.
\end{example}

\begin{assumption}\label{Assumption:W1}
Suppose that $W(x) \in C^2(\Omega)$ and there exists $0 < C_K < \frac{1}{2} \lambda$ such that $\| \nabla^2 W\|_{L^\infty} \leq C_K < \infty$. 
\end{assumption}

\begin{assumption}\label{Assumption:W2}
The mean field functional or the interaction energy  $F: \mathcal{P}_2(\R^{d}) \rightarrow \R$ defined as
\begin{equation*}
F(\rho) = \int_{\Omega \times \Omega} W(x-y) \ud \rho(x) \ud \rho(y)
\end{equation*}
is functional convex, i.e. for every $t \in [0,1]$ and every $\nu_1, \nu_2 \in \mathcal{P}_2(\R^d)$,
\begin{equation}\label{Condition: convexity}
F((1-t)\nu_1 + t\nu_2) \leq (1-t)F(\nu_1) + t F(\nu_2).
\end{equation}
\end{assumption}

\begin{example}
The harmonic interaction potential 
\[ W(x) = \frac{L_W}{2}|x|^2 \] 
with $L_W \leq \frac{\lambda}{2}$ is covered by Assumption \ref{Assumption:W1}. The mollified Coulomb potential when $d=3$,
\begin{equation*}
W(x) = \frac{a}{(|x|^k + b^k)^{\frac{1}{k}}} \ \ \ or \ \ W(x) = \arctan(|x|/r_0) \frac{1}{|x|}
\end{equation*}
with some constant $a, b>0$ or $r_0 > 0$. The later form satisfies Assumption \ref{Assumption:W2} (See Section 3 in \cite{Chaintron_2022-1}).
\end{example}

\begin{remark}
The harmonic interaction potential does not satisfy Assumption \ref{Assumption:W2}.
Let us take 
\[ F(\rho) = \int_{\Omega \times \Omega} (x-y)^2 \ud \rho(x) \ud \rho(y), \]
then for $t \in [0,1]$,
\begin{equation*}
\begin{aligned}
& F((1-t)\nu_1 + t\nu_2) \leq (1-t)F(\nu_1) + t F(\nu_2) \\ \Longleftrightarrow & \int_{\Omega \times \Omega} (x-y)^2 (\ud \nu_1 - \ud \nu_2)^{\otimes 2}(x,y) \geq 0 \\ \Longleftrightarrow & 2 \big(\int_{\Omega} x^2 \ud \nu_1 \big) \big(\int_{\Omega} x^2 \ud \nu_2 \big) \geq \big(\int_{\Omega} x^2 \ud \nu_1 \big)^2 + \big(\int_{\Omega} x^2 \ud \nu_2 \big)^2,
\end{aligned}
\end{equation*}
which does not hold in general.
\end{remark}

With those specific statement of assumptions as above, we can now state our main results.  
\begin{theorem}\label{Thm:theorem1.3}
Suppose that $V$ satisfies Assumption \ref{Assumption:V} and \ref{Assumption:V1}, and $W$ satisfies Assumption \ref{Assumption:W1} with $C_K < 1$. Then for initial data $f_N^0$ of the Liouville equation \eqref{Eq:Liouville equation} 
such that  $\mathcal{E}^M_N(f^0_N|f_{N,\infty}) < \infty$, we have
\begin{equation}
\mathcal{E}^M_N(f^t_N|f_{N,\infty}) \leq e^{-ct} \mathcal{E}^M_N(f^0_N|f_{N,\infty}),
\end{equation}
where $c = c(f_{N,\infty}, M) > 0$ is explicit and independent of $N$.
\end{theorem}

\begin{remark} We can choose matrix $M$ as the following
\begin{equation*}
M = 
\left(
\begin{array}{cc}
E  & F \\
F  & G
\end{array}
\right),
\end{equation*}
where
\begin{equation*}\label{remark:Thm1.4-M}
E = \text{diag}\{\delta a^3,...,\delta a^3\}, \ \ \ F = \text{diag}\{\delta a^2,...,\delta a^2\}, \ \ \ G = \text{diag}\{2\delta a,...,2\delta a\},
\end{equation*}
and two constants $a$ and $\delta$ satisfy 
\begin{equation*}
a \leq \frac{2\gamma}{C_K + C_V}, \ \ \ \delta \leq \frac{\sigma}{2(4 + 8a \gamma)^2}. 
\end{equation*}
Then the constant $c$ can be taken as $c = \frac{1}{2(1 + \rho_{LS})} \min\{ \frac{3}{2} \delta a^2, \frac{\sigma}{2}\}$. By the selection of $\delta$, we observe that $c \sim \sigma$, i.e. the larger diffusion strength we have, the faster convergence from $f_N^t$ to $f_{N,\infty}$. 
\end{remark}

Our second contribution is the uniform-in-$N$ exponential convergence from $f_N$ to $f_{\infty}^{\otimes N}$.
\begin{theorem}\label{Thm:theorem1.4}
Suppose that  $V$ and $W$ satisfy one of the following two cases,

(i) $V$ satisfies Assumption \ref{Assumption:V} and \ref{Assumption:V1}, $W$ satisfies Assumption \ref{Assumption:W1};

(ii) $V$ satisfies Assumption \ref{Assumption:V} and \ref{Assumption:V2}, $W$ satisfies Assumption \ref{Assumption:W1} and $\|\nabla W \|_{L^{\infty}} < \infty$.

For the first case, we take $M_1$ as a constant matrix, then for initial data $f_N^0$ of Eq.\eqref{Eq:Liouville equation} such that $\mathcal{E}_N^{M_1}(f^0_N|f^{\otimes N}_{\infty}) < \infty$ and $\sigma \geq \sigma^{\ast}_1 > 0$, we have
\begin{equation}\label{IEq:Theorem 1.4-1}
H_N(f^t_N|f_{\infty}^{\otimes N}) \leq \mathcal{E}_N^{M_1}(f^t_N|f^{\otimes N}_{\infty}) \leq e^{-c_1t} \mathcal{E}_N^{M_1}(f^0_N|f^{\otimes N}_{\infty}) + \frac{C_1}{N},
\end{equation}
where $c_1 = c_1(f_{\infty}, M_1) > 0$ and $C_1 = C_1(f_{\infty}, \sigma, \gamma, C_K, C_V) > 0$ are explicit and independent of $N$. 

For the second case, we take $M_2$ as a matrix function on $\R^{2d}$, then for initial data $f_N^0$ of Eq.\eqref{Eq:Liouville equation} such that $\mathcal{E}_N^{M_2}(f^0_N|f^{\otimes N}_{\infty}) < \infty$ and $\sigma \geq \sigma^{\ast}_2 > 0$, we have
\begin{equation}\label{IEq:Theorem 1.4-2}
H_N(f^t_N|f_{\infty}^{\otimes N}) \leq \mathcal{E}_N^{M_2}(f^t_N|f^{\otimes N}_{\infty}) \leq e^{-c_2t} \mathcal{E}_N^{M_2}(f^0_N|f^{\otimes N}_{\infty}) + \frac{C_2}{N},
\end{equation}
where $c_2 = c_2(f_{\infty}, M_2), C_2 = C_2(f_{\infty}, \sigma, \gamma, C_K, C^{\theta}_V) > 0$ are explicit and independent of $N$.

\end{theorem}

\begin{remark}
For the first case, we also choose $M_1$ as in  Remark \ref{remark:Thm1.4-M}, but two constants $a$ and $\delta$ now should satisfy
\begin{equation*}
\left \{
\begin{aligned}
& a \leq \min \Big \{ \frac{2 \gamma}{C_K + C_V}, \frac{1}{4C_K + 2},  \frac{\gamma}{5120 e\rho_{ls} (C_K+1)^2} \Big \}, \\
& \delta \leq \frac{\sigma}{4[8 + a + 28a\gamma + 32a^2 \gamma^2]},
\end{aligned}
\right.
\end{equation*}
then the constant $c_1$ can be taken as $c_1 = \frac{\delta a^2}{16(\rho_{ls} + 1)}$ which implies $c_1 \sim \sigma$, and the lower bound of diffusion constant $\sigma^{\ast}_1$ satisfies 
\begin{equation*}
\sigma_1^{\ast} \geq \max \bigg\{ \frac{160[10 + 28 \gamma + 32 \gamma^2]\rho_{ls}e}{a^2 \gamma}, 3200\rho_{ls}e\gamma \bigg\}C_K^2.
\end{equation*}
For the second case, we choose $M_2$ as
\begin{equation*}
\left \{
\begin{aligned}
& E = \text{diag}\{e(z_1)Id_{d \times d},...,e(z_N)Id_{d \times d}\}, \\ & F = \text{diag}\{f(z_1)Id_{d \times d},...,f(z_N)Id_{d \times d}\}, \\ & G = \text{diag}\{g(z_1)Id_{d \times d},...,g(z_N)Id_{d \times d}\}, 
\end{aligned}
\right.
\end{equation*}
where $E,F,G$ are $Nd \times Nd$ diagonal matrices. We choose $e(z), f(z)$ and $g(z)$ as
\begin{equation*}
\begin{split}
e(z) = \delta a^3(H(z))^{-3\theta},\ \ 
b(z) = \delta a^2(H(z))^{-2\theta},\ \  c(z) = 2 \delta a(H(z))^{-\theta},
\end{split}
\end{equation*}
where 
\begin{equation*}
H(z) = \frac{v^2}{2} + V(x) + H_0, \ \ \ H_0 > 0,
\end{equation*}
and two constants $a$ and $\delta$ satisfy
\begin{equation*}
\left\{
\begin{aligned}
& a \leq \min \Big \{ \frac{1}{4C_K + 6\theta + 2}, \frac{\gamma}{C_V^{\theta} + C_K}, \frac{\gamma}{6400 e\rho_{wls} (C_K+1)^2} \Big \}, \\
& \delta \leq \frac{3 \sigma}{8 + 32C_K + m_2'},
\end{aligned}
\right.
\end{equation*}
where $m_2' = [4 + 6\gamma a + 4a\theta(2\gamma + \|\nabla W\|_{L^{\infty}})]^2 + a[6\gamma + \theta(2\gamma + \| \nabla W\|_{L^{\infty}})]$, then the constant $c_2$ can be taken as $c_2 = \frac{\delta a^2}{16(\rho_{wls} + 1)}$, which implies $c_2 \sim \sigma$, and the lower bound of diffusion constant $\sigma^{\ast}_2$ satisfies 
\begin{equation*}
\sigma_2^{\ast} \geq \max \bigg\{ \frac{800(40+m_2'')\rho_{wls}e}{a^2 \gamma}, 3200\rho_{wls}e\gamma \bigg\} \cdot \max\{C_K^2, C_K^3\},
\end{equation*}
where $m_2'' = [4 + 6\gamma + 4\theta(2\gamma + \|\nabla W\|_{L^{\infty}})]^2 + [6\gamma + \theta(2\gamma + \| \nabla W\|_{L^{\infty}})]$.
\end{remark}
\begin{remark}
Theorem \ref{Thm:theorem1.4} implies that second order particle system \eqref{Eq:particle system} not only exponentially converges to its equilibrium, but also converges to the unique mean field equilibrium as $N \rightarrow \infty$. If we take $t \rightarrow \infty$, the results \eqref{IEq:Theorem 1.4-1} and \eqref{IEq:Theorem 1.4-2} imply that, 
\begin{equation*}
H_N(f_{N,\infty}|f^{\otimes N}_{\infty}) \leq \mathcal{E}_N^{M_i}(f_{N,\infty}|f^{\otimes N}_{\infty}) \leq \frac{C}{N},\ \ \ i =1,2,
\end{equation*}
which offers us a kind of dynamical approach to prove the concentration of the Gibbs or stationary measure of the second order particle system around the nonlinear equilibrium of the limiting equation \eqref{Eq:limiting equation}. 
\end{remark}

Combining the exponential convergence from $f_t$ to $f_{\infty}$, we could replace $f_{\infty}$ by $f_t$ in last theorem so that avoid the estimates about $\nabla \log f_t$ and $\nabla^2 \log f_t$. Based on this observation, we establish the uniform-in-time propagation of chaos  both in the sense of the Wasserstein distance and  relative entropy.

\begin{theorem}\label{Thm:theorem1.5}
Suppose that $V$ and $W$ satisfy one of the following two cases,

(i) $V$ satisfies Assumption \ref{Assumption:V} and \ref{Assumption:V1}, $W$ satisfies Assumption \ref{Assumption:W1}. Moreover, we assume either $C_K$ is small or interaction functional $F$ satisfies Assumption \ref{Assumption:W2}.

(ii) $V$ satisfies Assumption \ref{Assumption:V} and \ref{Assumption:V2}, $W$ satisfies Assumption \ref{Assumption:W1} and $\|\nabla W \|_{L^{\infty}} < \infty$. Moreover, we assume either $C_K$ is small or interaction functional $F$ satisfies Assumption \ref{Assumption:W2}.

For the first case, for initial data $f_N^0$ of Eq.\eqref{Eq:Liouville equation} such that $\mathcal{E}_N^{M_1}(f^0_N|f^{\otimes N}_{\infty}) < \infty$, $f^0$ of Eq.\eqref{Eq:limiting equation} such that $\mathcal{E}^{M}(f_0|\hat{f}_0) < \infty$, and $\sigma \geq \sigma_1^{\ast}$, we have
\begin{equation}\label{IEq:Theorem 1.5-1}
\mathcal{W}_2^2(f_{N,k}, f^{\otimes k}) \leq C_1'k e^{-c'_1t} + C_1 \frac{k}{N},
\end{equation}
where $c'_1 = \min\{c,c_1\} > 0$ and $C'_1 = C'_1(f_N^0, f_0, f_{\infty}, \rho_{LS}) > 0$ are explicit and independent of $N$. 

For the second case, for initial data $f_N^0$ of Eq.\eqref{Eq:Liouville equation} such that $\mathcal{E}_N^{M_2}(f^0_N|f^{\otimes N}_{\infty}) < \infty$, $f^0$ of Eq.\eqref{Eq:limiting equation} such that $\mathcal{E}^{M_2'}(f_0|\hat{f}_0) < \infty$, and $\sigma \geq \sigma_2^{\ast}$, we have
\begin{equation}\label{IEq:Theorem 1.5-2}
\mathcal{W}_2^2(f_{N,k}, f^{\otimes k}) \leq C_2'ke^{-c'_2t} + C_2\frac{k}{N},
\end{equation}
where $c'_2 = \min\{c_2, c_2''\} > 0$ and  $C'_2 = C'_2(f_N^0, f_0, f_{\infty}) > 0$ are explicit and independent of $N$.
\end{theorem}
\begin{remark}
The constants $C_1'$ and $C_2'$ are taken as 
\begin{equation*}
\begin{aligned}
C_1' & = (1+\rho_{LS})[\mathcal{E}_N^{M_1}(f^0_N|f^{\otimes N}_{\infty}) + \mathcal{E}^M(f_0|\hat{f}_0)], \\
C_2' & = \mathcal{E}_N^{M_2}(f^0_N|f^{\otimes N}_{\infty}) + \mathcal{E}^{M'_2}(f_0|\hat{f}_0),
\end{aligned}
\end{equation*}
and $c_2'', M_2'$ can be found in Theorem \ref{Thm:theorem-1.2}.  
\end{remark}

\begin{remark}
The main difference of assumptions in Theorem \ref{Thm:theorem1.5} compared with  those in Theorem \ref{Thm:theorem1.4} is that $C_K$ is small or the  interaction functional $F$ is convex. Those  two  conditions make sure that $f_t$ exponentially converges to $f_{\infty}$. The small condition of $C_K$
comes from Theorem 10 in \cite{UPIandULSI}, which obtains the exponential convergence from $f_t$ to $f_{\infty}$ by uniform log Sobolev inequality of $f_{N,\infty}$. The convexity condition of $F$ is inspired by Theorem 2.1 in \cite{chen2023uniform}, this kind of condition avoid the smallness assumption on  $C_K$. We extend their result to more general confining potentials in the Appendix.
\end{remark}

\begin{remark}
The uniform-in-time propagation of chaos of second order particle system \eqref{Eq:particle system} has been investigated in  \cite{coupling2} and \cite{chen2023uniform}. Compared with \cite{coupling2} that exploits the coupling method, our result applies to more general confining potentials. In terms of Theorem 2.3 in \cite{chen2023uniform}, we do not need the uniform-in-$m$ log-Sobolev inequality of measure proportional to $e^{-\frac{\delta F}{\delta m}(m,x)}$, which is not very easy to verify. 
\end{remark}

\subsection{Related literature}

{\em Hypocoercivity.}  Hypocoercivity is an important analytical tool to study the long time behavior of Langevin dynamics and the corresponding kinetic Fokker-Planck equation. It was  initiated by Villani \cite{Hypocoercivity} and and then later advanced by Dolbeault, Mouhot and Schmeiser  in \cite{Dolbeault2008HypocoercivityFK} and \cite{Dolbeault2010HYPOCOERCIVITYFL}. However, those now well-known results are only restricted to the one particle dynamics without any interactions. For the $N$ particle system given by the Liouville equation \eqref{Eq:Liouville equation}, the natural stationary measure is simply the Gibbs measure given by the following form  
\begin{equation*}
f_{N,\infty} = \frac{1}{Z_{N,\beta}} e^{-\beta H(z_1,...,z_N)}.
\end{equation*}
A natural problem is that whether or not the convergence rate  from $f_N$ toward $f_{N,\infty}$ depends on the number of particles $N$. Many researchers have contributed to this problem. Guillin, etc study the uniform in $N$ functional inequalities in \cite{UPIandULSI}. Guillin and Monmarch\'e show uniform-in-$N$ exponential decay rate in \cite{momashe1} and \cite{momashe2} by ``Generalized $\Gamma$ calculus" developed in \cite{momashe3} and uniform log-Sobolev inequality in \cite{UPIandULSI}. Guillin, etc also use $H^1$ type norm to show the uniform-in-$N$ exponential decay rate by hypocoercivity and uniform Poincar\'{e} inequality in \cite{UPIandULSI}. These result are all restricted to potentials with smallness of $\| \nabla^2 U \|_{L^{\infty}}$. There are also some results that treat systems with singular potentials. Baudoin, Gordina and Herzog showed convergence to equilibrium by Gamma calculus in \cite{GammaCBV} with singular potentials. Lu and Mattingly constructed new Lyapunov function to show egodicity for systems \eqref{Eq:particle system} with Coulomb potential in the sense of weighted total variation distance in \cite{lu2019geometric}. However, the convergence rates, if they provides one,   all depend on $N$. 
\\ 
\\
{\em Propagation  of  chaos  for  kinetic Vlasov equation. }  The main result presented in this article is a further development of the relative entropy method introduced in \cite{JFAJW}, where Jabin and and the 2nd author proved a quantitative propagation of chaos for Newton's system with  bounded  interaction kernel in terms of relative entropy.  Lacker  \cite{Lacker1} then developed an approach based on the  BBGKY hierarchy and the entropy dissipation to optimize  the local convergence rate of $k$-marginals towards the limiting law.   Bresch, Jabin and Soler  \cite{BJJ} exploited the BBGKY hierarchy approach to firstly include the 2d Vlasov-Possion-Fokker-Planck case.  More recently, Bresch, Jabin and Duerinckx \cite{bresch2024duality} introduced a duality approach to cover the arbitrary square-integrable interaction forces at possibly vanishing temperature. Up to now, the  mean field limit or the propagation of chaos results are still very limited for second order particle system with singular interaction forces. See also for the results in \cite{hauray2007n,jabin2015particles,lazarovici2017mean,carrillo2018mean} and the review \cite{jabin2014review} for more detailed discussions. \\
For long time propagation of chaos, Monmarch\'e showed   uniform-in-time propagation of chaos in the sense of Wasserstein distance of one marginal for systems with convex potentials, i.e.
\begin{equation*}
\mathcal{W}_2(f_{N,1},f) \leq \frac{C}{N^{\alpha}},
\end{equation*}
for some contant $\alpha > 0, C > 0$ independent on $N$ and $t$. The sharp rate with $\alpha = \frac{1}{2}$ for the case $W(x) = c|x|^2$ has also been established there. Guillin and Monmarch\'e \cite{momashe2} later improved the convergence result to all marginals but without optimality in terms of of $\alpha$, i.e.
\begin{equation*}
\mathcal{W}_2(f_{N,k},f^{\otimes k}) \leq \frac{C \sqrt{k}}{N^{\alpha}},
\end{equation*}
for some contant $\alpha > 0, C > 0$ independent on $N$ and $t$. Thanks to the reflection coupling method, Guillin, Bris and Monmarch\'e  \cite{coupling2} proved the optimal convergence rate of $N$ for all marginals with convex or non-convex interaction potentials, i.e.
\begin{equation*}
\mathcal{W}_2(f_{N,k},f^{\otimes k}) \leq \frac{C \sqrt{k}}{\sqrt{N}},
\end{equation*}
for some constant $C > 0$ independent on $N$ and $t$, with the smallness assumption of the Lipschitz constant of interaction force $K$.  Recently, Chen, Lin, Ren and Wang \cite{chen2023uniform} showed  uniform-in-time propagation of chaos with functional convexity condition. Even though they do not  need smallness of $\|\nabla K\|_{L^\infty}$ , they require some  uniform-in-time Poincar\'e inequality for  the solution of the limiting PDE \eqref{Eq:limiting equation} . To the best of our knowledge, there is no result of uniform-in-time propagation of chaos for second order systems with singular interaction forces yet. We leave this topic for our further study.  \\
\\
$Equilibrium\ of\ Vlasov\mbox{-}Fokker \mbox{-}Planck\ equation.$ Uniform-in-time propagation of chaos cannot hold in general. One critical obstacle is that the non-linear Vlasov-Fokker-Planck equation \eqref{Eq:limiting equation} may have multiple equilibria and  hence exhibit phrase transition. The convergence from $f_t$ towards $f_{\infty}$ prevents the phrase transition or the presence of multiple equilibria of the limiting system. There are some results about this kind of convergence but with very limited conditions about potentials. See for instance \cite{PhLSI} and the reference therein.  Villani \cite{Hypocoercivity} proved that  $f_t$ converge to the Maxwellian
\begin{equation*}
\frac{1}{(2\pi)^d} e^{-\beta \frac{v^2}{2}}
\end{equation*}
on $\mathbb{T}^d$ with any polynomial order in the sense of $L^1$ norm, which requires that $W \in C^{\infty}$ and $\| W \|_{L^{\infty}}$ is small enough. Guillin and Monmarch\'e showed that  $f$ converges to $f_{\infty}$ in the sense of ``mean-field entropy" in \cite{momashe2}, which defines as 
\begin{equation}\label{Quantity:mean field entropy}
H_{W}(\nu) = E(\nu) - \inf_{\mu \in \mathcal{P}(\Omega \times \R^{d})} E(\mu)
\end{equation}
for probability measure $\nu$,  where
$E(\nu) = H(\nu|\alpha) + \frac{1}{2} F(\nu)$ and $\alpha \propto e^{-\frac{1}{2}v^2 -V(x)}$. Baudoin, Feng and Li \cite{FengQi1} established  that $f$ converges
to $f_{\infty}$ with exponential decay rate in the sense of free energy combining with ``relative Fisher Information" (by our notation)
\begin{equation}\label{Quantity:mean field RFi}
\mathcal{E}^M(f_t|\hat{f}_t) = \mathcal{F}(f_t) - \mathcal{F}(f_{\infty}) + I^M(f_t|\hat{f}_t),
\end{equation}
where $M$ is a constant matrix and $\hat{f} \propto e^{-\frac{v^2}{2} - V(x) - W \ast \rho_t}$. The free energy they used is  defined as 
\begin{equation}\label{Quantity:free energy}
\mathcal{F}(f) = \frac{1}{2} \int_{\Omega \times \R^d} v^2 f \ud x \ud v + \int_{\Omega \times \R^d} f \log f \ud x \ud v + \int_{\Omega \times \R^d} Vf \ud x \ud v + \frac{1}{2} F(f).
\end{equation}
They used $\Gamma-$calculus to overcome the dissipation degeneracy in $x$ direction with convexity and smallness of $\nabla^2 V$ and $\nabla^2 W$.  Chen, Lin, Ren and Wang also exploited the quantity \eqref{Quantity:mean field RFi} to prove $\mathcal{E}^M(f_t|\hat{f}_t)$ exponentially converges to 0 in \cite{chen2023uniform} under conditions $\|\nabla^2 (W \ast \rho_t + V)\|_{L^{\infty}} < \infty$ and the functional convexity of  $F$ (Assumption \ref{Assumption:W2}). These two groups both used  the so called free energy  to quantify the convergence from $f_t$ to $f_{\infty}$, i.e.
\begin{equation*}
H_W(f_t) = \mathcal{F}(f_t) - \mathcal{F}({f_{\infty}}).
\end{equation*}
Finally, let us recall the convergence result in \cite{UPIandULSI} and extend Theorem 2.1 in \cite{chen2023uniform} to more general confining potentials. By \cite{UPIandULSI}, we have
\begin{theorem}\label{Thm: theorem-1.1}
Suppose that  $V$ satisfies Assumption \ref{Assumption:V} and \ref{Assumption:V1}, $W$ satisfies Assumption \ref{Assumption:W1}. Suppose $f_{N,\infty}$ satisfies uniform log Sobolev inequality with constant $\rho_{LS}$. Then for the solution $f_t$ of Eq.(\ref{Eq:limiting equation}) with initial data $f_0$ such that $H_W(f_0) < \infty$ and $\int_{\Omega \times \R^d} z^2 f_0(z) \ud 
z < \infty$, we have
\begin{equation*}
H_W(f_t) \leq e^{-ct}H_{W}(f_0), 
\end{equation*} 
and 
\begin{equation}
\mathcal{W}_2^2(f_t, f_{\infty}) \leq \rho_{LS} H_W(f_t) \leq \rho_{LS}\mathcal{E}^M(f_0|\hat{f}_0) e^{-ct}.
\end{equation}
where $c>0$ is the same as Theorem \ref{Thm:theorem1.3}.
\end{theorem}
We extend Theorem 2.1 in \cite{chen2023uniform} to more general confining potentials,
\begin{theorem}\label{Thm:theorem-1.2}
Suppose that  $V$ satisfies Assumption \ref{Assumption:V} and \ref{Assumption:V2}, $W$ satisfies Assumption \ref{Assumption:W1} and Assumption \ref{Assumption:W2}. Then for solution $f_t$ of Eq.\eqref{Eq:limiting equation} with initial data $f_0$ such that $\mathcal{E}^{M_2'}(f_0|\hat{f}_0) < \infty$, we have
\begin{equation}
\mathcal{W}_2^2(f_t, f_{\infty}) \leq \mathcal{E}^M(f_t|\hat{f}_t) \leq e^{-c''_2t}\mathcal{E}^M(f_0|\hat{f}_0), 
\end{equation}
where $c''_2 = \frac{\delta a^2}{16 + 16\rho_{wls}}$ with some choice of $M_2'$, $\delta$ and $a$ (detailed in the Appendix).
\end{theorem}

We will give the sketch of proofs of these two theorems in the Appendix.

\subsection{Outline of the article}
The paper is then organized as follow: In Section \ref{section:preliminary}, we develop the basic tools we will use throughout this article. In Section 2.1, we introduce the normalized relative Fisher information and compute its  time evolution  under the kinetic dynamics \eqref{Eq:Liouville equation} and \eqref{Eq:limiting equation}.  In Section 2.2, we select the nontrivial matrix $M$ for relative Fisher information to deal with the confining potentials greater than a quadratic  function at infinity, where the crucial idea  ``entropy multipliers" is inspired by  the one particle case as   in \cite{EntropyMutiplier}. In Section 2.3, we introduce the weighted log-Sobolev inequality, which is essentially obtained  with the entropy multiplier method. In Section 2.4, we prove a new Law of Large Number estimates for systems with Lipschitz interaction force $K$. In Section 3, we give the complete proof of our main results, Theorem \ref{Thm:theorem1.3}, \ref{Thm:theorem1.4} and \ref{Thm:theorem1.5}. In the  Appendix, we prove the convergence from  $f$ to $f_{\infty}$ under some conditions on $V$ and $W$. 

\setcounter{equation}{0}

\section{Preliminary}\label{section:preliminary}
Let us define some linear operators in $(\Omega \times \R^d)^N$ we will use in this section. We denote $A = (0, \nabla_V)$ on $(\Omega \times \R^d)^N$ and $A_i = (0, \nabla_{v_i})$ on $\Omega \times \R^d$. The operator $B$ collects all of one order part of the Liouville operator $L_N$ in (\ref{Liouville operator}), i.e.
\begin{equation*}
B = \sum_{i=1}^N B_i, \ \ B_i = v_i \cdot \nabla_{x_i} - \nabla_{x_i}U \cdot \nabla_{v_i} - \gamma v_i \cdot \nabla_{v_i},
\end{equation*}
where $U$ is defined  in  (\ref{Eq:potential}). We write  the infinitesimal generator of $N$-particle system \eqref{Eq:particle system} as 
\begin{equation}\label{op:Dual Liouville}
L^{\ast}_N = B + \sigma \Delta_V.
\end{equation}

We recall the time evolution of the relative entropy as in \cite{JFAJW,InventionJW}. 
\begin{lemma}\label{lemma:error term of RE}
Assume that $f_N$ is a solution of Eq.\eqref{Eq:Liouville equation}. Assume further that $f(t,z) \in W^{1, \infty}(\R \times \Omega \times \R^d)$ solves Eq.\eqref{Eq:limiting equation} with $f(t,z) > 0$ and $\int_{\Omega \times \R^d} f(t,z)dz = 1$. Then
\begin{equation}\label{Eq:relative entropy}
\begin{aligned}
\frac{d}{dt} H_N(f_N|\bar{f_N}) = & - \frac{\sigma}{N} \int_{(\Omega \times \R^d)^N} f_N \left| \nabla_V \log \frac{f_N}{\bar{f_N}} \right|^2 \ud Z \\ & - \frac{1}{N} \int_{(\Omega \times \R^d)^N} \nabla_V \log \frac{f_N}{\bar{f_N}} \cdot R^0_N \ud Z,
\end{aligned}
\end{equation}
where $R^0_N$ is a $dN$-dimensional vector defined as 
\begin{equation*}
R^0_N = \bigg\{ \frac{1}{N} \sum_{j=1,j\neq i}^N K(x_i - x_j) - K \ast \rho_t(x_i) \bigg\}_{i=1}^N
\end{equation*}
if we take $\bar{f_N} = f^{\otimes N}(t,Z)$, and the second line of \eqref{Eq:relative entropy} vanishes if we take $\bar{f_N} = f_{N,\infty}$.
\end{lemma}

\begin{proof}
The computation is standard. We recommend \cite{Hypocoercivity} and \cite{Lacker1} for the detailed computation when $f_N = f_{N,\infty}$ and $f^{\otimes N}$. Using Young inequality when $f_N = f^{\otimes N}$, we also have
\begin{equation*}
\begin{aligned}
\frac{d}{dt} H_N(f_N|\bar{f_N}) \leq & - \frac{3 \sigma}{4}\frac{1}{N} \int_{(\Omega \times \R^d)^N} f_N \left| \nabla_V \log \frac{f_N}{\bar{f_N}} \right|^2 \ud Z \\ & + \frac{4}{\sigma} \frac{1}{N} \int_{(\Omega \times \R^d)^N} |R^0_N|^2  \ud Z.
\end{aligned}
\end{equation*}
\end{proof}

In the next, we turn to the argument about Fisher information.

\subsection{Hypocoercivity in entropy sense}

In this subsection, we extend hypocoercivity in entropy sense in \cite{Hypocoercivity} to $N$ particle system with nontrivial interaction force. We also use more general reference measure in $(\Omega \times \R^d)^N$ ----- invariant measure $f_{N,\infty}$ or $N$ times tensor product of limiting measure $f^{\otimes N}$, corresponding to uniform egodicity problem and uniform-in-time propagation of chaos problem.

In the following, we use notation $h = \frac{f_N}{\bar{f_N}}$ and $u = \log h$ for convenience, $\bar{f_N}$ may take $f_{N,\infty}$ or $f^{\otimes N}$.
Before tedious manipulations, we firstly derive the equation of $u$.

\begin{lemma}\label{lemma:Eq-u}
Assume that $f_N$ is a solution of Eq.\eqref{Eq:Liouville equation}, and assume that $f(t,z) \in W^{2, \infty}(\R \times \Omega \times \R^d)$ solves Eq.\eqref{Eq:limiting equation}, then
\begin{equation}\label{Eq:u}
\partial_t u = -Bu - \sigma A^{\ast}Au + \sigma \nabla_V \log f_N \cdot A u - \overline{R}_N,
\end{equation}
where $\overline{R}_N$ takes
\begin{equation*}    
\overline{R}_N = \sum_{i=1}^{N} \nabla_{v_i}\log f(x_i,v_i) \cdot \bigg \{ \frac{1}{N}\sum_{j=1, j\neq i}^N K(x_i-x_j) - K\star \rho(x_i) \bigg \},
\end{equation*}
if $\bar{f_N} = f^{\otimes N}$ and $\overline{R}_N = 0$ if $\bar{f_N} = f_{N,\infty}$.
\end{lemma}
\begin{proof}
The proof is direct computation. In terms of  Eq.\eqref{Eq:Liouville equation}, we have
\begin{equation}\label{logfn}
\partial_t \log f_N  =  -B \log f_N + \sigma \frac{\Delta_V f_N}{f_N} + \gamma Nd,
\end{equation}
and for Eq.\eqref{Eq:limiting equation}, we have
\begin{equation}\label{logbarfn}
\partial_t \log \bar{f_N} = -B \log \bar{f_N} + \sigma \frac{\Delta_V \bar{f_N}}{\bar{f_N}} + \gamma Nd + \overline{R}_N,
\end{equation}
we could understand $\overline{R}_N$ as the difference of drift part between particle system \eqref{Eq:particle system} and McKean-Vlasov system \eqref{Eq:MV equation}. Combine Eq.(\ref{logfn}) and Eq.(\ref{logbarfn}), we have
\begin{equation}
\partial_t u = -Bu   + \sigma \left\{ \frac{\Delta_V f_N}{f_N} - \frac{\Delta_V \bar{f_N}}{\bar{f_N}} \right\} - \overline{R}_N,
\end{equation}
using identity 
\begin{equation*}
\frac{\Delta_V F}{F} = \Delta_V \log F + |\nabla_V \log F|^2, 
\end{equation*}
therefore, $u$ satisfies the equation
\begin{equation}\label{u1}
\partial_t u = -Bu + \sigma \Delta_V u + \sigma \nabla_V \log(f_N \bar{f_N}) \cdot \nabla_V u - \overline{R}_N,
\end{equation}
recall $A = (0,\nabla_V)$, now we regard $\bar{f_N}$ as reference measure and use Proposition 3 of \cite{Hypocoercivity}, for vector function $g: (\Omega \times \R^d)^N \rightarrow \R^{Nd}$, we have
\begin{equation}\label{conjugation}
A^{\ast}g = - A \cdot g - \z \nabla_V \log \bar{f_N}, g \y,
\end{equation}
now we rewrite Eq.(\ref{u1}) as following
\begin{equation*}
\partial_t u = -Bu - \sigma A^{\ast}Au + \sigma \nabla_V \log f_N \cdot A u - \overline{R}_N,
\end{equation*}
we complete the proof.
\end{proof}
\begin{remark}
For $\bar{f_N} = f^{\otimes N}$, the conjugation relationship (\ref{conjugation}) is in a series of Hilbert space $L^2(\bar{f_N})$ depends on time $t$ whose associated measure satisfy Eq.\eqref{Eq:limiting equation}. If we take reference measure as equilibrium of Eq.\eqref{Eq:Liouville equation}, then the equation about $u$ is the same as (\ref{u1}) but $\overline{R}_N = 0$: 
\begin{equation}
\partial_t u = -Bu - \sigma A^{\ast}Au + \sigma \nabla_V \log f_N \cdot A u,
\end{equation}
this is why we start our computation from Eq.(\ref{Eq:u}).
\end{remark}

Now let us compute the time derivation of relative Fisher Information. We omit the integration domain $(\Omega \times \R^d)^N$ for convenience.

\begin{lemma}\label{lemma:Keylemma1}
Assume that $f_N$ is a solution of Eq.\eqref{Eq:Liouville equation}. Assume that $f(t,z) \in W^{2, \infty}(\R \times \Omega \times \R^d)$ solves Eq.\eqref{Eq:limiting equation} with $f(t,z) > 0$ and $\int_{\Omega \times \R^d} f(t,z)dz = 1$. Let $B, C, C'$ be linear differential operators on $(\Omega \times \R^d)^{N}$, where
\begin{equation*}
B = \sum_{i=1}^N B_i, \ \ B_i = v_i \cdot \nabla_{x_i} - \nabla_{x_i}U \cdot \nabla_{v_i} - \gamma v_i \cdot \nabla_{v_i},
\end{equation*}
and $C, C'$ are to be confirmed, then

\begin{equation}\label{Eq:derivation of Fi}
\begin{split}
\frac{d}{dt} \left\{ \int f_N \langle Cu, C'u\rangle \right\} = & \ \int f_N \langle [B, C]u, C'u \rangle + \int f_N \langle Cu, [B, C']u \rangle  \\ & - \int f_N \langle  C \overline{R}_N, C'u \rangle  - \int f_N \langle  Cu, C'\overline{R}_N \rangle \\ & - 2 \sigma \int f_N \z CAu, C'Au \y + \sigma Q_{C,A} + \sigma Q_{C',A},
\end{split}
\end{equation}
where
\begin{equation}\label{Eq:QCA}
\begin{split}
Q_{C,A} = & -\int f_N \z [C,A^{\ast}]Au, C'u\y 
- \int f_N \z [\nabla_V \log \bar{f_N} \cdot A, C]u, C'u \y \\ & -\int f_N \z [C,A]u, Au \otimes C'u \y - \int f_N \z [A,C]u, C'Au \y,
\end{split}
\end{equation}

\begin{equation}\label{Eq:QC'A}
\begin{split}
Q_{C',A} = & -\int f_N \z Cu, [C',A^{\ast}]Au\y - \int f_N \z Cu, [\nabla_V \log \bar{f_N} \cdot A, C']u \y \\ & - \int f_N \z Au \otimes Cu, [C',A]u\y - \int f_N \z CAu, [A, C']u \y.
\end{split}
\end{equation}
\end{lemma}
\begin{remark}\label{remark:notations}
$B$ is a $d$-tuple differential operator, but $C, C'$ are $2Nd$-tuple differential operators. We denote $C_i$ and $C_i'$ as $(C_{i_j})_{1 \leq i \leq N, 1 \leq j \leq d}$ and $(C'_{i_j})_{1 \leq i \leq N, 1 \leq j \leq d}$ in the sense of coordinate $z_j = (x_j,v_j)$. We omit the index $j$ for convenience in the following, i.e. $C = (C_i)_{1 \leq i \leq N}$ and $C' = (C'_i)_{1 \leq i \leq N}$. Each of them can be identified with a vector field $D_{i}$, in such a way that $C_i f = D_i \cdot \nabla f$, so $D$ can be seen as a map valued in $(2Nd \times 2Nd)$ matrix. The inner products above should be understood as $\z Cu, C'u\y = \sum_{i=1}^N \z C_i u, C_i'u \y_{\R^{2d}}$, $\z CAu, C'Au\y = \sum_{i,j=1}^N \z C_i A_ju, C_i'A_ju \y_{\R^{2d \times 2d}}$.
\end{remark}

\begin{remark}\label{remark:commutators}
Let us explain the commutators we use. $[B,C]$ is a $2Nd$-tuple operator, understood as $[B,C]_{i_j} = [B, C_{i_j}]$. But $[C,A]$ is a operator with $2Nd \times 2Nd$ components, understood as $[C,A]_{ij} = [C_i,A_j]$, and others follow. If $C, C'$ are commutative with $A$, the only nontrivial operators are $[C,A^{\ast}]$ and $[\nabla_V \log \bar{f_N} \cdot A, C]$, we will compute them later.  
\end{remark}

\begin{proof}
$Step1.$ In this step, we claim
\begin{equation}\label{Step1:result}
\begin{split}
\frac{d}{dt} \int f_N \z Cu, C'u \y = & \int f_N \langle [B, C]u, C'u \rangle + \int f_N \langle Cu, [B, C']u \rangle  \\ & -  \int f_N \langle  C \overline{R}_N, C'u \rangle  - \int f_N \langle  Cu, C'\overline{R}_N \rangle + \sigma Q_{\sigma},
\end{split}
\end{equation}
here $Q_{\sigma}$ collects all terms with coefficient $\sigma$ and reads as
\begin{equation}\label{Step1:diffusion}
\begin{split}
Q_{\sigma} = & \int f_N \Delta_V \z Cu, C'u \y \\ & - \int f_N \z CA^{\ast}A u, C'u\y + \int f_N \z C(\nabla_V \log f_N \cdot Au), C'u\y \\ & - \int f_N \z Cu, C'A^{\ast}A u\y + \int f_N \z Cu, C'(\nabla_V \log f_N \cdot Au)\y.
\end{split}
\end{equation}
There terms comes from diffusion part of Eq.\eqref{Eq:Liouville equation} and Eq.\eqref{Eq:limiting equation}, we will deal with them in next step.

We directly take derivative and split into three terms:
\begin{equation}\label{Step1:1}
\begin{split}
\frac{d}{dt} \int f_N \langle Cu, C'u\rangle  = & \int \partial_t f_N \langle Cu, C'u\rangle  \\ & + \int f_N \langle C(\partial_t u), C'u \rangle + \int f_N \langle Cu, C'(\partial_t u)\rangle.
\end{split}
\end{equation}
For the first term, we use Eq.\eqref{Eq:Liouville equation} and integral by parts,
\begin{equation}\label{step1:1-1}
\begin{split}
\int \partial_t f_N \langle Cu, C'u \rangle = &
\int -L_N f_N \langle Cu, C'u\rangle \\ = & \int f_N (B \langle Cu, C'u\rangle) + \sigma \int f_N (\Delta_V \langle Cu, C'u\rangle) \\ = & \int f_N \z BCu, C'u\y + \int f_N \z Cu, BC'u\y + \sigma \int f_N (\Delta_V \langle Cu, C'u\rangle).
\end{split}
\end{equation}
For second term,
\begin{equation}\label{step1:1-2}
\begin{split}
\int f_N \langle C(\partial_t u), C'u \rangle = & - \int f_N \langle C(B u), C'u \rangle - \int f_N \z CR_N, C'u \y \\ & - \sigma \int f_N \z CA^{\ast}A u, C'u\y + \sigma \int f_N \z C(\nabla_V \log f_N \cdot Au), C'u\y.
\end{split}
\end{equation}
Similarly, for third term,
\begin{equation}\label{step1:1-3}
\begin{split}
\int f_N \langle Cu, C'(\partial_t u) \rangle = & - \int f_N \langle Cu, C'(Bu) \rangle - \int f_N \z Cu, C'R_N \y \\ & - \sigma \int f_N \z Cu, C'A^{\ast}A u\y + \sigma \int f_N \z Cu, C'(\nabla_V \log f_N \cdot Au)\y.
\end{split}
\end{equation}
Combine \eqref{step1:1-1}, \eqref{step1:1-2}, \eqref{step1:1-3}, we complete the claim in this step.\\

$Step2.$ In this step, we deal with the diffusion part, we claim 

\begin{equation}
Q_{\sigma} =
- 2\int f_N \z CAu, C'Au \y + Q_{C,A} + Q_{C',A},
\end{equation}
where
\begin{equation}
\begin{split}
Q_{C,A} = & -\int f_N \z [C,A]u, Au \otimes C'u \y - \int f_N \z [A,C]u, C'Au \y  \\ & -\int f_N \z [C,A^{\ast}]Au, C'u\y 
- \int f_N \z [\nabla_V \log \bar{f_N} \cdot A, C]u, C'u \y,
\end{split}
\end{equation}

\begin{equation}
\begin{split}
Q_{C',A} = & - \int f_N \z Au \otimes Cu, [C',A]u\y - \int f_N \z CAu, [A, C']u \y \\ & -\int f_N \z Cu, [C',A^{\ast}]Au\y - \int f_N \z Cu, [\nabla_V \log \bar{f_N} \cdot A, C']u \y.
\end{split}
\end{equation}
The terms $Q_{C,A}$ and $Q_{C',A}$ collect all terms including commutators, we will take suitable operators $C$ and $C'$ to simplify these commutators.

For the first term of \eqref{Step1:diffusion}, we use integral by parts with respect to Lebesgue measure,   
\begin{equation}\label{step2:1-1}
\begin{split}
& \int \Delta_V f_N \z Cu, C'u \y \\ = & - \int f_N A u  \cdot A \z Cu, C'u \y \underline{ - \int f_N \nabla_V \log \bar{f_N} \cdot A \z Cu, C'u\y }.
\end{split}
\end{equation}
Recall we denote $A = (0,\nabla_V)$, we continue the first term of \eqref{step2:1-1}
\begin{equation}\label{Step2:1-2}
\begin{split}
& - \int f_N A u  \cdot A \z Cu, C'u \y \\ = & -\int f_N \z ACu, Au \otimes C'u \y - \int f_N \z Au \otimes Cu, AC'u\y \\ = & -\int f_N \z CAu, Au \otimes C'u \y - \int f_N \z Au \otimes Cu, C'Au\y \\ & -\int f_N \z [C,A]u, Au \otimes C'u \y - \int f_N \z Au \otimes Cu, [C',A]u\y.
\end{split}
\end{equation}
Let us explain the notation we use: $\z ACu, Au \otimes C'u \y$ should be understood as 
\[ \sum_{i,j=1}^{N}\left\z C_iA_ju, \z A_ju, C'_i u\y \right \y, \] 
and 
\[ \z [C,A]u, Au \otimes C'u \y = \sum_{i,j=1}^N \z [C_i,A_j]u, \z A_ju, C'_iu \y \y. \]
Moreover, we take the conjugate operator of $[C,A]$ w.r.t measure $\bar{f_N}$, we have
\begin{equation}\label{Step2:1-conjugation_rule}
\begin{split}
-\int f_N \z [C,A]u, Au \otimes C'u \y = &  - \int f_N  [C,A]^{\ast} \z Au \otimes C'u \y \\ = & - \sum_{i,j=1}^N \int f_N  [C_i,A_j]^{\ast} \z A_ju, C'_iu \y.
\end{split}
\end{equation}
For the second term of \eqref{Step1:diffusion},
we rewrite it as 
\begin{equation*}
\begin{split}
& - \int f_N \z CA^{\ast}A u, C'u\y \\ = & - \int f_N \z A^{\ast}CAu, C'u\y - \int f_N \z [C,A^{\ast}]Au, C'u\y,
\end{split}
\end{equation*}
and
\begin{equation*}
\begin{split}
- \int f_N \z A^{\ast}CAu, C'u\y = & - \int \bar{f_N} \z A^{\ast}CAu, C'h\y \\ = & - \int \bar{f_N} \z CAu, AC'h\y \\ = & - \int \bar{f_N} \z CAu, AhC'u\y \\ = & - \int f_N \z CAu, AC'u \y - \int f_N \z CAu, Au \otimes C'u\y,
\end{split}
\end{equation*}
then we have
\begin{equation}\label{Step2:1-3}
\begin{split}
- \int f_N \z CA^{\ast}A u, C'u\y = & - \int f_N \z CAu, C'Au \y - \int f_N \z CAu, [A, C']u \y \\ & -\int f_N \z [C,A^{\ast}]Au, C'u\y - \int f_N \z CAu, Au \otimes C'u\y.
\end{split}
\end{equation}
Similarly, for the fourth term of \eqref{Step1:diffusion}, up to the exchange of $C$ and $C'$, we have
\begin{equation}\label{Step2:1-4}
\begin{split}
- \int f_N \z Cu, C'A^{\ast}Au\y = & - \int f_N \z CAu, C'Au \y - \int f_N \z [A,C]u, C'Au \y \\ & -\int f_N \z Cu, [C',A^{\ast}]Au\y - \int f_N \z Au \otimes Cu, C'Au\y.
\end{split}
\end{equation}
Finally, all of left terms are third and fifth term in \eqref{Step1:diffusion} and underlined term in \eqref{step2:1-1}, we collect them as below, 
\begin{equation}\label{Step2:2-1}
\begin{split}
& - \int f_N \nabla_V \log \bar{f_N} \cdot A \z Cu, C'u\y \\ & + \int f_N \z C(\nabla_V \log f_N \cdot Au), C'u\y \\ & + \int f_N \z Cu, C'(\nabla_V \log f_N \cdot Au)\y.
\end{split}
\end{equation}
We deal with the first term of \eqref{Step2:2-1} as following,
\begin{equation}\label{Step2:2-2}
\begin{split}
& - \int f_N \nabla_V \log \bar{f_N} \cdot A \z Cu, C'u\y \\ = & - \int f_N \z \nabla_V \log \bar{f_N} \cdot A (Cu), C'u\y - \int f_N \z Cu, \nabla_V \log \bar{f_N} \cdot A (C'u)\y \\ = & - \int f_N \z C(\nabla_V \log \bar{f_N} \cdot Au), C'u\y - \int f_N \z [\nabla_V \log \bar{f_N} \cdot A, C]u, C'u \y \\ & - \int f_N \z Cu, C'(\nabla_V \log \bar{f_N} \cdot A u)\y - \int f_N \z Cu, [\nabla_V \log \bar{f_N} \cdot A, C']u \y,
\end{split}
\end{equation}
combine with the last two term of \eqref{Step2:2-1}, we have
\begin{equation}\label{Step2:2-3}
\begin{split}
& - \int f_N \nabla_V \log \bar{f_N} \cdot A \z Cu, C'u\y \\ & + \int f_N \z C(\nabla_V \log f_N \cdot Au), C'u\y \\ & + \int f_N \z Cu, C'( \nabla_V \log f_N \cdot Au)\y \\ = & \int f_N \z C|Au|^2, C'u\y + \int f_N \z Cu, C'|Au|^2\y \\  & - \int f_N \z [\nabla_V \log \bar{f_N} \cdot A, C]u, C'u \y - \int f_N \z Cu, [\nabla_V \log \bar{f_N} \cdot A, C']u \y
\\ = & 2\int f_N \z CAu, Au \otimes C'u\y + 2\int f_N \z Au \otimes Cu, C'Au\y \\  & - \int f_N \z [\nabla_V \log \bar{f_N} \cdot A, C]u, C'u \y - \int f_N \z Cu, [\nabla_V \log \bar{f_N} \cdot A, C']u \y.
\end{split}
\end{equation}
After gathering \eqref{Step2:1-2}, \eqref{Step2:1-3}, \eqref{Step2:1-4} and \eqref{Step2:2-3}, we find
\begin{equation}\label{Step2:2-4}
\begin{split}
& - 2\int f_N \z CAu, C'Au \y 
\\ &-\int f_N \z [C,A]u, Au \otimes C'u \y - \int f_N \z Au \otimes Cu, [C',A]u\y 
\\ & - \int f_N \z [A,C]u, C'Au \y - \int f_N \z CAu, [A, C']u \y \\ & -\int f_N \z [C,A^{\ast}]Au, C'u\y -\int f_N \z Cu, [C',A^{\ast}]Au\y 
 \\ & - \int f_N \z [\nabla_V \log \bar{f_N} \cdot A, C]u, C'u \y - \int f_N \z Cu, [\nabla_V \log \bar{f_N} \cdot A, C']u \y.
\end{split}
\end{equation}
Based on the conjugate rule \eqref{Step2:1-conjugation_rule} for all similar terms, we complete the proof.
\end{proof}

\begin{corollary}\label{Cor:QCA}
Let $\bar{f}_N = f^{\otimes N}$, if $A$ is commutative with $C$ and $C'$, we have
\begin{equation}
Q_{C,A} = 2 \sum_{i=1}^N \int f_N (C_i \nabla_{v_i} \log f)\z A_iu, C_iu\y_{\R^{2d}},
\end{equation}
and 
\begin{equation}
Q_{C',A} = 2 \sum_{i=1}^N \int f_N (C'_i \nabla_{v_i} \log f)\z A_iu, C'_iu\y_{\R^{2d}}.
\end{equation}
\end{corollary}

\begin{proof}
We recall that
\begin{equation}\label{Cor:QCA-term}
Q_{C,A} = -\int f_N \z [C,A^{\ast}]Au, C'u\y 
- \int f_N \z [\nabla_V \log \bar{f_N} \cdot A, C]u, C'u \y.
\end{equation}
For the second term of \eqref{Cor:QCA-term}, we direct compute the commutator $[\nabla_V \log \bar{f_N} \cdot A, C]$, i.e.
\begin{equation*}
\begin{split}
& - \int f_N \z [\nabla_V \log \bar{f_N} \cdot A, C]u, C'u \y \\ = & - \sum_{i,j=1}^N \int f_N \z [\nabla_{v_j} \log \bar{f_N} \cdot A_j, C_i]u, C'_i u \y \\ = & \sum_{i,j=1}^N \int f_N \z C_i \nabla_{v_j} \log \bar{f_N}, A_j u \otimes C_i'u\y \\ = &  \sum_{i=1}^N \int f_N (C_i \nabla_{v_i} \log f)\z A_iu, C_iu\y_{\R^{2d}},
\end{split}
\end{equation*}
by $C_i \nabla_{v_j} \log \bar{f_N} = 0$ if $i \neq j$. For the first term of \eqref{Cor:QCA-term}, we understand the commutator $[C,A^{\ast}]$ as the row of operators 
\begin{equation}
[C,A^{\ast}] = ([C_1,A^{\ast}],...,[C_N,A^{\ast}] ),
\end{equation}
then we have
\begin{equation}
\begin{split}
& -\int f_N \z [C,A^{\ast}]Au, C'u\y \\ = & - \sum_{i=1}^N \int f_N \z [C_i,A^{\ast}]Au, C_i'u \y \\ = & - \sum_{i=1}^N \int f_N \z C_i(A^{\ast}Au)- A^{\ast}(C_iAu), C'_iu \y,
\end{split}
\end{equation}
here we understand $C_iAu$ as
a $N$-tuple vector reads as $(C_iA_1u, ... ,C_iA_Nu)$, which can be operated by $A^{\ast}$. Recall 
\begin{equation*}
A^{\ast}g = - A \cdot g - \z \nabla_V \log \bar{f_N}, g \y,
\end{equation*}
then we have
\begin{equation}
\begin{split}
& -\int f_N \z [C,A^{\ast}]Au, C'u\y \\ = &
\sum_{i,j=1}^N \sum_{k,l=1}^d \int f_N \left \{ \z C_{i_l}A_{j_k}A_{j_k}u + C_{i_l}\{(\nabla_{v_{j_k}} \log \bar{f_N}) (A_{j_k}u)\}, C'_{i_l}u\y \right\} \\ & - \sum_{i,j=1}^N \sum_{k,l=1}^d \int f_N \left\{ \z A_{j_k}C_{i_l}A_{j_k}u + \{(\nabla_{v_{j_k}} \log \bar{f_N}) (C_{i_l}A_{j_k}u)\}, C'_{i_l}u \y \right\}
\\ = & \sum_{i=1}^N \int f_N (C_i \nabla_{v_i} \log f)\z A_iu, C'_iu\y_{\R^{2d}}.
\end{split}
\end{equation}
Up to change the position of $C$ and $C'$, we complete the proof.
\end{proof}

\subsection{Entropy multipliers}

In order to deal with more general potentials $V$ and $W$, we develop the method of entropy multipliers for $N$-particle relative Fisher information. We recommend Part 1, section 8 in \cite{Hypocoercivity} and \cite{EntropyMutiplier} for ``one particle version'' without interaction potentials, and they only consider the invariant measure of single particle Fokker Planck equation as reference measure. Now let us consider a $2Nd \times 2Nd$ weight matrix $M(t,Z)$ to distort the relative Fisher Information, i.e.
\begin{equation*}
\frac{1}{N}\int f_N\z Cu, M_t C'u\y dZ,
\end{equation*}
we evolve this quantity along time $t$ in the following lemma, many ideas of manipulation are similar with Lemma \ref{lemma:Keylemma1}.

\begin{lemma}\label{lemma:Keylemma2}
Assume that $f_N$ is a solution of Eq.\eqref{Eq:Liouville equation}. Assume that $f(t,z) \in W^{1, \infty}(\R \times \Omega \times \R^d)$ solves Eq.\eqref{Eq:limiting equation} with $f(t,z) > 0$ and $\int_{\Omega \times \R^d} f(t,z)dz = 1$. Let $B, C, C'$ be differential operators on $(\Omega \times \R^d)^{N}$, where
\begin{equation*}
B = \sum_{i=1}^N B_i, \ \ B_i = v_i \cdot \nabla_{x_i} - \nabla_{x_i}U \cdot \nabla_{v_i} - \gamma v_i \cdot \nabla_{v_i},
\end{equation*}
and $C, C'$ are to be confirmed. Let $M_t(Z): \R \times (\Omega \times \R^d)^N \rightarrow \R^{2Nd \times 2Nd}$ be a matrix valued function smooth enough for all variables, then
\begin{equation}\label{keylemma2}
\begin{split}
\frac{d}{dt} \left\{ \int f_N \langle Cu, M_tC'u\rangle \right\} = & \ \int f_N \langle M_t[B, C]u, C'u \rangle + \int f_N \langle Cu, M_t[B, C']u \rangle  \\ & - \int f_N \langle  C\overline{R}_N, M_tC'u \rangle  - \int f_N \langle  M_tCu, C'\overline{R}_N \rangle \\ & - 2 \sigma \int f_N \z CAu, M_tC'Au \y + \sigma Q_{C,A} + \sigma Q_{C',A} \\ & + \int f_N \z Cu, (\partial_t M_t + BM_t + \sigma \Delta_V M_t)C'u \y,
\end{split}
\end{equation}
where
\begin{equation}\label{Eq:QCA2}
\begin{split}
Q_{C,A} =  & -\int f_N \z [C,A^{\ast}]Au, M_tC'u\y 
- \int f_N \z [\nabla_V \log \bar{f_N} \cdot A, C]u, M_tC'u \y \\ & -\int f_N \z [C,A]u, Au \otimes M_tC'u \y - \int f_N \z M_t[A,C]u, C'Au \y \\ & + \int f_N \z [A,C]u, (AM_t)C'u \y,
\end{split}
\end{equation}

\begin{equation}\label{Eq:QC'A2}
\begin{split}
Q_{C',A} = & -\int f_N \z M_tCu, [C',A^{\ast}]Au\y - \int f_N \z M_tCu, [\nabla_V \log \bar{f_N} \cdot A, C']u \y \\ & - \int f_N \z Au \otimes M_tCu, [C',A]u\y - \int f_N \z CAu, M_t[A, C']u \y\\ & + \int f_N \z (AM_t)Cu, [A,C']u \y.
\end{split}
\end{equation}
\end{lemma}
\begin{remark}
The notations of inner product and commutator operators appeared above are the same as in Remark \ref{remark:notations} and \ref{remark:commutators}. Let us explain the notations associated with $M_t$ we used above. We denote $Cu,[B,C]u, C\overline{R}_N$ are $Nd$-tuple vectors, then it is reasonable to multiply $M_t$ with them. $\z CAu, M_tC'Au \y$ should be understood as
\begin{equation*}
\sum_{i,j,l=1}^N \z C_i A_l u, M_{i j}C'_{j} A_{l}u\y_{\R^{2d}}
\end{equation*}
in the sense of coordinate $z_j \in \Omega \times \R^d$. The derivative of $M_t$ (i.e. $\partial_t M_t, AM_t, BM_t, \Delta_V M_t$) is taken componentwise. 
\end{remark}

\begin{proof}
We use the similar argument with Lemma \ref{lemma:Keylemma1}. We directly take derivative and obtain,  
\begin{equation}\label{MStep1:1}
\begin{split}
\frac{d}{dt} \int f_N \langle Cu, M_tC'u\rangle  = & \int \partial_t f_N \langle Cu, M_tC'u\rangle  \\ & + \int f_N \langle C(\partial_t u), M_tC'u \rangle + \int f_N \langle Cu, M_tC'(\partial_t u)\rangle \\ & + \int f_N \z Cu, (\partial_t M_t)C'u\y.
\end{split}
\end{equation}
Observe that the last term
\begin{equation}
\int f_N \z Cu, (\partial_t M_t)C'u\y
\end{equation}
is new. Next we use the similar step in Lemma \ref{lemma:Keylemma1} to go on. \\

$Step1.$ In this step, we claim the following
\begin{equation}\label{Mstep1:result}
\begin{split}
\frac{d}{dt} \int f_N \z Cu, M_tC'u \y = & \int f_N \langle [B, C]u, M_tC'u \rangle + \int f_N \langle Cu, M_t[B, C']u \rangle  \\ & -  \int f_N \langle  CR_N, M_tC'u \rangle  - \int f_N \langle  Cu, M_tC'R_N \rangle + \sigma Q_{\sigma} \\ & + \int f_N
\z Cu, (\partial_t M_t + BM_t)C'u\y,
\end{split}
\end{equation}
where $Q_{\sigma}$ collects all terms with coefficient $\sigma$, which reads as
\begin{equation}\label{MStep1:diffusion}
\begin{split}
Q_{\sigma} = & \int f_N \Delta_V \z Cu, M_tC'u \y \\ & - \int f_N \z CA^{\ast}A u, M_tC'u\y + \int f_N \z C(\nabla_V \log f_N \cdot Au), M_tC'u\y \\ & - \int f_N \z Cu, M_tC'A^{\ast}A u\y + \int f_N \z Cu, M_tC'(\nabla_V \log f_N \cdot Au)\y.
\end{split}
\end{equation}

\begin{remark}
Let us explain some notations we use. Since $[B,C']u$, $C'(A^{\ast}Au)$ and $C'(\nabla_V \log f_N \cdot Au)$ are all $2Nd$-tuple vectors, it is reasonable to multiple them with $M_t$. After that, $M_t[B,C']u$, $M_tC'(A^{\ast}Au)$ and $M_t C'(\nabla_V \log f_N \cdot Au)$ are $2Nd$-tuple vectors.    
\end{remark}

For the first term of \eqref{MStep1:1}, recall Liouville operator \eqref{Liouville operator} we have 
\begin{equation}\label{MStep1:1-1}
\begin{split}
\int \partial_t f_N \langle Cu, M_tC'u \rangle = &
\int -L_N f_N \langle Cu, M_tC'u\rangle \\ = & \int f_N (B \langle Cu, M_tC'u\rangle) + \sigma \int f_N (\Delta_V \langle Cu, M_tC'u\rangle) \\ = & \int f_N \z BCu, M_tC'u\y + \int f_N \z Cu, M_tBC'u\y + \sigma \int f_N (\Delta_V \langle Cu, M_tC'u\rangle) \\ & + \int f_N \z Cu,(BM_t)C'u\y.
\end{split}
\end{equation}
For the second term and third term of \eqref{MStep1:1}, using Eq.(\ref{Eq:u}), we have
\begin{equation}\label{MStep1:1-2}
\begin{split}
\int f_N \langle C(\partial_t u), M_tC'u \rangle = & - \int f_N \langle C(B u), M_tC'u \rangle - \int f_N \z CR_N, M_tC'u \y \\ & - \sigma \int f_N \z CA^{\ast}A u, M_tC'u\y + \sigma \int f_N \z C(\nabla_V \log f_N \cdot Au), M_tC'u\y.
\end{split}
\end{equation}

\begin{equation}\label{MStep1:1-3}
\begin{split}
\int f_N \langle Cu, M_tC'(\partial_t u) \rangle = & - \int f_N \langle Cu, M_tC'(Bu) \rangle - \int f_N \z Cu, M_tC'R_N \y \\ & - \sigma \int f_N \z Cu, M_tC'A^{\ast}A u\y + \sigma \int f_N \z Cu, M_tC'(\nabla_V \log f_N \cdot Au)\y,
\end{split}
\end{equation}
Combine terms \eqref{MStep1:1-1},\eqref{MStep1:1-2} and \eqref{MStep1:1-3}, we obtain the \eqref{Mstep1:result}. \\

$Step2.$ In this step, we focus on the term $Q_{\sigma}$ as before, whose terms all come from diffusion. Now we claim,
\begin{equation}
\begin{split}
Q_{\sigma} = & 
- 2\int f_N \z CAu, M_tC'Au \y + Q_{C,A} + Q_{C',A} \\ & + \int f_N \z Cu, (\Delta_V M_t)C'u \y,
\end{split}
\end{equation}
where
\begin{equation}
\begin{split}
Q_{C,A} =  & -\int f_N \z [C,A^{\ast}]Au, M_tC'u\y 
- \int f_N \z [\nabla_V \log \bar{f_N} \cdot A, C]u, M_tC'u \y \\ & -\int f_N \z [C,A]u, Au \otimes M_tC'u \y - \int f_N \z M_t[A,C]u, C'Au \y \\ & + \int f_N \z [A,C]u, (AM_t)C'u \y,
\end{split}
\end{equation}

\begin{equation}
\begin{split}
Q_{C',A} = & -\int f_N \z M_tCu, [C',A^{\ast}]Au\y - \int f_N \z M_tCu, [\nabla_V \log \bar{f_N} \cdot A, C']u \y \\ & - \int f_N \z Au \otimes M_tCu, [C',A]u\y - \int f_N \z CAu, M_t[A, C']u \y\\ & + \int f_N \z (AM_t)Cu, [A,C']u \y,
\end{split}
\end{equation}
The terms $Q_{C,A}$ and $Q_{C',A}$ collect all terms including commutators, we will take suitable operators $C$ and $C'$ to simplify these commutators.

For the first term of \eqref{MStep1:diffusion}, using integral by parts with respect to Lebesgue measure,   
\begin{equation}\label{Mstep2:2-1}
\begin{split}
& \int \Delta_V f_N \z Cu, M_tC'u \y \\ = & - \int f_N A u  \cdot A \z Cu, M_tC'u \y  \underline{ - \int f_N \nabla_V \log \bar{f_N} \cdot A \z Cu, M_tC'u\y }.
\end{split}
\end{equation}
Recall we denote $A = (0, \nabla_V)$, we continue the first term of \eqref{Mstep2:2-1} 
\begin{equation}
\begin{split}
& - \int f_N A u  \cdot A \z Cu, M_tC'u \y \\ = & -\int f_N \z ACu, Au \otimes M_tC'u \y - \int f_N \z Au \otimes M_tCu, AC'u\y \\ & - \int f_N Au \cdot \z Cu, (AM_t)C'u\y \\ = & -\int f_N \z CAu, Au \otimes M_tC'u \y - \int f_N \z Au \otimes M_tCu, C'Au\y \\ & + \int f_N \z [C,A]u, Au \otimes M_tC'u \y + \int f_N \z Au \otimes M_tCu, [C',A]u\y
\\ & - \int f_N Au \cdot \z Cu, (AM_t)C'u\y.
\end{split}
\end{equation}
For the second terms of \eqref{Mstep2:2-1}, we  rewrite it as
\begin{equation}\label{Mstep2:2-2}
\begin{split}
& - \int f_N \z CA^{\ast}A u, M_tC'u\y \\ = & - \int f_N \z A^{\ast}CAu,M_tC'u\y - \int f_N \z [C,A^{\ast}]Au, M_tC'u\y,
\end{split}
\end{equation}
and we continue the first term of \eqref{Mstep2:2-2}, 
\begin{equation*}
\begin{split}
& - \int f_N \z A^{\ast}CAu, M_tC'u\y \\ = & - \int \bar{f_N} \z A^{\ast}CAu, M_tC'h\y \\ = & - \int \bar{f_N} \z CAu, M_t(AC'h)\y - \int \bar{f_N} \z CAu, (AM_t)C'h \y \\ = & - \int f_N \z CAu, M_t(AC'u) \y - \int \bar{f_N} \z CAu, Ah \otimes M_tC'u\y - \int \bar{f_N} \z CAu, (AM_t)C'h \y \\ = & - \int f_N \z CAu, M_t(AC'u) \y - \int f_N \z CAu, Au \otimes M_tC'u\y - \int f_N \z CAu, (AM_t)C'u \y,
\end{split}
\end{equation*}
now we have
\begin{equation*}
\begin{split}
- \int f_N \z CA^{\ast}A u, M_tC'u\y = & - \int f_N \z CAu, M_t(C'Au) \y - \int f_N \z CAu, M_t[A, C']u \y \\ & -\int f_N \z [C,A^{\ast}]Au, M_tC'u\y - \int f_N \z CAu, Au \otimes M_tC'u\y \\ & - \int f_N \z CAu, (AM_t)C'u \y.
\end{split}
\end{equation*}
Similarly, for the fourth term of \eqref{MStep1:diffusion}, since $M_t$ is symmetric, using its conjugation in $\R^{dN}$, 
\begin{equation*}
- \int f_N \z Cu, M_tC'A^{\ast}Au\y = - \int f_N \z M_tCu, C'A^{\ast}Au\y,
\end{equation*}
up to the exchange of $C$ and $C'$, we have
\begin{equation}\label{Mstep24}
\begin{split}
- \int f_N \z Cu, M_tC'A^{\ast}Au\y = & - \int f_N \z M_t(CAu), C'Au \y - \int f_N \z M_t[A,C]u, C'Au \y \\ & -\int f_N \z M_tCu, [C',A^{\ast}]Au\y - \int f_N \z Au \otimes M_tCu, C'Au\y \\ & - \int f_N \z (AM_t)Cu, C'Au \y.
\end{split}
\end{equation}
Finally, all of left terms are third and fifth term in \eqref{MStep1:diffusion}, and underlined term in \eqref{Mstep2:2-1}, we collect them as below,
\begin{equation}\label{Mstepleft}
\begin{split}
& - \int f_N \nabla_V \log \bar{f_N} \cdot A \z Cu, M_tC'u\y \\ & + \int f_N \z C(\nabla_V \log f_N \cdot Au), M_tC'u\y \\ & + \int f_N \z M_tCu, C'(\nabla_V \log f_N \cdot Au)\y.
\end{split}
\end{equation}
We deal with the first term of (\ref{Mstepleft}) as following,
\begin{equation}\label{Mstep26}
\begin{split}
& - \int f_N \nabla_V \log \bar{f_N} \cdot A \z Cu, M_tC'u\y \\ = & - \int f_N \z \nabla_V \log \bar{f_N} \cdot A (Cu), M_tC'u\y - \int f_N \z M_tCu, \nabla_V \log \bar{f_N} \cdot A (C'u)\y \\ & - \int f_N \z Cu, (\nabla_V \log f_N \cdot AM_t)C'u \y \\ = & - \int f_N \z C(\nabla_V \log \bar{f_N} \cdot Au), M_tC'u\y - \int f_N \z [\nabla_V \log \bar{f_N} \cdot A, C]u, M_tC'u \y \\ & - \int f_N \z M_tCu, C'(\nabla_V \log \bar{f_N} \cdot A u)\y - \int f_N \z M_tCu, [\nabla_V \log \bar{f_N} \cdot A, C']u \y \\ & - \int f_N \z Cu, (\nabla_V \log f_N \cdot AM_t)C'u \y,
\end{split}
\end{equation}
by similar argument with \eqref{Step2:2-2} and \eqref{Step2:2-3} in Lemma \ref{lemma:Keylemma1}, we have 
\begin{equation}
\begin{split}
& 2\int f_N \z CAu, Au \otimes C'u\y + 2\int f_N \z Au \otimes Cu, C'Au\y \\  & - \int f_N \z [\nabla_V \log \bar{f_N} \cdot A, C]u, C'u \y - \int f_N \z Cu, [\nabla_V \log \bar{f_N} \cdot A, C']u \y \\ & - \int f_N \z Cu, (\nabla_V \log f_N \cdot AM_t)C'u \y.
\end{split}
\end{equation}
Until now, let us collect all terms don't appear in the Lemma \ref{lemma:Keylemma1}:
\begin{equation}\label{only first term}
\begin{split}
& - \int f_N Au \cdot \z Cu, (AM_t)C'u\y \\ & - \int f_N \z CAu, (AM_t)C'u \y - \int f_N \z (AM_t)Cu, C'Au \y \\ & -\int f_N \z Cu, (\nabla_V \log \bar{f_N} \cdot A M_t + \partial_t M_t + BM_t)C'u\y.
\end{split}
\end{equation}
We compute the first term more precisely,
\begin{equation}
-\int f_N Au \cdot \z Cu, (AM_t)C'u\y \\ = - \int \bar{f_N} Ah \cdot \z Cu, (AM_t)C'u\y, 
\end{equation}
using the conjugate relationship w.r.t. measure $\bar{f}_N$, we have
\begin{equation}
- \int \bar{f_N} Ah \cdot \z Cu, (AM_t)C'u\y = - \int f_N A^{\ast} \z Cu, (AM_t)C'u\y,
\end{equation}
recall $A^{\ast} = - A \cdot g - \z \nabla_V \log \bar{f_N}, g \y$, then we have
\begin{equation}
\begin{split}
& - \int f_N A^{\ast} \z Cu, (AM_t)C'u\y \\ = & \sum_{i=1}^N \int f_N A_i \cdot \z Cu, (A_iM_t)C'u \y + \int f_N \nabla_V \log \bar{f_N} \cdot \z Cu, (AM_t)C'u\y \\ = & \sum_{i=1}^{N} \int f_N \z A_iCu, (A_iM_t)C'u \y + \sum_{i=1}^N \int f_N \z (A_iM_t)Cu, A_iC'u \y \\ & + \int f_N \z Cu, (\Delta_V M) C'u\y + \int f_N \z Cu, (\nabla_V \log \bar{f_N} \cdot A M_t)C'u\y \\ = &  \int f_N \z ACu, (AM_t)C'u \y + \int f_N \z (AM_t)Cu, AC'u \y \\ & + \int f_N \z Cu, (\Delta_V M) C'u\y + \int f_N \z Cu, (\nabla_V \log \bar{f_N} \cdot A M_t)C'u\y.
\end{split}
\end{equation}
Combine with last three terms of (\ref{only first term}), all of remaining terms compared with \eqref{Step2:2-4} in the proof of Lemma \ref{lemma:Keylemma1} are  
\begin{equation}
\begin{split}
& \int f_N \z Cu, (\Delta_V M_t)C'u \y \\ & + \int f_N \z [A,C]u, (AM_t)C'u \y + \int f_N \z (AM_t)Cu, [A,C']u \y.
\end{split}
\end{equation}
Together with \eqref{Mstep1:result} in $Step1$, we finish the proof.
\end{proof}

\begin{corollary}\label{Cor:block matrix}
If we take $M$ as a block positive defined matrix, i.e.
\begin{equation}\label{matrix:M}
M = 
\left(
\begin{array}{cc}
E  & F \\
F  & G
\end{array}
\right),
\end{equation}
where $E,F,G: \R \times (\Omega \times \R^d)^N \rightarrow \R^{Nd \times Nd}$ are matrix-valued functions smooth enough, and take $C = C' = (\nabla_X, \nabla_V)$, then we have
\begin{equation}\label{IEq:evlov-Fisher}
\begin{split}
\frac{d}{dt} \left\{\int f_N \z Cu, M_t Cu\y \right\} = & - \int \z Cu, S_tCu \y - 2 \int f_N \z C\overline{R}_N, M_t Cu\y \\ & - 2 \sigma \int f_N \z CAu, M_tC'Au \y,
\end{split}
\end{equation}
where $S_t = S(t,Z)$ is a matrix reads as
\begin{equation}
S_t = 
\left(
\begin{array}{cc}
2F - \partial_t E - L^{\ast}_NE \ \ & D_1 - \partial_t F - L^{\ast}_NF - 2E \nabla^2 U + 2 \gamma F\\
2G - \partial_t F - L^{\ast}_NF \ \ & D_2 - \partial_t G - L^{\ast}_N G - 2F \nabla^2 U + 2 \gamma G
\end{array}
\right)
\end{equation}
and 
\[ D_1 = -4 \sigma (F \nabla_V \nabla_V \log \bar{f_N} + E \nabla_X \nabla_V \log \bar{f_N}), \]
\[ D_2 = -4 \sigma (F \nabla_X \nabla_V \log \bar{f_N} + G \nabla_V
\nabla_V \log \bar{f_N}). \]
\end{corollary}
\begin{remark}
A easy computation shows that  
\begin{equation}\label{Quantity:U}
\nabla^2 U =
\begin{pmatrix}
\nabla^2 V(x_1) + \frac{1}{N} \sum_{j \neq 1}^N \nabla^2 W(x_1 - x_j) & \cdots & -\frac{1}{N} \nabla^2 W(x_1 - x_N) \\ \vdots & \ddots & \vdots \\ -\frac{1}{N} \nabla^2 W(x_N - x_1) & \cdots & \nabla^2 V(x_N) + \frac{1}{N} \sum_{j \neq N}^N \nabla^2 W(x_N - x_j)
\end{pmatrix},
\end{equation}
which is a $Nd \times Nd$ matrix.
\end{remark}
\begin{proof}
The main idea is to analyse every term in (\ref{keylemma2}). Let us compute what commutator $[B,C]$ is firstly. For each component of operator $C = (\nabla_{X},\nabla_{V})$, we have
\begin{equation}
[\nabla_{X}, B]_i = \sum_{k=1}^{N} [\nabla_{x_i}, B_k] = - \sum_{k=1}^N [\nabla_{x_i},  \nabla_{x_k} U \cdot \nabla_{v_k}] = - \sum_{k=1}^N \nabla^2_{x_i,x_k} U \cdot \nabla_{v_k},
\end{equation}
and
\begin{equation}
\begin{split}
[\nabla_{V}, B]_i = \sum_{k=1}^{N} [\nabla_{v_i}, B_k] & = \sum_{k=1}^{N} [\nabla_{v_i}, v_k \cdot \nabla_{x_k} - \gamma v_k \cdot \nabla_{v_k}] \\ & = \sum_{k=1}^{N} \delta_{ik} (\nabla_{x_k} - \gamma \nabla_{v_k}) = \nabla_{x_i} - \gamma \nabla_{v_i}.
\end{split}
\end{equation}
In other word, we have
\begin{equation}
\begin{split}
[B,C]^{\top} = - (-\nabla^2 U \cdot \nabla_V, \nabla_X - \gamma \nabla_V)^{\top} =
\left(
\begin{array}{cc}
0 & \nabla^2 U \\
-Id & \gamma Id
\end{array}
\right)
\left(
\begin{array}{c}
\nabla_X \\
\nabla_V
\end{array}
\right).
\end{split}
\end{equation}
Then we regard $\z M_t[B, C]u, Cu \y + \z Cu, M_t[B, C]u \y$ as a quadratic form with matrix
\begin{equation}\label{Mat:commutatorBC}
\left(
\begin{array}{cc}
-2F & E \nabla^2 U - \gamma F - G \\
E \nabla^2 U - \gamma F - G & 2F \nabla^2 U - 2 \gamma G
\end{array}
\right).
\end{equation}
Since $M_t$ is a positive defined matrix, $-\z CAu, M_tCAu \y$ must be negative, so we just keep it. Moreover, $C, C'$ are commutative with $A$, only nontrivial terms in (\ref{Eq:QCA2}) and (\ref{Eq:QC'A2}) are first line two terms. We recall $[C,A^{\ast}]$ and $[\nabla_V \log \bar{f_N} \cdot A, C]$ in Corollary \ref{Cor:QCA}, and we regard $Q_{C,A} + Q_{C',A}$ as a quadratic form with following matrix,
\begin{equation}\label{Mat:C-Astar}
\left(
\begin{array}{cc}
E & F \\
F & G
\end{array}
\right)
\left(
\begin{array}{cc}
0 & 4 \sigma \nabla_X \nabla_V \log \bar{f_N} \\
0 & 4 \sigma \nabla_V \nabla_V \log \bar{f_N}
\end{array}
\right)
=
\left(
\begin{array}{cc}
0 & 4 \sigma (E \nabla_X \nabla_V \log \bar{f_N} + F \nabla_V \nabla_V \log \bar{f_N})\\
0 & 4 \sigma (F \nabla_X \nabla_V \log \bar{f_N} + G \nabla_V \nabla_V \log \bar{f_N})
\end{array}
\right).
\end{equation}
For the last line in (\ref{keylemma2}), we recall duality of Liouville operator (\ref{op:Dual Liouville}), then we have
\begin{equation}\label{Mat:dualLiouville}
\partial_t M_t + BM_t + \sigma M_t = 
\left(
\begin{array}{cc}
\partial_t E + L_N^{\ast}E & \partial_t F + L_N^{\ast}F \\
\partial_t F + L_N^{\ast}F & \partial_t G + L_N^{\ast}G
\end{array}
\right).
\end{equation}
Combining (\ref{Mat:commutatorBC}),(\ref{Mat:C-Astar}) and (\ref{Mat:dualLiouville}), we finish the proof.
\end{proof}

\subsection{Weighted log Sobolev inequality}\label{section:WLSI}

In this section, we establish the weighted Log-Sobolev inequality for nonlinear equilibrium $f_{\infty}$ defined by Eq.\eqref{Eq:nonlinear equilibrium}, then we extend to the weighted $N$-particle version by verifying tensorized invariance of weighted Log-Sobolev inequality. Before that, let us talk about the situation of first order particle system. Guillin et al. consider the uniform-in-time propagation of chaos with the following limiting equation on $\mathbb{T}^d$ in \cite{Guillin2021UniformIT},
\begin{equation}\label{Eq:first order limiting equation}
\partial_t \rho + \text{div}_x(\rho K \ast \rho ) = \sigma \Delta_x \rho, \ \ \ x \in \mathbb{T}^d,
\end{equation}
where $\rho(t, x): \R \times \Omega \rightarrow \R$ is a probability density. An very useful observation in \cite{Guillin2021UniformIT} is that, if there exists some constant $\lambda > 1$ such that
\begin{equation}\label{IEq:upper-lower-rho(0)}
\frac{1}{\lambda} \leq \rho_0 \leq \lambda,
\end{equation}
for initial data $\rho_0$ of Eq.(\ref{Eq:first order limiting equation}) on $\mathbb{T}^d$, then they propagate this property to all time $t$ uniformly, i.e.
\begin{equation}\label{IEq:upper-lower-rho(t)}
\frac{1}{\lambda} \leq \rho(t) \leq \lambda.
\end{equation}
After controlling the upper bound of  $\|\nabla^n \rho \|_{L^{\infty}}$ by standard energy estimates, they obtain the upper bound of $\| \nabla \log \rho \|_{L^{\infty}}$ and $\| \nabla^2 \log \rho \|_{L^{\infty}}$, which is essential for the proof of Theorem 1 in \cite{Guillin2021UniformIT}. Another important observation in \cite{Guillin2021UniformIT} is that $\rho_t(x)$ satisfies Log-Sobolev inequality uniformly in $t$ as a result of perturbation of uniform distribution by (\ref{IEq:upper-lower-rho(t)}) (See Proposition 5.1.6, \cite{BGL}). These two facts help them obtain uniform-in-time propagation of chaos even for Biot-Savart kernel. But in the case of Vlasov-Fokker-Planck equation \eqref{Eq:limiting equation}, the situation becomes totally different. The best we can expect to initial data of Eq.\eqref{Eq:limiting equation} is 
\begin{equation}\label{IEq:lower-bound-VFP}
f_0(x,t) \geq Ce^{-\frac{av^2}{2}},
\ \ \ (x,v) \in \Omega \times \R^d
\end{equation}
for some constant $C > 0, a > 0$. The lack of positive lower bound of \eqref{IEq:lower-bound-VFP} makes the uniform-in-time upper bound of $\| \nabla^2 \log f \|_{L^{\infty}}$ fail by the same strategy in \cite{Guillin2021UniformIT}, that is the reason why we replace the reference measure $f$ by $f_{\infty}$. 

Inspired by the argument of Gibbs measure for one particle in \cite{EntropyMutiplier},
\begin{equation}\label{Eq:oneparticle-invmeasure}
d\mu = \frac{1}{Z}e^{- \beta H(x,v)} \ud x \ud v, \ \ \ (x,v) \in \Omega \times \R^d,
\end{equation}
where $H(x,v) = \frac{v^2}{2} + V(x)$ is the Hamiltonian defined on $\Omega \times \R^d$ and $Z$ is partition function. We omit the temperature constant $\beta$ in the following and recall some related results in \cite{EntropyMutiplier}.

\begin{definition}\label{Assumption:mu-LSI}
$\mu$ satisfies the following weighted Log-Sobolev inequality in $\R^d \times \R^d$ if there exists some constant $\rho_{wls}(\mu) > 0$ s.t. for all smooth enough $g$ with $\int g^2 d\mu = 1$:
\begin{equation}\label{IEq:W-log-sob}
Ent_{\mu}(g^2) \leq \rho_{wls}(\mu) \int (H^{-2\eta}|\nabla_x g|^2 + |\nabla_v g|^2) \ud \mu.
\end{equation}
\end{definition}
The weighted Log-Sobolev inequality \eqref{IEq:W-log-sob} associates with a new second order operator $L_{\eta}$ on $\R^d \times \R^d$, 
\begin{equation}\label{op:W-L}
L_{\eta} := H^{-2\eta} \Delta_x + \Delta_v - H^{-2\eta} \left( 2 \eta \frac{\nabla_x H}{H} + \nabla_x H \right) \cdot \nabla_x - \nabla_v H \cdot \nabla_v,
\end{equation}
which is symmetric in $L^2(\mu)$ and satisfies
\begin{equation*}
\int f (L_{\eta}g)d\mu = - \int (H^{-2\eta} \nabla_x f \cdot \nabla_x g + \nabla_v f \cdot \nabla_v g) \ud \mu.
\end{equation*}
The following theorem tells us how to verifying the weighted Log-Sobolev inequality for suitable condition of function $H$.
\begin{theorem}
Assume that $H$ goes to infinity at infinity and that there exists $a > 0$ such that $e^{aH} \in L^{1}(\mu)$. 

(1) If $\mu$ satisfies the weighted Log-Sobolev inequality (\ref{IEq:W-log-sob}), then, there exists a Lyapunov function, i.e. a smooth function $W$ such that $W(x,y) \geq w > 0$ for all $(x,v)$, two positive constant $\lambda$ and $b$ such that
\begin{equation}\label{IEq:lyapounov}
L_{\eta}W \leq - \lambda H W + b.
\end{equation}

(2) Conversely, assume that there exists a Lyapunov function satisfying (\ref{IEq:lyapounov}) and that $|\nabla H| \geq c > 0$ for $|(x,y)|$ large enough. Define
\begin{equation}
\theta(r) = \sup_{z \in \partial A_r} \max_{i,j=1,...,2d}|\frac{\partial^2 H}{\partial z_i \partial z_j}|
\end{equation}
and assume that $\theta(r) \leq ce^{C_0r}$ with some positive constants $C_0$ and $c$ for $r$ sufficiently large. Then $\mu$ satisfies the weighted Log-Sobolev inequality (\ref{IEq:W-log-sob}).
\end{theorem}

A natural choice of Lyapunov function is $W(x,v) = e^{\alpha(\frac{v^2}{2} + V(x))}$ with $\alpha \in (0,1)$, we refer \cite{2016HITTING} and \cite{EntropyMutiplier} to readers for more details. A simple perturbation argument in \cite{BGL} could extend the weighted Log-Sobolev inequality to $f_{\infty}$.
\begin{proposition}\label{Pro:perturbation}
Assume that $\mu$ satisfies the weighted Log-Sobolev inequality (\ref{IEq:W-log-sob}), let $\mu_1$ be a probability measure with density $h$ with respect to $\mu$ such that $\frac{1}{\lambda} \leq h \leq \lambda$ for some constant $\lambda > 0$, then $\mu_1$ satisfies
\begin{equation}\label{IEq:Pertubation}
Ent_{\mu_1}(g^2) \leq \rho_{wls}(\mu) \lambda^2 \int (H^{-2\eta}|\nabla_x g|^2 + |\nabla_v g|^2) \ud \mu_1.
\end{equation}
\end{proposition}

\begin{proof}
We use the following lemma in \cite{BGL} (Page240, Lemma 5.1.7),
\begin{lemma}
Let $\phi: I \rightarrow \R$ on some open interval $I \subset \R$ be convex of class $C^2$. For every bounded or suitably integrable measurable function $f: E \rightarrow \R$ with values in $I$,
\begin{equation}
\int_E \phi(f)d\mu - \phi\left( \int_Ef d\mu \right) = \inf_{r \in I} \int_E [\phi(f) - \phi(r) - \phi'(r)(f - r)] \ud \mu. 
\end{equation}
\end{lemma}
For the function $\phi(r) = rlogr$ on $I = (0, \infty)$, we can easily establish (\ref{IEq:Pertubation}) by
\begin{equation*}
Ent_{\mu_1}(g^2) = \inf_{r \in I} \int_E [\phi(g^2) - \phi(r) - \phi'(r)(g^2 - r)] d\mu_1 \leq \lambda Ent_{\mu}(g^2), 
\end{equation*}
and 
\begin{equation*}
\begin{split}
\lambda Ent_{\mu}(g^2) & \leq \rho_{wls}(\mu) \lambda \int (H^{-2\eta}|\nabla_x g|^2 + |\nabla_v g|^2) \ud \mu \\ & \leq \rho_{wls}(\mu) \lambda^2\int (H^{-2\eta}|\nabla_x g|^2 + |\nabla_v g|^2) \ud \mu_1.
\end{split}
\end{equation*}
\end{proof}

Now we verify the invariance of weighted Log-Sobolev constant, by which we extend weighted Log-Sobolev inequality \eqref{IEq:W-log-sob} to the measure of tensorized form $\mu^{\otimes N}$.   

\begin{proposition}\label{Pro:weight tensoriztion}
Assume that $\mu_1, \mu_2$ satisfy the weighted Log-Sobolev inequality (\ref{IEq:W-log-sob}) with constant $\rho_{wls}(\mu_1)$ and $\rho_{wls}(\mu_2)$, then $\mu_1 \otimes \mu_2$ satisfies 
\begin{equation}\label{IEq:N-W-log-sob}
Ent_{\mu_1 \otimes \mu_2}(f^2) \leq \rho_{wls} \sum_{i=1}^2 \int_{E_1 \times E_2} (H^{-2\eta}(z_i)|\nabla_{x_i} f|^2 + |\nabla_{v_i} f|^2) \ud \mu_1 \otimes \mu_2
\end{equation}
with $\rho_{wls} = \max\{\rho_{wls}(\mu_1), \rho_{wls}(\mu_1)\}$.
\end{proposition}
\begin{proof}
We denote
\begin{equation*}
g^2(z_1) = \int_{E_2} f^2(z_1,z_2) \ud \mu_2(z_2),
\end{equation*}
and of course $g \geq 0$, then
\begin{equation*}
\begin{split}
Ent_{\mu_1 \otimes \mu_2}(f^2) = &  Ent_{\mu_1}(g^2) \\ + & \int_{E_1} \left( \int_{E_2} f^2(z_1,z_2) \log f^2(z_1,z_2) \ud \mu_2(z_2) \right. \\ & \left. - \int_{E_2} f^2(z_1,z_2) \ud \mu_2(z_2) \log \int_{E_2}f^2(z_1,z_2) \ud \mu_2(z_2) \right) \ud \mu_1(z_1).
\end{split}
\end{equation*}
Observe that
\begin{equation*}
\begin{split}
Eut_{\mu_2}(f^2) = & \int_{E_2} f^2(z_1,z_2) \log f^2(z_1,z_2) \ud \mu_2(z_2) \\ & - \int_{E_2} f^2(z_1,z_2) \ud \mu_2(z_2) \log \int_{E_2}f^2(z_1,z_2) \ud \mu_2(z_2) \\ \leq & \rho_{wls}(\mu_2) \int_{E_2} (H^{-2\eta}(z_2
)|\nabla_{x_2}f|^2 + |\nabla_{v_2}f|^2) \ud \mu_2(z_2).
\end{split}
\end{equation*}
Next we estimate $Ent_{\mu_1}(g^2)$, 
\begin{equation*}
\begin{split}
Ent_{\mu_1}(g^2) \leq  \rho_{wls}(\mu_1) \int_{E_2} (H^{-2\eta}(z_1)|\nabla_{x_1}g|^2 + |\nabla_{v_1}g|^2) \ud \mu_1(z_1),
\end{split}
\end{equation*}
for terms on the right hand side, observe the weight $H(z_1)$ only concerns the first variable,  
\begin{equation*}
\begin{split}
|H^{-\eta}(z_1)\nabla_{x_1}g|^2  & = \left| H^{-\eta}(z_1) \nabla_{x_1} \sqrt{\int_{E_2} f^2(z_1,z_2) \ud \mu(z_2)} \right|^2 \\ & \leq  \frac{\left(\int_{E_2} f (H^{-\eta}(z_1) |\nabla_{x_1}f|) \ud \mu(z_2)\right)^2}{\int_{E_2} f^2(z_1,z_2) \ud \mu(z_2)},
\end{split}
\end{equation*}
and use Cauchy-Schwarz inequality, we have
\begin{equation*}
|H^{-\eta}(z_1)\nabla_{x_1}g|^2 \leq \int_{E_2} H^{-2\eta}(z_1)|\nabla_{x_1}f|^2 \ud \mu_2(z_2).
\end{equation*}
Similarly,
\begin{equation*}
|\nabla_{v_1}g|^2 \leq \int_{E_2} |\nabla_{v_1}f|^2 \ud \mu_2(z_2),
\end{equation*}
then we finish the proof.
\end{proof}

\begin{corollary}\label{Cor:LSI-f_infty}
Assume that $V$ satisfies $(i)$ of Assumption \ref{Assumption:V1}, and $W$ satisfies Assumption \ref{Assumption:W1}, then $f_{\infty}$ satisfies weighted Log-Sobolev inequality \eqref{IEq:W-log-sob}. Moreover, $f_{\infty}^{\otimes N}$ satisfies weighted Log-Sobolev inequality \eqref{IEq:W-log-sob} with the same constant. 
\end{corollary}
\begin{proof}
According to Assumption \ref{Assumption:W1}, $|\nabla W(x)| \leq \frac{\lambda}{2}|x|$, then we have
\begin{equation*}
|W(x)| \leq |W(0)| + \frac{\lambda}{2}|x|^2,    
\end{equation*}
hence
\begin{equation*}
\begin{aligned}
|W \ast \rho_{\infty}(x)| & \leq |W(0)| + \int |y|^2 f_{\infty}(y,v) \ud y \ud 
v + \frac{\lambda}{2}|x|^2 \leq C + \frac{\lambda}{2}|x|^2,
\end{aligned}
\end{equation*}
for some constant $C > 0$ with $\mathbb{E}_{\rho_{\infty}}|x|^2 < \infty$. By $(i)$ of Assumption \ref{Assumption:V1} for $V$, we have
\begin{equation}
a \mu \leq f_{\infty} \leq b \mu, 
\end{equation}
for some $a, b > 0$ and $\mu = \frac{1}{Z} e^{-V(x) - \frac{1}{2}v^2}$. Using Proposition \ref{Pro:perturbation}, we obtain the weighted Log-Sobolev inequality for $f_{\infty}$, then Proposition \ref{Pro:weight tensoriztion} makes sure the same fact about $f_{\infty}^{\otimes N}$. 
\end{proof}

\subsection{Large deviation estimates}\label{section:Large deviation estimates}
In this subsection, we deal with the error terms we mentioned before. For relative entropy, we recall Lemma \ref{lemma:error term of RE} and the error term reads as,
\begin{equation}\label{Quantity:error term of RE}
\frac{4}{\sigma} \sum_{i=1}^N \int_{(\Omega \times \R^d)^N} f^t_N \bigg| \frac{1}{N}\sum_{j,j\neq i}K(x_i-x_j) - K\star \rho_{\infty}(x_i) \bigg|^2 \ud Z.
\end{equation}
For relative fisher information, we recall Corollary $\ref{Cor:block matrix}$ and we estimate the following term,
\begin{equation}\label{Quantity:error term of RFi}
\int_{(\Omega \times \R^d)^N} f^t_N \z C\overline{R}_N, M_t Cu\y \ud Z.
\end{equation}
Our main tool comes from Jabin-Wang's large derivation type estimates for propagation of chaos for singular kernels in series of paper \cite{JFAJW}, \cite{InventionJW}. We refer Theorem 3 in \cite{JFAJW} or Theorem 3 and Theorem 4 in \cite{InventionJW} to readers for more details to this kind of technique, now we apply it to our cases --- bounded or lipschtiz kernels.

The following lemma gives more precise analysis to error term \eqref{Quantity:error term of RFi}. Here we take $C = (\nabla_X, \nabla_V)$.

\begin{lemma}\label{Lemma:CS-error}
Let $M: \Omega \times \R^d \rightarrow \R^{2Nd \times 2Nd}$ be a positive defined matrix function which takes as (\ref{matrix:M}). Assume that $E, F, G$ are block diagonal positive defined matrices, i.e.
$E = \text{diag} \{E^{1},...,E^{N}\}$, $F = \text{diag} \{F^{1},...,F^{N}\}$ and $G = \text{diag} \{G^{1},...,G^{N}\}$, where $E^i, F^i$ and $G^i$ are $d \times d$ positive defined matrices. Then we have
\begin{equation}
\begin{split}
\int f_N \z \nabla \overline{R}_N, M \nabla u\y \leq & \sum_{i=1}^N \frac{4}{\sigma} \int f_N |\sqrt{E^{i}} R_{N,i}^1|^2 + \frac{4}{\sigma}\int f_N |(F^{i})^{\frac{3}{4}}R_{N,i}^1|^2 \\ & 
\sum_{i=1}^N \frac{\gamma^2}{2\sigma^2} \int f_N |\sqrt{E^{i}}R^3_{N,i}|^2 + \frac{\gamma^2}{2\sigma} \int f_N |\sqrt{F^{i}} R^3_{N,i}|^2 \\ &
\sum_{i=1}^N \frac{\gamma^2}{2\sigma^2} \int f_N |\sqrt{F^{i}} R^2_{N,i}|^2 + \frac{4 \gamma^2}{\sigma^3} \int f_N |G^{i} R^2_{N,i}|^2
\\ & + C_K \int f_N |\sqrt{E} \nabla_V u|^2 + (C_K + \frac{1}{2}) \int f_N |\sqrt{F} \nabla_V u|^2 + \frac{\sigma}{4} \int f_N |\nabla_V u|^2 \\ & + (C_K + \frac{1}{2}) \int f_N |\sqrt{E} \nabla_X u|^2 + \frac{1}{4} \int f_N |\sqrt{F} \nabla_X u|^2 \\ & + \sum_{i=1}^N \frac{\sigma}{16} \int f_N | \sqrt{E^{i}} \nabla_{v_i} \nabla_{x_i} u|_{\R^{2d}}^2 + \frac{\sigma}{16} \int f_N |(F^{i})^{\frac{1}{4}} \nabla_{v_i} \nabla_{v_i} u|_{\R^{2d}}^2 \\ & - \sum_{i=1}^N \int f_N \langle R^1_{N,i}, (\nabla_{v_i} E^{i}) \nabla_{x_i} u \rangle_{\R^{2d}} - \int f_N \langle R^1_{N,i}, (\nabla_{v_i} F^{i}) \nabla_{v_i} u \rangle_{\R^{2d}}. 
\end{split}
\end{equation}
where $\sigma$ is diffusion constant, $C_K = \|\nabla K\|_{L^{\infty}}$ and 
\begin{equation}
R^1_{N,i} = \frac{1}{N}\sum_{j\neq i} \nabla K(x_i-x_j) - \nabla K\star \rho_{\infty}(x_i), 
\end{equation}

\begin{equation}
\ \ \ R^3_{N,i} = \frac{1}{N} \sum_{j, j\neq i} \nabla K(x_j - x_i) \cdot [v \ast \rho^v_{\infty}(v_i) - (v_i - v_j)], \ \ \rho^v_{\infty}(v) = \int f_{\infty} \ud x,
\end{equation}

\begin{equation}
R^2_{N,i} = - \frac{1}{N}\sum_{j\neq i}K(x_i-x_j) - K\star \rho_{\infty}(x_i).
\end{equation}
\end{lemma}
\begin{remark}
Here we understand $\nabla K$ or $R^1_{NX,i}$ as a $d \times d$ matrix, $\nabla K \cdot v$ and $K$ are $d$ dimensional vectors, we omit the summation about these components for convenience. The meaning of this lemma is to divide the error term of relative Fisher Information into $L^2$ norm of $R_N^i, i=1,2,3$ and small terms we can absorb by negative part in \eqref{IEq:evlov-Fisher}.
\end{remark}
\begin{proof}
We directly compute $\nabla \overline{R}_N$ and use the Young inequality. For each component $i$ in position direction, we have
\begin{equation}\label{Eq: error term-FI-X}
\begin{split}
\nabla_{x_i} \overline{R}_N & = \nabla_{v_i}\log f(x_i, v_i) \cdot \bigg \{\frac{1}{N} \sum_{j=1,j\neq i}^N \nabla K(x_i - x_j) - \nabla K\star \rho(x_i) \bigg\}  \\ &  \ \ \ + \nabla_{x_i} \nabla_{v_i} \log f(x_i,v_i) \cdot \bigg\{ \frac{1}{N} \sum_{j=1,j\neq i}^N K(x_i-x_j) - K\star \rho(x_i) \bigg\} \\ & \ \ \ - \frac{1}{N} \sum_{j=1,j\neq i}^N \nabla K(x_j - x_i) \cdot \nabla_{v_j} \log f(x_j, v_j),
\end{split}
\end{equation}
we denote these three lines above by
\begin{equation*}
\nabla_{x_i} \overline{R}_N = \nabla_{v_i} \log f(x_i,v_i) \cdot R^1_{N,i} + \nabla_{x_i} \nabla_{v_i} \log f(x_i,v_i) \cdot R^2_{N,i} + \frac{\gamma}{\sigma} R^3_{N,i},   
\end{equation*}
where
\begin{equation*}
\begin{split}
R^1_{N,i} = \bigg\{\frac{1}{N}\sum_{j=1,j\neq i}^N \nabla K(x_i-x_j) - \nabla K\star \rho_{\infty}(x_i) \bigg\}
\end{split}
\end{equation*}
is a $d \times d$ matrix,
\begin{equation*}
\begin{split}
R_{N,i}^2 = \bigg\{ \frac{1}{N}\sum_{j=1,j\neq i}^N K(x_i-x_j) - K\star \rho_{\infty}(x_i) \bigg\}
\end{split}
\end{equation*}
is a $d$ dimensional vector and 
\begin{equation*}
\begin{split}
R^3_{N,i} & =  - \frac{\sigma}{\gamma} \frac{1}{N} \sum_{j=1,j\neq i}^N \nabla K(x_j - x_i) \cdot \nabla_{v_j} \log f(x_j, v_j)
\end{split}
\end{equation*}
is a $d$ dimensional vector. Recall that we take $M$ as a block matrix 
\begin{equation*}
M = 
\left(
\begin{array}{cc}
E  & F \\
F  & G
\end{array}
\right),
\end{equation*}
and $E, F, G$ are $d \times d$ diagonal matrices. Hence, we use integral by parts for each component of the first term of (\ref{Eq: error term-FI-X}), we have
\begin{equation}\label{IEq:error term XX1}
\begin{aligned}
& \sum_{i=1}^N \int f_N \langle \nabla_{v_i} \log f \cdot R_{N,i}^1, E^i \nabla_{x_i} u \rangle \\ = & - \sum_{i=1}^N \int f_N \langle \nabla_{v_i} u \cdot R_{N,i}^1, E^i \nabla_{x_i} u \rangle + \sum_{i=1}^N \int f_N \langle \nabla_{v_i} \log f_N \cdot R_{N,i}^1, E^i \nabla_{x_i} u \rangle \\ = & - \sum_{i=1}^N \int f_N \langle \nabla_{v_i} u \cdot R_{N,i}^1, E^i \nabla_{x_i} u \rangle + \sum_{i=1}^N \int \langle \nabla_{v_i} f_N \cdot R^1_{N,i}, E^i \nabla_{x_i} u \rangle \\ = & - \sum_{i=1}^N \int f_N \langle \nabla_{v_i} u \cdot R_{N,i}^1, E^i \nabla_{x_i} u \rangle - \sum_{i=1}^N \int f_N \langle R^1_{N,i}, E^i \nabla_{v_i} \nabla_{x_i} u \rangle_{\R^{2d}} \\ & - \sum_{i=1}^N \int f_N \langle R^1_{N,i}, (\nabla_{v_i} E^i) \nabla_{x_i} u \rangle_{\R^{2d}}.
\end{aligned}
\end{equation}
Using Young inequality, we have, 
\begin{equation}
\begin{aligned}
& \sum_{i=1}^N \int f_N \langle \nabla_{v_i} \log f \cdot R_{N,i}^1, E^i \nabla_{x_i} u \rangle \\ \leq & C_K \int f_N |\sqrt{E} \nabla_V u|^2 + C_K \int f_N |\sqrt{E} \nabla_X u|^2 \\ & + \sum_{i=1}^N \frac{\sigma}{16} \int f_N | \sqrt{E^i} \nabla_{v_i} \nabla_{x_i} u|_{\R^{2d}}^2 + \sum_{i=1}^N \frac{4}{\sigma} \int f_N |\sqrt{E^i} R_{N,i}^1|_{\R^{2d}}^2 \\ & - \sum_{i=1}^N \int f_N \langle R^1_{N,i}, (\nabla_{v_i} E^i) \nabla_{x_i} u \rangle_{\R^{2d}},
\end{aligned}
\end{equation}
similarly,
\begin{equation}\label{IEq:error term XV1}
\begin{split}
& \sum_{i=1}^N \int f_N \langle \nabla_{v_i} \log f \cdot R_{N,i}^1, F^i \nabla_{v_i} u \rangle  \\ \leq & C_K \int f_N |\sqrt{F} \nabla_V u|^2 \\ & + \sum_{i=1}^N \frac{\sigma}{16} \int f_N | (F^i)^{\frac{1}{4}} \nabla_{v_i} \nabla_{v_i} u|_{\R^{2d}}^2 + \sum_{i=1}^N \frac{4}{\sigma} \int f_N |(F^i)^{\frac{3}{4}} R_{N,i}^1|_{\R^{2d}}^2 \\ & - \sum_{i=1}^N \int f_N \langle R^1_{N,i}, (\nabla_{v_i} F^i) \nabla_{v_i} u \rangle_{\R^{2d}}.
\end{split}
\end{equation}
Actually the computation above is the same if we take $f = f_{\infty}$, then we finish the estimates. For the third term of (\ref{Eq: error term-FI-X}), by Young inequality, we have
\begin{equation}\label{IEq:error term XX3}
\begin{split}
\sum_{i=1}^N \int f_N \langle R^3_{N,i}, E^i \nabla_{x_i} u \rangle \leq \sum_{i=1}^N \frac{\gamma^2}{2\sigma^2} \int f_N |\sqrt{E^i}R^3_{N,i}|^2 + \frac{1}{2} \int f_N |\sqrt{E} \nabla_X u|^2,
\end{split}
\end{equation}
\begin{equation}\label{IEq:error term XV3}
\begin{split}
\sum_{i=1}^N \int f_N \langle R^3_{N,i}, F^i \nabla_{v_i} u \rangle \leq \sum_{i=1}^N \frac{\gamma^2}{2\sigma^2} \int f_N |\sqrt{F^i} R^3_{N,i}|^2 + \frac{1}{2} \int f_N |\sqrt{F} \nabla_V u|^2,
\end{split}
\end{equation}
moreover, we rewrite $R^3_{N,i}$ as 
\begin{equation*}
\begin{aligned}
R^3_{N,i} = \frac{1}{N} \sum_{j=1, j \neq i}^N \nabla K(x_j - x_i) \cdot [v \ast \rho_{\infty}^{v} - (v_i - v_j)], \ \ \ \rho_{\infty}^v(v) = \int f_{\infty}dx,
\end{aligned}
\end{equation*}
by $\nabla_{v_i} \log f_{\infty} = -\frac{\gamma}{\sigma} v_i$. In the end, the second term  $\nabla_{x_i} \nabla_{v_i} \log f \cdot R^2_{N,i}$ disappears by $\nabla_{x_i} \nabla_{v_i} \log f_{\infty} = 0$ if we take $f = f_{\infty}$. Now we finish the estimates of all error terms in position direction. \\ 

For each component $i$ in velocity direction, we simply have only one term
\begin{equation*}
\nabla_{v_i} \overline{R}_N = \nabla_{v_i} \nabla_{v_i} \log f \cdot \bigg\{ \frac{1}{N}\sum_{j\neq i}K(x_i-x_j) - K\star \rho(x_i) \bigg\} = \nabla_{v_i} \nabla_{v_i} \log f \cdot R^2_{N,i}.
\end{equation*}
We take $\nabla_{v_i} \nabla_{v_i} \log f(x_i,v_i) = \nabla_{v_i} \nabla_{v_i} \log f_{\infty}(x_i,v_i) = - \frac{\gamma}{\sigma} Id_{d\times d}$ and use Young inequality, then we obtain
\begin{equation}\label{IEq:error term VX}
\begin{split}
& \sum_{i=1}^N \int f_N \langle \nabla_{v_i} \nabla_{v_i} \log f_{\infty} \cdot  R^2_{N,i}, F^i \nabla_{x_i} u \rangle \\ \leq & \sum_{i=1}^N \int f_N |\nabla_{v_i} \nabla_{v_i} \log f_{\infty} \cdot \sqrt{F^i} R^2_{N,i}|^2 + \frac{1}{4} \int f_N |\sqrt{F} \nabla_X u|^2,
\end{split}
\end{equation}
\begin{equation}\label{IEq:error term VV}
\begin{split}
& \sum_{i=1}^N \int f_N \langle \nabla_{v_i} \nabla_{v_i} \log f_{\infty} \cdot R^2_{N,i}, G^i \nabla_{v_i} u \rangle \\ \leq & \sum_{i=1}^N \frac{1}{\sigma}\int f_N |\nabla_{v_i} \nabla_{v_i} \log f_{\infty} \cdot G^i R^2_{N,i}|^2 + \frac{\sigma}{4} \int f_N |\nabla_V u|^2,
\end{split}
\end{equation}
now we finish the estimates of error terms on velocity direction.

Finally, we collect all terms of (\ref{IEq:error term XX1}), (\ref{IEq:error term XV1}), (\ref{IEq:error term XX3}), (\ref{IEq:error term XV3}), (\ref{IEq:error term VX}) and(\ref{IEq:error term VV}), we finish the proof.
\end{proof}

\begin{remark}\label{remark: nabla^2-log-f}
For error terms of relative Fisher Information, the most difficult part is to show the uniform-in-time estimates for the following terms,
\begin{equation*}
\sum_{i,k=1}^N \int f_N \langle \nabla_{v_i} \nabla_{v_i} \log f \cdot R^2_{N,i}, M_t^{ik} C_k u \rangle,
\end{equation*}
\begin{equation*}
\sum_{i,k=1}^N \int f_N \langle \nabla_{x_i} \nabla_{v_i} \log f \cdot R^2_{N,i}, M_t^{ik} C_k u \rangle.
\end{equation*}
It is not easy to deal with the terms $\nabla_{v_i} \nabla_{v_i} \log f$ and $\nabla_{x_i} \nabla_{v_i} \log f$ by uniform-in-time way, that is one of the most important reason we select the reference measure as $f_{\infty}$. It is natural to conjecture that $|\nabla^2 \log f| \leq C(1 + \frac{1}{2}v^2 + V(x))^k$ for some positive order $k>0$. We might explore how to control these terms better in further study.
\end{remark}

Now let us focus on three error terms in Lemma \ref{Lemma:CS-error}. We rewrite them as following as Theorem 3 in \cite{InventionJW} has talked about,
\begin{equation}\label{Quantity:error1}
\sum_{i=1}^N \int f_N |R^1_{N,i}|^2 = \int f_N \bigg|\frac{1}{\sqrt{N}}\sum_{j,j\neq 1} \nabla K(x_1-x_j) - \nabla K\star \rho_{\infty}(x_1) \bigg|^2, 
\end{equation}

\begin{equation}\label{Quantity:error2}
\sum_{i=1}^N \int f_N |R^3_{N,i}|^2 = \int f_N \bigg| \frac{1}{\sqrt{N}} \sum_{j, j\neq 1} \nabla K(x_j - x_1) \cdot [v \ast \rho^v_{\infty}(v_1) - (v_1 - v_j)]\bigg|^2,
\end{equation}

\begin{equation}\label{Quantity:error3}
\sum_{i=1}^N \int f_N |R^2_{N,i}|^2 = \int f_N \bigg| \frac{1}{\sqrt{N}}\sum_{j,j\neq 1}K(x_1-x_j) - K\star \rho_{\infty}(x_1) \bigg|^2,
\end{equation}
for terms (\ref{Quantity:error1}), (\ref{Quantity:error2}) and (\ref{Quantity:error3}), we define 
\begin{equation*}
\psi_1(x_1,x_j) = \nabla K(x_1-x_j) - \nabla K\star \rho_{\infty}(x_1),
\end{equation*}

\begin{equation*}
\psi_2(x_1,x_j) = K(x_1-x_j) - K\star \rho_{\infty}(x_1),
\end{equation*}

\begin{equation*}
\psi_3(v_1,v_j) = \nabla K(x_j - x_1) \cdot [v \ast \rho^v_{\infty}(v_1) - (v_1 - v_j)].
\end{equation*}
We also define the $2p$ moment of $f_{\infty}$ for any $p \geq 1$,
\begin{equation}
M_p = \left( \int f_{\infty} (1 + \frac{1}{2}v^2 + \frac{1}{2}x^2)^p \right)^{\frac{1}{p}}. 
\end{equation}
By similar argument in Section 1.4 in \cite{JFAJW}, we have
\begin{equation}
\sup_p \frac{M_p}{p} := \Lambda < \frac{1}{e \lambda},
\end{equation}
where $\lambda$ satisfies
\begin{equation}
\int f_{\infty} e^{\lambda(1 + 
\frac{1}{2}v^2 + \frac{1}{2}x^2)} \ud z < \infty,
\end{equation}
We easily take $\lambda \in (0,\beta)$ in the following.

By Lemma 1 in \cite{InventionJW}, i.e. Gibbs inequality, we have 
\begin{equation*}
\frac{1}{N} \int f_N |R^k_N|^2 dZ \leq \frac{1}{\nu}H_N(f_N| \bar{f_N}) + \frac{1}{N} \int \bar{f_N} \exp{(\nu |R^k_N|^2)} \ud Z
\end{equation*}
for some $\nu > 0$ and $k = 1,2,3$, it is sufficient to show that
\begin{equation}\label{Quantity:psi-12}
\int \bar{f_N} \exp \bigg(
\frac{\nu}{N} \sum_{j_1,j_2=1}^N \psi_i(x_1, x_{j_1}) \psi_i(x_1,x_{j_2})\bigg) \ud Z < \infty, \ \ \ i = 1,2,
\end{equation}
and
\begin{equation}\label{Quantity:psi-3}
\int \bar{f_N} \exp \bigg(
\frac{\nu}{N} \sum_{j_1,j_2=1}^N \psi_3(v_1, v_{j_1}) \psi_3(v_1,v_{j_2})\bigg) \ud Z < \infty
\end{equation}
for some $\nu > 0$. We observe that $\psi_1(x_1,x_j)$ is bounded by Assumption \ref{Assumption:W1} and $\int \psi_1(z,x) dx = 0$, which implies that we can directly use  Theorem 3 in \cite{InventionJW} for this type of $\psi_1(x_1 - x_j)$. We recall  the following proposition in \cite{InventionJW},
\begin{proposition}\label{proposition:LDE1}
Assume that $\psi \in L^{\infty}$ with $\| \psi\|_{L^{\infty}} < \frac{1}{2e}$, and for any fixed $y$,
\begin{equation*}
\int \psi(y,x) f_{\infty} \ud z = 0,
\end{equation*}
then 
\begin{equation}
\int \bar{f_N} \exp \bigg(
\frac{1}{N} \sum_{j_1,j_2=1}^N \psi(x_1, x_{j_1}) \psi(x_1,x_{j_2})\bigg) \ud Z < C,
\end{equation}
where $\bar{f_N} = f_{\infty}^{\otimes N}$ and $C > 0$ is independent of $N$.
\end{proposition}
\begin{proof}
Theorem 3 of \cite{InventionJW}.
\end{proof}
Now let $\psi = \nu \psi_1$ for some $\nu > 0$, we obtain the desired result. For the cases $\psi_{2}$ and $\psi_{3}$, we prove the following proposition:
\begin{proposition}\label{proposition:LDE2}
Assume that $|\psi(m,n)| \leq C(1 + |n|)$ with $C \leq \frac{1}{8e^2 \Lambda}$, and for any fixed $m$, $\int \psi(m,n) f_{\infty}dz = 0$. Here $(m,n)$ takes position pair $(y,x) \in \Omega \times \Omega$ or velocity pair $(u,v) \in \R^d \times \R^d$, then we have
\begin{equation}
\int \bar{f_N} \exp \bigg (
\frac{1}{N} \sum_{j_1,j_2=1}^N \psi(m_1, n_{j_1}) \psi(m_1,n_{j_2})\bigg) \ud Z < C,
\end{equation}
where $\bar{f_N} = f_{\infty}^{\otimes N}$ and $C > 0$ is independent of $N$.
\end{proposition}

\begin{proof}
We use $(m,n)$ to denote position pair $(y,x) \in \Omega \times \Omega$ or velocity pair $(u,v)$, and we show the proof of these two cases together. Since
\begin{equation*}
\exp(A) \leq \exp(A) + \exp(-A) = 2 \sum_{k=0}^{\infty} \frac{1}{(2k)!} A^{2k}, 
\end{equation*}
it suffices only to bound the series with even terms,
\begin{equation*}
2 \sum_{k=0}^{\infty} \frac{1}{(2k)!} \int \bar{f_N} \bigg(
\frac{1}{N} \sum_{j_1,j_2=1}^N \psi(m_1, n_{j_1}) \psi(m_1,n_{j_2}) \bigg)^{2k} \ud Z,
\end{equation*}
where in general the $k$-th term can be expanded as
\begin{equation}\label{IEq:proposition4-expand}
\begin{split}
& \frac{1}{(2k)!} \int \bar{f_N} \bigg(
\frac{1}{N} \sum_{j_1,j_2=1}^N \psi(m_1, n_{j_1}) \psi(m_1,n_{j_2})\bigg)^{2k} \ud Z \\ = & \frac{1}{(2k)!} \frac{1}{N^{2k}} \sum_{j_1,...,j_{4k}=1}^N
\int \bar{f_N} \psi(m_1,n_{j_1})... \psi(m_1,n_{j_{4k}}) \ud Z.
\end{split}
\end{equation}
Now we divide the proof in two different cases: where $k$ is small compared to $N$ and in the simpler case where $k$ is comparable to or larger than $N$. \\

$Case: 3 \leq 3k \leq N.$ In this case, for each term of \eqref{IEq:proposition4-expand}, we have
\begin{equation}
\int \bar{f_N}\psi(m_1,n_{j_1})... \psi(m_1,n_{j_{4k}})dZ \leq C^{4k} \int \bar{f_N}(1 + |n_{j_1}|)...(1 + |n_{j_{4k}}|) \ud Z ,
\end{equation}
since we have cancellation condition 
$\int \psi(m,n) f_{\infty}dz = 0$, the whole estimates relies on how many choices of multi-indices $(j_1,...,j_{4k})$ leading to a non-vanishing term. By notations of Theorem 3 and Theorem 4 in \cite{JFAJW}, we could denote the set of multi-indices $(j_1,...,j_{4k})$ s.t.
\begin{equation}
\int \bar{f_N} \psi(m_1,n_{j_1})... \psi(m_1,n_{j_{4k}}) \ud Z \neq 0,
\end{equation}
by $\mathcal{E}_{N-1,4k}$ since $j_i \neq 1$ for all $1 \leq i \leq 4k$, we also denote $(a_1,...,a_N)$ the multiplicity for $(j_1,...,j_{4k})$, i.e.
\begin{equation}
a_l = |\{ \nu \in \{1,...,4k\}, j_{\nu} = l\}|.
\end{equation}
Actually, if there exists $l \neq 1$ s.t. $a_l = 1$, then the variable $x_l$ enters exactly once in the integration, assume for simplicity that $j_1 = l$ then
\begin{equation*}
\begin{split}
&\int \bar{f_N}\psi(m_1,n_{j_1})... \psi(m_1,n_{j_{4k}}) \ud Z \\ = & \int \psi(m_1,n_{j_1})... \psi(m_1,n_{j_{4k}})\Pi_{i\neq l} f_{\infty}(z_i) \int \psi(m_1,n_l)f_{\infty}(z_l) = 0,
\end{split}
\end{equation*}
by the assumption of vanishing condition for $\psi$, provided $l = j_1 \neq 1$. Using these notations, we have 
\begin{equation}\label{IEq:CS-4k}
\begin{split}
& \frac{1}{(2k)!} \frac{1}{N^{2k}} \sum_{j_1,...,j_{4k} = 1}^N \int \bar{f_N} \psi(m_1,n_{j_1})... \psi(m_1,n_{j_{4k}}) \ud Z \\ = &  \frac{1}{(2k)!} \frac{1}{N^{2k}} \sum_{(j_1,...,j_{4k}) \in \mathcal{E}_{N-1,4k}} \int \bar{f_N} \psi(m_1,n_{j_1})... \psi(m_1,n_{j_{4k}}) \ud Z \\ \leq &  \frac{1}{(2k)!} \frac{1}{N^{2k}} \sum_{(j_1,...,j_{4k}) \in \mathcal{E}_{N-1,4k}} \Bigg( \sum_{(j_{2k+1},...,j_{4k}) \in \mathcal{E}_{N-1,4k}} 1 \Bigg) \int \bar{f_N} |\psi(m_1,n_{j_1})|^2... |\psi(m_1,n_{j_{2k}})|^2 \ud Z \\ & +  \frac{1}{(2k)!} \frac{1}{N^{2k}} \sum_{(j_1,...,j_{4k}) \in \mathcal{E}_{N-1,4k}} \Bigg( \sum_{(j_1,...,j_{2k}) \in \mathcal{E}_{N-1,4k}} 1 \Bigg) \int \bar{f_N} |\psi(m_1,n_{j_{2k+1}})|^2... |\psi(m_1,n_{j_{4k}})|^2 \ud Z,
\end{split}
\end{equation}
where $\left( \sum_{(j_{2k+1},...,j_{4k}) \in \mathcal{E}_{N-1,4k}} 1 \right)$ for the first line, we mean that for each fixed former $2k$ indices in $(j_1,...,j_{4k}) \in \mathcal{E}_{N-1,4k}$, we add all possible cases about later $2k$ indices in $\mathcal{E}_{N-1,4k}$, moreover, we have
\begin{equation}
\begin{split}
\int \bar{f_N} |\psi(m_1, n_{j_1})|^2...|\psi(m_1, n_{j_{2k}})|^2dZ = & C^{4k}\int \bar{f_N} [(1 + |n_2|)^2]^{a_2}...[(1 + |n_N|)^2]^{a_N} \ud Z \\ = &  C^{4k} M_{a_1}^{a_1}M_{a_2}^{a_2}...M_{a_N}^{a_N} \leq (C\Lambda)^{2k} a_1^{a_1}...a_N^{a_N},
\end{split}
\end{equation}
with the convention that $0^0 = 1$, then we have
\begin{equation*}
\begin{split}
& \frac{1}{(2k)!} \frac{1}{N^{2k}} \sum_{(j_1,...,j_{4k}) \in \mathcal{E}_{N-1,4k}} \Bigg( \sum_{(j_{2k+1},...,j_{4k}) \in \mathcal{E}_{N-1,4k}} 1 \Bigg) \int \bar{f_N} |\psi(x_1, x_{j_1})|^2...|\psi(x_1, x_{j_{2k}})|^2 \ud Z \\ \leq & \frac{1}{(2k)!} \frac{1}{N^{2k}}(2C^2\Lambda)^{2k} \sum_{l=1}^{2k} \sum_{J_{4k} \in \mathcal{E}_{N-1,4k}, |S(J_{2k})| = l} \Bigg( \sum_{(j_{2k+1},...,j_{4k}) \in \mathcal{E}_{N-1,4k}} 1 \Bigg) a_1^{a_1}...a_N^{a_N},
\end{split}
\end{equation*}
and 
\begin{equation}\label{IEq:last 2k}
\Bigg( \sum_{(j_{2k+1},...,j_{4k}) \in \mathcal{E}_{N-1,4k}} 1 \Bigg) \leq \sum_{a'_1 + ... + a'_l = 2k, a'_1 \geq 1,...,a'_l \geq 1} \frac{(2k)!}{a'_1!...a'_l!} \leq l^{2k},
\end{equation}
we denote
\begin{equation}
\sum_{l=1}^{2k} \sum_{J_{4k} \in \mathcal{E}_{N-1,4k}, |S(J_{2k})| = l} l^{2k} a_1^{a_1}...a_N^{a_N} = \sum_{l=1}^{2k} l^{2k} C_N^l U^l_{N,2k},
\end{equation}
where
\begin{equation}
U_{N,2k}^l := \sum_{a_1 +... + a_l = 2k, a_1 \geq 1,...,a_l \geq 1} \frac{(2k)!}{a_1!...a_l!} a_1^{a_1}...a_l^{a_l},
\end{equation}
using (2.13) and Lemma 6 of \cite{JFAJW}, we have
\begin{equation}\label{IEq:U-2k-l}
U_{N,2k}^l \leq e^{2k} (2k)! \Bigg( \sum_{a_1 +... + a_l = 2k, a_1 \geq 1,...,a_l \geq 1} 1 \Bigg) \leq \frac{1}{\sqrt{k}} (2e)^{2k}(2k)!,
\end{equation}
Insert (\ref{IEq:last 2k}) and (\ref{IEq:U-2k-l}) into (\ref{IEq:CS-4k}), we obtain that
\begin{equation}
\begin{split}
& \frac{1}{(2k)!} \frac{1}{N^{2k}} \sum_{j_1,...,j_{4k} = 1}^N \int \bar{f_N} \psi(m_1,n_{j_1})... \psi(m_1,n_{j_{4k}}) \ud Z \\  \leq &  \frac{(4eC^2\Lambda)^{2k}}{(2k)!} \frac{2}{N^{2k}}  (2k)!
\frac{1}{\sqrt{k}} \sum_{l=1}^{2k} C_N^l l^{2k} \\ = & \frac{2}{\sqrt{k}} (8e^2C^2\Lambda)^{2k} \sum_{l=1}^{2k} C_N^l l^{2k} N^{-2k}, 
\end{split}
\end{equation}
furthermore, by Stirling's formula
\begin{equation}
C_N^l \frac{l^{2k}}{N^{2k}} = \frac{l^{2k-l}}{N^{2k-l}} \frac{l^l}{N^l}\frac{N!}{(N-l)! l!} \leq \frac{1}{\sqrt{\pi l}}\sqrt{\frac{N}{N-l}} \left(\frac{N}{N-l}\right)^{N-l},
\end{equation}
and by assumption $l < 3k \leq N$ gives that $\frac{N}{N-l} \leq \frac{3}{2}$, thus
\begin{equation}\label{IEq:propostion4-case1}
\frac{1}{(2k)!} \frac{1}{N^{2k}} \sum_{j_1,...,j_{4k} = 1}^N \int \bar{f_N} \psi(m_1,n_{j_1})... \psi(m_1,n_{j_{4k}}) \ud Z \leq 4k (8e^2C^2\Lambda)^{2k},
\end{equation}
we finish the proof of this case. \\
 
$Case: 3k > N.$ By Cauchy inequality, we have
\begin{equation}
\begin{split}
\frac{1}{N} \sum_{j_1,j_2 = 1}^N \psi(m_1, n_{j_1}) \psi(m_1, n_{j_2}) \leq \frac{C^2}{N} \sum_{j_1, j_2}^N (1 + |x_{j_1}|)(1 + |x_{j_2}|) \leq C^2 \sum_{j=1}^N (1 + |x_j|)^2,
\end{split}
\end{equation}
then we have
\begin{equation}\label{IEq:propostion4-case2}
\begin{aligned}
& \frac{1}{(2k)!} \frac{1}{N^{2k}} \sum_{j_1,...,j_{4k} = 1}^N \int \bar{f_N} \psi(m_1,n_{j_1})... \psi(m_1,n_{j_{4k}}) \ud Z \\ \leq & \frac{C^{4k}}{(2k)!} \int \bar{f_N} \bigg( \sum_{j=1}^N (1 + |x_j|)^2 \bigg)^{2k} \ud Z \leq (5e^2C^2 \Lambda)^{2k}, 
\end{aligned}
\end{equation}
which follows the same argument of Proposition 4 of \cite{JFAJW}.

Finally, we combine \eqref{IEq:propostion4-case1} and \eqref{IEq:propostion4-case2}, then we complete the proof.
\end{proof}

Now let us take $\psi = \nu \psi_2$ or $\nu \psi_3$, we obtain the desired  bound $O(1)$ of ($\ref{Quantity:psi-12}$) and ($\ref{Quantity:psi-3}$) by above proposition.

\section{Proof of the main results}
In this section, we prove three main theorems we claimed in Section \ref{Main results and examples}. We roughly divide all three proofs into two steps. 

$\bullet$ Step one is show the uniform-in-$N$ estimates for the time evolution of our new quantity $\mathcal{E}_N(f_N|\bar{f}_N)$. The key point is to select suitable weight matrix function $M_t$ to estimate the lower bound of new matrix $S_t$ on the right hand side of (\ref{IEq:evlov-Fisher}), and the lower bound of matrix $S_t$ should be independent on particle number $N$. In the case of uniform-in-time propagation of chaos, we also need to deal with the second term on the right hand side of \eqref{IEq:evlov-Fisher},
\begin{equation*}
-2 \int_{(\Omega \times \R^d)^N} f_N \z \nabla \overline{R}_N, M \nabla \log \frac{f_N}{f^{\otimes N}}\y \ud Z,
\end{equation*}
this is exactly what we do in Lemma \ref{Lemma:CS-error}.

$\bullet$ Step two is to establish the Gronwall type inequality for time evolution of our new quantity $\mathcal{E}_N(f_N|\bar{f}_N)$ by some functional inequalities developed in Section \ref{section:WLSI} and large deviation estimates proved in Section \ref{section:Large deviation estimates}. In the case of uniform-in-time propagation of chaos, we require Log-Sobolev inequality and weighted Log-Sobolev inequality 
for non-linear equilibrium $f_{\infty}$. However, we require that the $N$ particle invariant measure $f_{N,\infty}$ satisfies uniform-in-$N$ Log-Sobolev inequality for the long time convergence of Eq.\eqref{Eq:Liouville equation}, which is more harder than the case of uniform in time propagation of chaos which takes the reference measure as the tensor form $f_{\infty}^{\otimes N}$. We refer \cite{UPIandULSI} for the study of uniform-in-$N$ functional inequalities.

\subsection{Proof of Theorem \ref{Thm:theorem1.3}}
\begin{proof}

$Step1.$ Let us take $\bar{f}_N = f_{N,\infty}$ and $\overline{R}_N = 0$. We observe that $\nabla_X \nabla_V \log f_{N,\infty} = \nabla_V \nabla_X \log f_{N,\infty} = 0$ and $\nabla_V \nabla_V \log f_{N,\infty} = - \frac{1}{2}\sigma^{-1} \gamma Id_{Nd \times Id}$, then the matrix $S$ in (\ref{IEq:evlov-Fisher}) of Corollary \ref{Cor:block matrix} reads as
\begin{equation*}
S = 
\left(
\begin{array}{cc}
2F - \partial_t E - L^{\ast}_NE \ \ & 4\gamma F - \partial_t F - L^{\ast}_NF + 2E \nabla^2 U \\
2G - \partial_t F - L^{\ast}_NF \ \ & 4\gamma G - \partial_t G - L^{\ast}_N G + 2F \nabla^2 U
\end{array}
\right).
\end{equation*}
Now we select the weight matrix $M$ as a constant matrix, 
\begin{equation}\label{M:case-1.1}
E = \text{diag}\{\delta a^3,...,\delta a^3\}, \ \ \ F = \text{diag}\{\delta a^2,...,\delta a^2\}, \ \ \ G = \text{diag}\{2\delta a,...,2\delta a\},
\end{equation}
where $E,F,G$ are $Nd \times Nd$ constant diagonal matrices
and $a > 0, \delta > 0$ are to be confirmed, then the matrix $S_t$ reduces to
\begin{equation*}
S = 
\left(
\begin{array}{cc}
2F \ \ & 4\gamma F + 2E \nabla^2 U \\
2G \ \ & 4\gamma 
G + 2F \nabla^2 U
\end{array}
\right).
\end{equation*}
Observe that the top left corner matrix $F$ implies the entropy dissipation in position direction, and right down corner concerns the entropy dissipation in velocity direction. If $\nabla^2 U$ has a positive lower bound, i.e.$\nabla^2 U \geq \kappa Id_{dN \times dN}$ for some $\kappa > 0$, then we have $4 \gamma G + 2F \nabla^2 U \geq (8 \delta a \gamma + 2\delta a^2 \kappa) Id_{dN \times dN}$. However, our assumptions also treat $\nabla^2 U$ doesn't have positive lower bound. Since we assume that $\|\nabla^2 V \|_{L^{\infty}} < \infty$ and $\|\nabla^2 W \|_{L^{\infty}} < \infty$, we have
$\|\nabla^2 U \|_{L^{\infty}} \leq C_K + C_V < \infty$. Now for right down corner, we choose $a$ such that $a\|\nabla^2 U\|_{L^{\infty}} < 2 \gamma$, then we have
\begin{equation}\label{Restriction:a-1.1}
a \leq \frac{2\gamma}{C_K + C_V},
\end{equation}
and
\begin{equation}\label{VV-theorem-1.1}
\langle \nabla_V u, (4\gamma G + 2F \nabla^2 U) \nabla_Vu \rangle \geq 4 \delta a \gamma|\nabla_V u|^2.
\end{equation}
For top right corner and left down corner, we deal with cross terms.
It is easy to obtain that
\begin{equation}\label{VX-theorem-1.1}
\begin{aligned}
& \langle \nabla_Xu, [4\gamma F + 2E \nabla^2 U + 2G] \nabla_Vu \rangle \\ \leq & \delta a[4 + 4a\gamma + 2 a^2(C_K + C_V)]|\nabla_V u||\nabla_X u| \\ \leq & \delta a(4 + 8a\gamma) |\nabla_V u||\nabla_X u| \\ \leq & \frac{1}{2}\delta a^2 |\nabla_X u|^2 + \frac{1}{2}\delta(4 + 8a\gamma)^2 |\nabla_V u|^2
\end{aligned}
\end{equation}
by \eqref{Restriction:a-1.1} in the third line and Cauchy inequality in the fourth line. Combining \eqref{VV-theorem-1.1} and \eqref{VX-theorem-1.1}, we have
\begin{equation}
\begin{aligned}
- \int f_N \z \nabla u, S\nabla u \y \leq & - \frac{3}{2}\delta a^2 \int f_N |\nabla_X u|^2 - 4\delta a\gamma \int f_N |\nabla_V u|^2 \\ & + m_1 \int f_N |\nabla_V u|^2,
\end{aligned}
\end{equation}
provided $m_1 = \frac{1}{2}\delta(4 + 8a \gamma)^2 > 0$. Now we choose $\delta$ such that $m_1 \leq \frac{\sigma}{4}$, then we have
\begin{equation}
\delta \leq \frac{\sigma}{2(4 + 8a\gamma)^2},
\end{equation}
and
\begin{equation}
\frac{d}{dt} \mathcal{E}^M_N(t) \leq - \frac{c_1}{N} \int f_N |\nabla_X u|^2 - \frac{c_2}{N} \int f_N |\nabla_V u|^2, 
\end{equation}
provided $c_1 = \frac{3}{2} \delta a^2$ and $c_2 = 4\delta a \gamma + \frac{\sigma}{2}$. Finally we obtain the entropy dissipation in $X$ and $V$ direction with constants which are independent of $N$.

$Step2.$ We directly use the results in \cite{UPIandULSI} in this step, under Assumption \ref{Assumption:V1} and \ref{Assumption:W1} with $C_K < 1$, we have uniform-in-$N$ Log Sobolev inequality of $N$ particle invariant measure $f_{N, \infty}$,
\begin{equation*}
Ent_{\mu_N} (f^2) \leq 
\rho_{LS} \int_{\R^{dN}}
|\nabla f|^2 d\mu_N, \ \ f \in C^1_b(\R^{dN}),
\end{equation*}
by our notation, let $\mu_N = f_{N,\infty}$, we have
\begin{equation}
\frac{d}{dt} \mathcal{E}^M_N(t) \leq -c \mathcal{E}^M_N(t),
\end{equation}
where $c = \frac{1}{2(1+\rho_{LS})}\min\{ \frac{3}{2} \delta a^2, \frac{\sigma}{2}\}$ is a constant depending on $\rho_{LS}, U, \sigma, \gamma$, then we finish the proof.
\end{proof}

\subsection{Proof of Theorem \ref{Thm:theorem1.4}}
In this subsection, we prove the particle system \eqref{Eq:particle system} approximates the unique equilibrium \eqref{Eq:nonlinear equilibrium} of Vlasov-Fokker-Planck equation \eqref{Eq:limiting equation} when diffusion strength $\sigma$ is large enough. We deal with two kinds of confining potentials by selecting different weight matrix $M$. When confining potentials $V$ satisfy Assumption \ref{Assumption:V1}, we choose
\begin{equation}\label{M:case1-1.2}
E = \text{diag}\{\delta a^3,...,\delta a^3\}, \ \ \ F = \text{diag}\{\delta a^2,...,\delta a^2\}, \ \ \ G = \text{diag}\{2\delta a,...,2\delta a\},
\end{equation}
where $E,F,G$ are $Nd \times Nd$ constant diagonal matrices
and $\delta, a > 0$ are to be confirmed. When confining potentials $V$ satisfy Assumption \ref{Assumption:V2}, we choose
\begin{equation*}\label{M:case2-1.2}
\left \{
\begin{aligned}
& E = \text{diag}\{e(z_1)Id_{d \times d},...,e(z_N)Id_{d \times d}\}, \\ & F = \text{diag}\{f(z_1)Id_{d \times d},...,f(z_N)Id_{d \times d}\}, \\ & G = \text{diag}\{g(z_1)Id_{d \times d},...,g(z_N)Id_{d \times d}\}, 
\end{aligned}
\right.
\end{equation*}

where $E,F,G$ are also $Nd \times Nd$ diagonal matrices, and $e(z), f(z), g(z)$ are $d \times d$ diagonal matrices which should be understood as $e(z)Id_{d\times d}$, $f(z)Id_{d\times d}$ and $g(z)Id_{d\times d}$, we omit the symbol ``$Id_{d \times d}$" for convenience. We choose $e(z), f(z)$ and $g(z)$ as
\begin{equation*}
\begin{split}
e(z) = \delta a^3(H(z))^{-3\theta},\ \ 
b(z) = \delta a^2(H(z))^{-2\theta},\ \  c(z) = 2 \delta a(H(z))^{-\theta},
\end{split}
\end{equation*}
where 
\begin{equation*}
H(z) = \frac{v^2}{2} + V(x) + H_0, \ \ \ H_0 > 0,
\end{equation*}
and $\delta, a, \theta, H_0 > 0$ are to be confirmed. In the next, we establish Gronwall type inequality for the quantity $\mathcal{E}_N^M(f^t_N|f^{\otimes N}_{\infty})$. \\

\begin{proof}
$Case 1.$ The confining potentials $V$ satisfy Assumption \ref{Assumption:V1} and we select $M$ as constant diagonal matrix defined by \eqref{M:case1-1.2}.

$Step 1.$ The matrix $M$ is exactly the same as we have used in the last subsection, we immediately have 
\begin{equation}\label{1.2-case1-dissipation}
\begin{aligned}
- \int f_N \z \nabla u, S \nabla u \y \leq & - \frac{3}{2}\delta a^2 \int f_N |\nabla_X u|^2 - 4\delta a\gamma \int f_N |\nabla_V u|^2 \\ & + m_1 \int f_N |\nabla_V u|^2,
\end{aligned}
\end{equation}
provided $m_1 = \frac{1}{2}\delta(4 + 8a \gamma)^2$ and $a \leq \frac{2\gamma}{C_K + C_V}$. By Lemma \ref{Lemma:CS-error}, we have
\begin{equation}\label{1.2-case1-error}
\begin{aligned}
\int f_N \z \nabla \overline{R}_N, M \nabla u\y \leq & n_1 \int f_N |\nabla_X u|^2 + n_2 \int f_N |\nabla_V u|^2 \\ & + \frac{8\delta a^3}{\sigma} \int f_N |R_N^1|^2 + \frac{\delta a^2 \gamma^2(1+a)}{2\sigma^2} \int f_N |R_N^3|^2 \\ & + \big( \frac{2 \gamma^2 \delta a^3}{\sigma^2} + \frac{\gamma^2 \delta^2 a^2}{\sigma^3}\big) \int f_N |R_N^2|^2 \\ & + \sum_{i=1}^N \frac{\sigma}{16} \int f_N | \sqrt{E^{ii}} \nabla_{v_i} \nabla_{x_i} u|_{\R^{2d}}^2 + \frac{\sigma}{16} \int f_N | \sqrt{F^{ii}} \nabla_{v_i} \nabla_{v_i} u|_{\R^{2d}}^2,
\end{aligned}
\end{equation}
provided $n_1 = \frac{1}{4}\delta a^2 + (C_K + \frac{1}{2})\delta a^3$ and $n_2 = \frac{\sigma}{4} + \delta a^2(C_K a + C_K + \frac{1}{2})$. By Lemma \ref{lemma:error term of RE} and Young inequality, we have
\begin{equation}\label{1.2-case1-entropy}
\frac{d}{dt} H_{N}(f^t_N|f^{\otimes N}_{\infty}) \leq - \frac{3 \sigma}{4} \frac{1}{N} \int f_N|\nabla_V u|^2 + \frac{1}{\sigma} \frac{1}{N} \int f_N|R^0_N|^2.
\end{equation}
Combining $\eqref{1.2-case1-dissipation}$, $\eqref{1.2-case1-error}$ and $\eqref{1.2-case1-entropy}$, we obtain 
\begin{equation}\label{1.2-case1-step1-c'}
\begin{aligned}
\frac{d}{dt}  \mathcal{E}_N^M(f_N|f_{\infty}^{\otimes N}) \leq & - \frac{c_1}{N} \int f_N |\nabla_X u|^2 - \frac{c_2}{N} \int f_N |\nabla_V u|^2 \\ & + \frac{8\delta a^3}{\sigma} \frac{1}{N} \int f_N |R_N^1|^2 + \frac{\delta a^2 \gamma^2(1+a)}{2\sigma^2} \frac{1}{N}\int f_N |R_N^3|^2 \\ & + \big( \frac{2 \gamma^2 \delta a^3}{\sigma^2} + \frac{\gamma^2 \delta^2 a^2}{\sigma^3}\big) \frac{1}{N} \int f_N |R_N^2|^2 + \frac{1}{\sigma} \frac{1}{N} \int f_N|R_N|^2,
\end{aligned}
\end{equation}
where $c_1 = -\frac{1}{2}\delta a^2 + (C_K + \frac{1}{2})\delta a^3$ and $c_2 = -\frac{\sigma}{2} - 4 \delta a \gamma + \delta a^2(C_K a + C_K + \frac{1}{2}) + m_1$. Recall we choose $a$ such that $a \leq \frac{2 \gamma}{C_K + C_V}$, now we update $a$ such that
\begin{equation}\label{1.2-case1-step1-a}
a \leq \min \Big \{ \frac{2 \gamma}{C_K + C_V}, \frac{1}{4C_K + 2} \Big \},
\end{equation}
then we have $c_1 \leq -\frac{1}{4}\delta a^2$ and $\delta a^2(C_K a + C_K + \frac{1}{2}) \leq \delta a$. In the next we choose $\delta$ such that 
\begin{equation}\label{1.2-case1-step1-delta}
\delta \leq \frac{\sigma}{4[8 + a + 28a\gamma + 32a^2 \gamma^2]},
\end{equation}
we have $c_2 \leq - \frac{\sigma}{4}$. Finally, we update \eqref{1.2-case1-step1-c'} as
\begin{equation}\label{1.2-case1-step1-final}
\begin{aligned}
\frac{d}{dt}  \mathcal{E}_N^M(f_N|f_{\infty}^{\otimes N}) \leq & - \frac{\delta a^2}{4} \frac{1}{N} \int f_N |\nabla_X u|^2 - \frac{\sigma}{4} \frac{1}{N} \int f_N |\nabla_V u|^2 \\ & + \frac{8\delta a^3}{\sigma} \frac{1}{N} \int f_N |R_N^1|^2 + \frac{\delta a^2 \gamma^2(1+a)}{2\sigma^2} \frac{1}{N}\int f_N |R_N^3|^2 \\ & + \big( \frac{2 \gamma^2 \delta a^3}{\sigma^2} + \frac{\gamma^2 \delta^2 a^2}{\sigma^3}\big) \frac{1}{N} \int f_N |R_N^2|^2 + \frac{1}{\sigma} \frac{1}{N} \int f_N|R_N|^2.
\end{aligned}
\end{equation}

$Step2.$ We start from \eqref{1.2-case1-step1-final}. By Proposition \ref{proposition:LDE1}, we have
\begin{equation}\label{1.2-case1-step1-LDE1}
\frac{1}{N}\int f_N|R^1_N|^2 \leq \frac{1}{\nu}H_N(f^t_N|f^{\otimes N}_{\infty}) + \frac{C}{N}, 
\end{equation}
and by Proposition \ref{proposition:LDE2}, we have
\begin{equation}\label{1.2-case1-step1-LDE2}
\frac{1}{N}\int f_N|R^i_N|^2 \leq \frac{1}{\nu}H_N(f^t_N|f^{\otimes N}_{\infty}) + \frac{C}{N}, 
\end{equation}
where $i = 0,2,3$, the constant $\nu > 0$ satisfies $(16C^2_Ke^2 \Lambda) \nu < 1$ and the constant $C > 0$ depends on $\nu, \sigma, \gamma, K$ and $V$. Using \eqref{1.2-case1-step1-LDE1} and \eqref{1.2-case1-step1-LDE2}, we have
\begin{equation}\label{1.2-case1-step1-last two}
\begin{aligned}
\frac{d}{dt}  \mathcal{E}_N^M(f_N|f_{\infty}^{\otimes N}) \leq & - \frac{\delta a^2}{4} \frac{1}{N} \int f_N |\nabla_X u|^2 - \frac{\sigma}{4} \frac{1}{N} \int f_N |\nabla_V u|^2 \\ & + \frac{n_3}{\nu} H_N(f_N|f_{\infty}^{\otimes N}) + \frac{C}{N},   
\end{aligned}
\end{equation}
where
\begin{equation*}
n_3 = \Big[\frac{8\delta a^3}{\sigma} + \frac{\delta a^2 \gamma^2(1+a)}{2\sigma^2} + \frac{2 \gamma^2 \delta a^3}{\sigma^2} + \frac{\gamma^2 \delta^2 a^2}{\sigma^3} + \frac{1}{\sigma} \Big].
\end{equation*}
Now let us determine the value of $a > 0$, $\delta > 0$ and $\nu > 0$ such that
\begin{equation*}
\frac{n_3 \rho_{ls}}{\nu} \leq \min \Big\{\frac{\delta a^2}{8}, \frac{\sigma}{8} \Big\},
\end{equation*}
where $\rho_{ls}$ is the Log-Sobolev inequality constant of $f_{\infty}$ (See Corollary \ref{Cor:LSI-f_infty}).
Recall the constant $a > 0$ satisfies
\begin{equation*}
a \leq \min \Big \{ \frac{2 \gamma}{C_K + C_V}, \frac{1}{4C_K + 2} \Big \} < 1,
\end{equation*}
now we choose
\begin{equation}\label{1.2-case1-step1-delta-final}
\begin{aligned}
\delta = \frac{\sigma}{4[10 + 28 \gamma + 32 \gamma^2]} \leq \frac{\sigma}{4[8 + a + 28a\gamma + 32a^2 \gamma^2]}
\end{aligned}
\end{equation}
by \eqref{1.2-case1-step1-delta}. We also choose $\frac{1}{\nu} = 20C_K^2 e \frac{\sigma}{\gamma}$ with $\Lambda \leq \frac{1}{e \lambda}$ satisfies $\lambda \in (0, \beta)$ in Proposition \ref{proposition:LDE1} and \ref{proposition:LDE2}. Observing that $\delta a^2 < \sigma$ by $a < 1$ and $\delta < \sigma$, we only need
\begin{equation*}
\frac{n_3 \rho_{ls}}{\nu} \leq \frac{\delta a^2}{8}.
\end{equation*}
Now we choose $a$ and $\sigma$ to make every term of $n_3$ less than $\frac{\delta a^2 \nu}{40 \rho_{ls}}$,
\begin{equation}\label{1.2-case1-step2-ieq}
\left\{
\begin{aligned}
& \frac{8\delta a^3}{\sigma} \leq \frac{\delta a^2 \nu}{40 \rho_{ls}} \Longrightarrow a \leq \frac{\gamma}{6400 e\rho_{ls} C_K^2}, \\ & \frac{\delta a^2 \gamma^2}{\sigma^2} \leq \frac{\delta a^2 \nu}{40 \rho_{ls}} \Longrightarrow \sigma \geq 3200\rho_{ls}e \gamma C_K^2, \\ & \frac{2 \gamma^2 \delta a^2}{\sigma^2} \leq \frac{\delta a^2 \nu}{40 \rho_{ls}} \Longrightarrow \sigma \geq 1600\rho_{ls}e \gamma C_K^2, \\ & \frac{\gamma^2 \delta a^2}{\sigma^2} \leq \frac{\delta a^2 \nu}{40 \rho_{ls}} \Longrightarrow \sigma \geq 3200\rho_{ls}e \gamma C_K^2, \\ & \frac{1}{\sigma} \leq \frac{\delta a^2 \nu}{40 \rho_{ls}} \Longrightarrow \sigma \geq \frac{160[10 + 28 \gamma + 32 \gamma^2]\rho_{ls}eC_K^2}{a^2 \gamma}.
\end{aligned}
\right.
\end{equation}
We choose $a$ such that
\begin{equation*}
a \leq \min \Big \{ \frac{2 \gamma}{C_K + C_V}, \frac{1}{4C_K + 2},  \frac{\gamma}{6400 e\rho_{ls} C_K^2} \Big \},
\end{equation*}
by the first line above and \eqref{1.2-case1-step1-a}, for example, saying  
\begin{equation*}
a = \frac{\min\{\gamma,1\}}{6400 e\rho_{ls} (C_K+1)^2 + 4C_V + 4C_K + 2},   \end{equation*}
then we have the range of diffusion constant in our result,
\begin{equation}\label{1.2-case1-step-sigma-final}
\sigma \geq \max \bigg\{ \frac{160[10 + 28 \gamma + 32 \gamma^2]\rho_{ls}e}{a^2 \gamma}, 3200\rho_{ls}e\gamma \bigg\}C_K^2.
\end{equation}
Finally, we update \eqref{1.2-case1-step1-last two} as the following
\begin{equation}
\begin{aligned}
\frac{d}{dt}  \mathcal{E}_N^M(f^t_N|f_{\infty}^{\otimes N}) \leq & - \frac{\delta a^2}{16} \frac{1}{N} \int f_N |\nabla_X u|^2 - \frac{\sigma}{16} \frac{1}{N} \int f_N |\nabla_V u|^2 \\ & - \frac{\delta a^2}{16 \rho_{ls}}H_N(f^t_N|f_{\infty}^{\otimes N}) + \frac{C}{N} \\ \leq & -\frac{\delta a^2}{16(\rho_{ls} + 1)}\mathcal{E}_N^M(f^t_N|f_{\infty}^{\otimes N}) + \frac{C}{N},   
\end{aligned}
\end{equation}
we finish the proof of this case. \\

$Case 2.$ The confining potentials $V$ satisfy Assumption \ref{Assumption:V2} and we select $M$ as diagonal mateix function defined by \eqref{M:case2-1.2}.  

$Step1.$ We take $E,F,G$ as the following tensorized matrices,
\begin{equation}\label{M:case2}
E = \text{diag}\{e(z_1),...,e(z_N)\}, \ \ F = \text{diag}\{f(z_1),...,f(z_N)\}, \ \ G = \text{diag}\{g(z_1),...,g(z_N)\}, 
\end{equation}
we take $e,f,g$ as scalar functions read as
\begin{equation}
\begin{split}
e(z) = \delta a^3 (H(z))^{-3\theta},\ \ 
f(z) = \delta a^2 (H(z))^{-2\theta},\ \ 
g(z) = 2 \delta a (H(z))^{-\theta},
\end{split}
\end{equation}
where 
\begin{equation}
\ \ H(z) = \frac{v^2}{2} + V(x) + H_0, \ \ \ H_0 > 0, 
\end{equation}
and $\delta, a, \theta, H_0 > 0$ are to be confirmed. We recommend \cite{EntropyMutiplier} to readers for similar argument of one particle version. Now let us estimate the lower bound of matrix $S$ in Corollary \ref{Cor:block matrix}. We recall 
\begin{equation*}
S = 
\left(
\begin{array}{cc}
2F - \partial_t E - L^{\ast}_NE \ \ & D_1 - \partial_t F - L^{\ast}_NF - 2E \nabla^2 U + 2 \gamma F\\
2G - \partial_t F - L^{\ast}_NF \ \ & D_2 - \partial_t G - L^{\ast}_N G - 2F \nabla^2 U + 2 \gamma G
\end{array}
\right),
\end{equation*}
after we take $\bar{f_N} = f^{\otimes N}_{\infty}$, we have 
\begin{equation}
S = 
\left(
\begin{array}{cc}
2F - L^{\ast}_NE \ \ & 4 \gamma F  - L^{\ast}_NF + 2E \nabla^2 U  \\
2G -  L^{\ast}_NF \ \ & 4 \gamma G - L^{\ast}_N G + 2F \nabla^2 U
\end{array}
\right).
\end{equation}
For the left top corner element, the diagonal matrix 
$F$ still offers the dissipation in position direction. Observing that $2F - L^{\ast}_N E$ is a $Nd \times Nd$ diagonal matrix, we could estimate them like scalar. Before that, we firstly estimate the upper bound of $L^{\ast}_N E, L^{\ast}_N F$ and $L^{\ast}_N G$. For some $s > 0$, we have
\begin{equation*}
\begin{aligned}
& L^{\ast}_N (H^{-s\theta}(z_i))
\\ = & \sum_{j=1}^N (v_j \cdot \nabla_{x_j} - \nabla_{x_j}U \cdot \nabla_{v_j} - \gamma v_j \cdot \nabla_{v_j} + \sigma \Delta_{x_j})H^{-s\theta}(z_i)
\\ = & -(s\theta) \bigg\{ \frac{\sigma d}{H} + \frac{\sigma (-s\theta -1)v_i^2}{H^2} - \frac{\gamma v_i^2}{H} - (\frac{1}{N} \sum_{j:j \neq i} v_i \cdot \nabla W(x_i - x_j))H^{-1} \bigg\}H^{-s
\theta}(z_i).
\end{aligned}
\end{equation*}
Recall we assume that $\|\nabla W\|_{L^{\infty}} < \infty$ in this case, we have
\begin{equation*}
|L^{\ast}_N (H^{-s\theta}(z_i))| \leq s\theta \big(\frac{\sigma d + \sigma(s\theta + 1)}{H_0} + \gamma + \| \nabla W\|_{L^{\infty}} \big) H^{-s \theta}(z_i),
\end{equation*}
now we choose $H_0$ large enough, which satisfies $\frac{\sigma(d + 3\theta + 1)}{H_0} \leq \gamma$, we have
\begin{equation}\label{1.2-case2-step1-LH}
|L^{\ast}_N (H(z_i)^{-s\theta})| \leq s\theta(2 \gamma + \| \nabla W\|_{L^{\infty}})H^{-s \theta}.
\end{equation}
For $i$-th component of $2F - L_N^{\ast}E$, we have  
\begin{equation*}
\begin{aligned}
2f(z_i)  - L_N^{\ast}(e(z_i)) = & 2 \delta a^2 H^{-2\theta}(z_i) - \delta a^3 L^{\ast}_N (H^{-3\theta}(z_i)) \\ \geq & 2 \delta a^2 H^{-2\theta} - 3\theta \delta a^3(2\gamma + \|\nabla W\|_{L^{\infty}}) H^{-3\theta}, 
\end{aligned}
\end{equation*}
again we choose $H_0$ such that $\frac{3\theta a (2\gamma + \|\nabla W\|_{L^{\infty}})}{H_0^{\theta}} \leq 1$, we have
\begin{equation}\label{1.2-case2-step1-XX}
\langle \nabla_Xu, (2F - L_N^{\ast}E)\nabla_X u \rangle \geq \delta a^2 \sum_{i=1}^N H^{-2\theta}(z_i
)|\nabla_{x_i}u|^2,
\end{equation}
provided $H_0$ satisfies
\begin{equation}\label{1.2-case2-step1-H0}
H_0 \geq \max \Big \{\frac{\sigma(d + 3 \theta + 1)}{\gamma}, [3\theta a (2\gamma + \|\nabla W\|_{L^{\infty}})]^{\frac{1}{\theta}} \Big \}.
\end{equation}
For the right down corner of $S$, we recall the formulation of $dN \times dN$ matrix $\nabla^2 U$ in $(\ref{Quantity:U})$, and we have $\|V^{-2\theta} \nabla^2 V\|_{L^{\infty}} \leq C_V^{\theta}$ by Assumption \ref{Assumption:V2}, then we obtain
\begin{equation}\label{1.2-case2-step1-V}
\|H^{-2\theta}\nabla^2 U\|_{L^{\infty}} \leq \|V^{-2\theta} \nabla^2 V\|_{L^{\infty}} + \frac{C_K}{H_0} \leq C_V^{\theta} + C_K < \infty,
\end{equation}
provided $H_0 > 1$ for convenience. Combining \eqref{1.2-case2-step1-LH} and \eqref{1.2-case2-step1-V}, we immediately have
\begin{equation*}
\begin{aligned}
& - \z \nabla_V u, (4 \gamma G - L^{\ast}_N G + 2F \nabla^2 U) \nabla_V u \y \\ \leq & \delta a[4\gamma + \theta(2\gamma + \| \nabla W\|_{L^{\infty}}) + 2a(C_V^{\theta} + C_K)]|\nabla_V u|^2,
\end{aligned}
\end{equation*}
we choose $a$ such that $a(C^{\theta}_V + C_K) \leq \gamma$, we have
\begin{equation}\label{1.2-case2-step1-VV}
\begin{aligned}
& - \z \nabla_V u, (4 \gamma G - L^{\ast}_N G + 2F \nabla^2 U)\nabla_V u \y \\ \leq & \delta a[6\gamma + \theta(2\gamma + \| \nabla W\|_{L^{\infty}})
]|\nabla_V u|^2.
\end{aligned}
\end{equation}
For cross terms, by \eqref{1.2-case2-step1-LH} and similar argument above, we have
\begin{equation}\label{1.2-case2-step1-VX}
\begin{aligned}
& \z \nabla_X u, (4 \gamma F + 2G - 2 L^{\ast}_NF + 2E \nabla^2 U) \nabla_V u \y \\  \leq & [4 \gamma F + 2G + 2\delta a^2|L^{\ast}_N(H^{-2\theta}(z_i))| + 2 \delta a^3(C_V^{\theta} + C_K)H^{-\theta}]|\nabla_Vu||\nabla_Xu| \\ \leq & \delta a H^{-\theta}[4 + 4\gamma a + 4a \theta(2\gamma + \|\nabla W\|_{L^{\infty}})H^{-\theta} + 2a^2(C_V^{\theta} + C_K)]|\nabla_Vu||\nabla_Xu| \\ \leq & \delta a H^{-\theta}[4 + 6\gamma a + 4a\theta(2\gamma + \|\nabla W\|_{L^{\infty}})]|\nabla_Vu||\nabla_Xu| \\ \leq & \frac{1}{4}\delta a^2 H^{-2\theta}|\nabla_Xu|^2 + \delta [4 + 6\gamma a + 4a\theta(2\gamma + \|\nabla W\|_{L^{\infty}})]^2|\nabla_Vu|^2.
\end{aligned}
\end{equation}
Now we collect \eqref{1.2-case2-step1-VV}, \eqref{1.2-case2-step1-VX} and \eqref{1.2-case2-step1-XX}, we have
\begin{equation}\label{1.2-case2-step1-dissipation}
\begin{split}
- \int f_N\z \nabla u, S \nabla u\y \leq - \frac{3}{4} \delta a^2 \sum_{i=1}^N \int f_N H^{-2\theta}(z_i)|\nabla_{x_i}u|^2 + m_2 \int f_N |\nabla_Vu|^2,  
\end{split}
\end{equation}
where $m_2 = \delta [4 + 6\gamma a + 4a\theta(2\gamma + \|\nabla W\|_{L^{\infty}})]^2 + \delta a[6\gamma + \theta(2\gamma + \| \nabla W\|_{L^{\infty}})]$. Recall we choose $H_0 > 1$ for convenience, and again by Lemma \ref{Lemma:CS-error},
\begin{equation}\label{1.2-case2-error}
\begin{aligned}
\int f_N \z \nabla \overline{R}_N, M \nabla u\y \leq & n'_1 \sum_{i=1}^N \int f_N H^{-2\theta}(z_i)|\nabla_{x_i} u|^2 + n'_2 \int f_N |\nabla_V u|^2 \\ & + \frac{8\delta a^3}{\sigma} \int f_N |R_N^1|^2 + \frac{\delta a^2 \gamma^2(1+a)}{2\sigma^2} \int f_N |R_N^3|^2 \\ & + \big( \frac{2 \gamma^2 \delta a^3}{\sigma^2} + \frac{\gamma^2 \delta^2 a^2}{\sigma^3}\big) \int f_N |R_N^2|^2 \\ & + \sum_{i=1}^N \frac{\sigma}{16} \int f_N | \sqrt{E^{ii}} \nabla_{v_i} \nabla_{x_i} u|_{\R^{2d}}^2 + \frac{\sigma}{16} \int f_N | (F^{ii})^{\frac{1}{4}} \nabla_{v_i} \nabla_{v_i} u|_{\R^{2d}}^2 \\ & + \sum_{i=1}^N \int f_N \langle R^1_{N,i}, (\nabla_{v_i} E^{ii}) \nabla_{x_i} u \rangle_{\R^{2d}} + \int f_N \langle R^1_{N,i}, (\nabla_{v_i} F^{ii}) \nabla_{v_i} u \rangle_{\R^{2d}}.
\end{aligned}
\end{equation}
provided $n'_1 = \frac{1}{4}\delta a^2 + (C_K + \frac{1}{2})\delta a^3$ and $n'_2 = \frac{\sigma}{4} + \delta a^2(C_K a + C_K + \frac{1}{2})$. 
We only need to estimate the last two term. By the computation \eqref{1.2-case2-step1-LH}, we have
\begin{equation}\label{1.2-case2-step1-error-1}
\begin{aligned}
& \int f_N \langle R^1_{N,i}, (\nabla_{v_i} E^{ii}) \nabla_{x_i} u \rangle_{\R^{2d}} \\ \leq & \frac{3 \theta \delta a^3}{2H^{3\theta}_0} \int f_N |R_{N,i}^1|^2 + \frac{3 \theta \delta a^3}{2} \int f_N H^{-2\theta}(z_i)|\nabla_{x_i} u|^2,
\end{aligned}
\end{equation}
and 
\begin{equation}\label{1.2-case2-step1-error-2}
\begin{aligned}
& \int f_N \langle R^1_{N,i}, (\nabla_{v_i} F^{ii}) \nabla_{v_i} u \rangle_{\R^{2d}} \\ \leq & \frac{8\theta \delta^2 a^4}{H_0^{4\theta} \sigma} \int f_N |R_{N,i}^1|^2 + \frac{\sigma}{16} \int f_N |\nabla_{v_i}u|^2.
\end{aligned}
\end{equation}
Now we insert \eqref{1.2-case2-step1-error-1} and \eqref{1.2-case2-step1-error-2} into \eqref{1.2-case2-error}, and combine with \eqref{1.2-case2-step1-dissipation}, we obtain
\begin{equation}\label{1.2-case2-error-final}
\begin{aligned}
\int f_N \z \nabla u, M \nabla u\y \leq & n_1 \sum_{i=1}^N \int f_N H^{-2\theta}(z_i)|\nabla_{x_i} u|^2 + n_2 \int f_N |\nabla_V u|^2 \\ & + (\frac{8\delta a^3}{\sigma} + \frac{8\theta \delta^2 a^4}{H_0^{4\theta} \sigma} + \frac{3 \theta \delta a^3}{2H^{3\theta}_0}) \int f_N |R_N^1|^2 \\ & + \big( \frac{2 \gamma^2 \delta a^3}{\sigma^2} + \frac{\gamma^2 \delta^2 a^2}{\sigma^3}\big) \int f_N |R_N^2|^2 \\ & + \frac{\delta a^2 \gamma^2(1+a)}{2\sigma^2} \int f_N |R_N^3|^2
\end{aligned}
\end{equation}
provided $n_1 = - \frac{1}{2} \delta a^2 + (C_K + \frac{1}{2} + \frac{3\theta}{2}) \delta a^3$ and $n_2 = \frac{5\sigma}{16} + \delta a^2(C_K a + C_K + \frac{1}{2}) + m_2$. By Lemma \ref{lemma:error term of RE} and Young inequality, we have
\begin{equation}\label{1.2-case2-entropy}
\frac{d}{dt} H_{N}(f^t_N|f^{\otimes N}_{\infty}) \leq - \frac{3 \sigma}{4} \frac{1}{N} \int f_N|\nabla_V u|^2 + \frac{1}{\sigma} \frac{1}{N} \int f_N|R^0_N|^2.
\end{equation}
Finally, we combine \eqref{1.2-case2-error-final} and \eqref{1.2-case2-entropy}, and update the constants $a$ and $\delta$ as
\begin{equation}\label{1.2-case2-step1-a}
a \leq \min \Big \{ \frac{1}{4C_K + 6\theta + 2}, \frac{\gamma}{C_V^{\theta} + C_K} \Big \} 
\end{equation}
and 
\begin{equation}\label{1.2-case2-step1-delta}
\delta \leq \frac{3 \sigma}{8 + 32C_K + m_2'},
\end{equation}
where $m_2' = [4 + 6\gamma a + 4a\theta(2\gamma + \|\nabla W\|_{L^{\infty}})]^2 + a[6\gamma + \theta(2\gamma + \| \nabla W\|_{L^{\infty}})]$ comes from \eqref{1.2-case2-step1-dissipation}, we obtain,
\begin{equation}\label{1.2-case2-step1-final}
\begin{aligned}
\frac{d}{dt} \mathcal{E}_N^M(f_N^t|f_{\infty}^{\otimes N}) \leq & - \frac{\delta a^2}{4}\frac{1}{N} \sum_{i=1}^N \int f_N H^{-2\theta}(z_i)|\nabla_{x_i} u|^2 - \frac{\sigma}{4}\frac{1}{N} \int f_N |\nabla_V u|^2 \\ & + (\frac{8\delta a^3}{\sigma} + \frac{8\theta \delta^2 a^4}{H_0^{4\theta} \sigma} + \frac{3 \theta \delta a^3}{2H^{3\theta}_0}) \frac{1}{N} \int f_N |R_N^1|^2 \\ & + \big( \frac{2 \gamma^2 \delta a^3}{\sigma^2} + \frac{\gamma^2 \delta^2 a^2}{\sigma^3}\big) \frac{1}{N} \int f_N |R_N^2|^2 \\ & + \frac{\delta a^2 \gamma^2(1+a)}{2\sigma^2} \frac{1}{N} \int f_N |R_N^3|^2 + \frac{1}{\sigma} \frac{1}{N} \int f_N |R_N^0|^2.
\end{aligned}
\end{equation}

$Step2$. We start from \eqref{1.2-case2-step1-final}. By similar argument of the last case and using Proposition \ref{proposition:LDE1} and \ref{proposition:LDE2},
\begin{equation*}
\frac{1}{N}\int f_N|R^i_N|^2 \leq \frac{1}{\nu}H_N(f^t_N|f^{\otimes N}_{\infty}) + \frac{C}{N}, 
\end{equation*}
where $i = 0,1,2,3$, we have
\begin{equation}\label{1.2-case2-step1-last two}
\begin{aligned}
\frac{d}{dt}  \mathcal{E}_N^M(f_N|f_{\infty}^{\otimes N}) \leq & - \frac{\delta a^2}{4} \frac{1}{N} \int f_N |\nabla_X u|^2 - \frac{\sigma}{4} \frac{1}{N} \int f_N |\nabla_V u|^2 \\ & + \frac{n_4}{\nu} H_N(f_N|f_{\infty}^{\otimes N}) + \frac{C}{N},   
\end{aligned}
\end{equation}
the constant $\nu > 0$ satisfies $(16C^2_Ke^2 \Lambda) \nu < 1$, the constant $C > 0$ depends on $\nu, \sigma, \gamma, K$ and $V$, and the constant $\nu$ collects all coefficients of error terms which reads as
\begin{equation*}
n_4 = \Big[(\frac{8\delta a^3}{\sigma} + \frac{8\theta \delta^2 a^4}{H_0^{4\theta} \sigma} + \frac{3 \theta \delta a^3}{2H^{3\theta}_0}) + \frac{\delta a^2 \gamma^2(1+a)}{2\sigma^2} + \frac{2 \gamma^2 \delta a^3}{\sigma^2} + \frac{\gamma^2 \delta^2 a^2}{\sigma^3} + \frac{1}{\sigma} \Big].
\end{equation*}
Now let us determine the value of $\delta, a, H_0, \nu > 0$ such that
\begin{equation*}
\frac{n_4 \rho_{wls}}{\nu} \leq \min \Big\{\frac{\delta a^2}{8}, \frac{\sigma}{8} \Big\},
\end{equation*}
where $\rho_{wls}$ is the constant of weighted Log-Sobolev inequality constant of $f_{\infty}$ (See Corollary \ref{Cor:LSI-f_infty}). Recall the constant $a > 0$ satisfies
\begin{equation*}
a \leq \min \Big \{ \frac{1}{4C_K + 6\theta + 2}, \frac{\gamma}{C_K + C^{\theta}_V} \Big \} < 1
\end{equation*}
by \eqref{1.2-case2-step1-a}. Now we choose
\begin{equation}\label{1.2-case2-step1-delta-final}
\begin{aligned}
\delta = \frac{\sigma}{8 + 32C_K + m_2''} \leq \frac{3 \sigma}{8 + 32C_K + m_2'}
\end{aligned}
\end{equation}
by \eqref{1.2-case2-step1-delta}, where $m_2'' = [4 + 6\gamma + 4\theta(2\gamma + \|\nabla W\|_{L^{\infty}})]^2 + [6\gamma + \theta(2\gamma + \| \nabla W\|_{L^{\infty}})]$. We also choose $\frac{1}{\nu} = 20C_K^2 e \frac{\sigma}{\gamma}$ with $\Lambda \leq \frac{1}{e \lambda}$ satisfies $\lambda \in (0, \beta)$ in Proposition \ref{proposition:LDE1} and \ref{proposition:LDE2}. Observing that $\delta a^2 < \sigma$ by $a < 1$ and $\delta < \sigma$, we only need
\begin{equation*}
\frac{n_4 \rho_{ls}}{\nu} \leq \frac{\delta a^2}{8}.
\end{equation*}
Recall the lower bound of $H_0$ in $Step1$,
\begin{equation*}
H_0 \geq \max \Big \{\frac{\sigma(d + 3 \theta + 1)}{\gamma}, [3\theta a (2\gamma + \|\nabla W\|_{L^{\infty}})]^{\frac{1}{\theta}}, 1\Big \},
\end{equation*}
by \eqref{1.2-case2-step1-H0} and \eqref{1.2-case2-step1-V}, now we choose $H_0$ such that
\begin{equation}
H_0 \geq \max \{(8\theta \sigma)^{\frac{1}{4\theta}}, (2 \theta\sigma)^{\frac{1}{3\theta}}\}, 
\end{equation}
we have
\begin{equation}
\Big (\frac{8\delta a^3}{\sigma} + \frac{8\theta \delta^2 a^4}{H_0^{4\theta} \sigma} + \frac{3 \theta \delta a^3}{2H^{3\theta}_0} \Big) \leq \frac{10 \delta a^3}{\sigma}. 
\end{equation}
In the next, we make every term of $n_4$ less than $\frac{\delta a^2 \nu}{40\rho_{wls}}$. For the first term of $n_4$,    
\[ \frac{10 \delta a^3}{\sigma} \leq \frac{\delta a^2 \nu}{40 \rho_{wls}} \Longrightarrow a \leq \frac{\gamma}{8000 e\rho_{wls} C_K^2}, \]
we obtain the range of constant $a$
\begin{equation*}
a \leq \min \Big \{ \frac{1}{4C_K + 6\theta + 2}, \frac{\gamma}{C_V^{\theta} + C_K}, \frac{\gamma}{8000 e\rho_{wls} C_K^2} \Big \},  
\end{equation*}
for example, we can take $a$ as
\begin{equation*}
a = \frac{\min\{\gamma,1\}}{8000 e\rho_{wls} C_K^2 + 5C_K + 6\theta + C_V^{\theta} + 2}.
\end{equation*}
For other terms, we obtain the range of diffusion constant $\sigma$ by similar argument of \eqref{1.2-case1-step2-ieq} in the last case,
\begin{equation}
\sigma \geq \max \bigg\{ \frac{800(40+m_2'')\rho_{wls}e}{a^2 \gamma}, 3200\rho_{wls}e\gamma \bigg\} \cdot \max\{C_K^2, C_K^3\},
\end{equation}
Finally, we update \eqref{1.2-case2-step1-last two} as the following
\begin{equation}
\begin{aligned}
\frac{d}{dt}  \mathcal{E}_N^M(f^t_N|f_{\infty}^{\otimes N}) \leq & - \frac{\delta a^2}{16} \frac{1}{N} \int f_N |\nabla_X u|^2 - \frac{\sigma}{16} \frac{1}{N} \int f_N |\nabla_V u|^2 \\ & - \frac{\delta a^2}{16 \rho_{wls}}H_N(f^t_N|f_{\infty}^{\otimes N}) + \frac{C}{N} \\ \leq & -\frac{\delta a^2}{16(\rho_{wls} + 1)}\mathcal{E}_N^M(f^t_N|f_{\infty}^{\otimes N}) + \frac{C}{N},   
\end{aligned}
\end{equation}
we finish the proof of this case.
\end{proof}

\subsection{Proof of Theorem \ref{Thm:theorem1.5}}
In the last subsection, we have shown that $f_N$ converges to $f_{\infty}^{\otimes N}$ in the sense of quantity $\mathcal{E}_N^M(f_N^t|f^{\otimes N}_{\infty})$ as $t \rightarrow \infty$ and $N \rightarrow \infty$. In this subsection, we eliminate the effect of replacement of $f$ by $f_{\infty}$ by proving exponential decay from $f_t$ to $f_{\infty}$. By triangle inequality of Wasserstein distance, the error between $\mathcal{W}_2^2(f^t_N, f_t^{\otimes N})$ and $\mathcal{W}_2^2(f^t_N, f^{\otimes N}_{\infty})$ is only a time exponential decay factor. Combining with short time propagation of chaos result, we obtain uniform-in-time propagation of chaos. \\

\begin{proof}
$Case1.$ In this case, we take weight matrix $M_1$ as \eqref{M:case1-1.2}. We have shown that 
\begin{equation*}
\mathcal{W}_2^2(f^t_N, f^{\otimes N}_{\infty}) \leq H(f^t_N|f_{\infty}^{\otimes N}) \leq  N e^{-c_1t} \mathcal{E}_N^{M_1}(f^0_N|f^{\otimes N}_{\infty}) + C
\end{equation*}
in Theorem \ref{Thm:theorem1.4}, 
where $c_1 = \frac{\delta a^2}{16(1 + \rho_{ls})}$ and $C > 0$ depends on $\sigma, M, C_K$ and $C_V$.  Moreover, Theorem \ref{Thm: theorem-1.1} implies that
\begin{equation*}
\mathcal{W}_2^2(f_t^{\otimes N}, f^{\otimes N}_{\infty}) \leq H(f_t^{\otimes N}|  f^{\otimes N}_{\infty}) \leq \rho_{LS} N e^{-ct} \mathcal{E}^M(f_0|\hat{f}_0)
\end{equation*}
with some $c > 0$ under condition $C_K < 1$. Theorem \ref{Thm:theorem-1.2} also implies the same conclusion under condition that the interaction function $F$ is functional convex (See Assumption \ref{Assumption:W2}). Hence, 
by triangle inequality of $2$-Wasserstein distance, we have
\begin{equation}
\mathcal{W}^2_2(f^t_N,f_t^{\otimes N}) \leq N (1+\rho_{LS})[\mathcal{E}_N^{M_1}(f^0_N|f^{\otimes N}_{\infty}) + \mathcal{E}^M(f_0|\hat{f}_0)]e^{-c_1't} + C,
\end{equation}
where constant $c_1' = \min\{c_1, c\}$. If we take chaotic initial data for Liouville equation \eqref{Eq:Liouville equation}, i.e. $f^0_N = f_0^{\otimes N}$, we have
\begin{equation}
\mathcal{W}^2_2(f^t_N,f_t^{\otimes N}) \leq 2N (1+\rho_{LS})\mathcal{E}^M(f_0|\hat{f}_0)e^{-c_1't} + C.
\end{equation}
For $k$-marginal distribution, we have 
\begin{equation}
\mathcal{W}^2_2(f^t_{N,k}, f_t^{\otimes k}) \leq  2k(1+\rho_{LS})e^{-c_1't} \mathcal{E}^M(f_0|\hat{f}_0) + C\frac{k}{N}.  
\end{equation}

$Case2.$ In this case, we take weight matrix $M_2$ as \eqref{M:case2-1.2}. The main argument of this case is the same with $Case1$. By Theorem \ref{Thm:theorem1.4}, we have
\begin{equation*}
\mathcal{W}_2^2(f^t_N, f^{\otimes N}_{\infty}) \leq H(f^t_N|f_{\infty}^{\otimes N}) \leq  N e^{-c_2t} \mathcal{E}_N^{M_2}(f^0_N|f^{\otimes N}_{\infty}) + C.
\end{equation*}
provided $c_2 = \frac{\delta a^2}{16(1 + \rho_{wls})}$ and $C = C(\sigma, M_2, C_K, C_V^{\theta})$. By Theorem \ref{Thm:theorem-1.2}, we have 
\begin{equation*}
\mathcal{W}_2^2(f_t^{\otimes N}, f^{\otimes N}_{\infty}) \leq H(f_t^{\otimes N}|  f^{\otimes N}_{\infty}) \leq N e^{-c_2''t} \mathcal{E}^{M_2'}(f_0|\hat{f}_0)
\end{equation*}
with constant $c_2'' > 0$ under condition that the interaction function $F$ is functional convex. We can also use the similar result under condition that the constant $C_K$ is small, which have been proved by Guillin, Le Bris and Momarch\'e in \cite{coupling1} (See Theorem 1.1). All in all, we have
\begin{equation*}
\mathcal{W}_2^2(f_t^{\otimes N}, f^{\otimes N}_{\infty}) \leq N  e^{-c''_2t}\mathcal{E}^M(f_0|\hat{f}_0),
\end{equation*}
provided some constant $c_2'' > 0$. Now we take $c_2' = \min\{c_2,c_2''\}$, we have
\begin{equation*}
\mathcal{W}^2_2(f^t_N,f_t^{\otimes N}) \leq N  [\mathcal{E}_N^{M_2}(f^0_N|f^{\otimes N}_{\infty}) + \mathcal{E}^{M'_2}(f_0|\hat{f}_0)]e^{-c'_2t} + C,
\end{equation*}
If we take chaotic initial data for Liouville equation \eqref{Eq:Liouville equation}, i.e. $f^0_N = f_0^{\otimes N}$, we have
\begin{equation}
\mathcal{W}^2_2(f^t_N,f_t^{\otimes N}) \leq 2N  \mathcal{E}^{M'_2}(f_0|\hat{f}_0)e^{-c'_2t} + C.
\end{equation}
For $k$-marginal distribution, we have 
\begin{equation}
\mathcal{W}^2_2(f^t_{N,k}, f_t^{\otimes k}) \leq  2ke^{-c_2't} \mathcal{E}^{M'_2}(f_0|\hat{f}_0) + C \frac{k}{N}.  
\end{equation}
\end{proof}

\section{Appendix}
In appendix, we give the sketch of proof of Theorem \ref{Thm: theorem-1.1} and Theorem \ref{Thm:theorem-1.2}. The main idea of proof has already appeared in other articles, like \cite{momashe2} and \cite{chen2023uniform}. For the completeness of the article, we reprove them under the current setting in the following.

\subsection{Proof of Theorem \ref{Thm: theorem-1.1}}
The main idea originates from Theorem 10 in \cite{UPIandULSI} and Theorem 3 in \cite{momashe2}. By the finite time propagation of chaos result for lipschitz force $\nabla W$ and $\nabla V$, we could conclude that $f_{N,1}(t,z)$ converges to $f(t,z)$ weakly for each $t > 0$, then by lower semi-continuity of relative entropy, we have
\begin{equation*}
H(f|\alpha) \leq \liminf_{N \rightarrow \infty} H(f_{N,1}|\alpha),
\end{equation*}
where $\alpha \propto e^{- \frac{v^2}{2} - V(x)}$, then we get
\begin{equation*}
\begin{split}
\liminf_{N \rightarrow \infty} H_N(f_N|f_{N,\infty}) = \liminf_{N \rightarrow \infty } \left\{ H_N(f_N|\alpha^{\otimes N}) + \frac{1}{2N^2} \int W(x_1 - x_2) f_{N,2} + \frac{1}{N} \log Z_{N,\beta} \right\}.  
\end{split}
\end{equation*}
For the first term, we use Lemma 18 in \cite{UPIandULSI}, we have
\begin{equation}
H_N(f_N| \alpha^{\otimes N}) \geq H(f_{N,1}|\alpha).
\end{equation}
For the second term, we use $|W(x)| \leq C(1 + |x|^2)$ and $f_{N,2}$ weakly converges to $f^{\otimes 2}$ for each $t > 0$, hence
\begin{equation*}
\liminf_{N \rightarrow \infty} \frac{1}{2N^2} \int W(x_1 - x_2) f_{N,2} = \frac{1}{2} \int W \ast \rho \rho, \ \ \rho = \int_{\R^d} f(x,v)dv.
\end{equation*}
For third term, we conclude by \cite{PhLSI} that
\begin{equation*}
\liminf_{N \rightarrow \infty} \frac{1}{N} \log Z_{N, \beta} = - \inf_{f \in \mathcal{M}_1(\Omega)} E(f).
\end{equation*}
Hence we have
\begin{equation*}
H_W(f) \leq \liminf_{N \rightarrow \infty} H_N(f_N|f_{N,\infty}).
\end{equation*} 
Now by Theorem \ref{Thm:theorem1.3} with constant $c > 0$, we obtain that
\begin{equation*}
H_W(f) \leq e^{-ct}\mathcal{E}_N^M(f_0^{\otimes N}|f_{N,\infty}),
\end{equation*}
where the constant $c = c(f_{N,\infty}, M)$ is the same as Theorem \ref{Thm:theorem1.3}.
Furthermore, we use Lemma 17 and Lemma 23 in \cite{UPIandULSI} with $\alpha \propto e^{-\frac{v^2}{2} - V(x)}$, let $N \rightarrow \infty$, we have
\begin{equation*}
H_W(f) \leq e^{-ct}\mathcal{E}^M(f_0|\hat{f}_0).
\end{equation*}
Again, by Theorem 3 in \cite{UPIandULSI}, we have
\begin{equation*}
\mathcal{W}_2^2(f_t, f_{\infty}) \leq \rho_{LS} H_W(f_t) \leq \rho_{LS}\mathcal{E}^M(f_0) e^{-ct}.
\end{equation*}

\subsection{Proof of Theorem \ref{Thm:theorem-1.2}}
In this subsection, we show the convergence from $f$ to $f_{\infty}$ for limiting equation (\ref{Eq:limiting equation}) with more general confinement potential $V$. We use the quantity $\mathcal{E}^M(f_t|\hat{f}_t)$ to show the convergence, which combines free energy and relative Fisher Information of $f_t$ and $\hat{f}_t$, i.e.
\begin{equation}
\mathcal{E}^M(f_t|\hat{f}_t) = \mathcal{F}(f_t) - \mathcal{F}(f_{\infty}) + I^M(f_t|\hat{f}_t),
\end{equation}
where $\hat{f}_t \propto e^{-\frac{1}{2}v^2 - V(x) - W \ast \rho_t}$, the later term is
\begin{equation}
I^M(f_t| \hat{f}_t) = \int f_t \z \nabla u, M \nabla u \y dz, 
\end{equation}
where $\nabla = (\nabla_x, \nabla_v)$, $u = \log \frac{f_t}{\hat{f}_t}$ and $M$ is the weighted matrix which is to be decided. We follow the idea of Theorem 2.1 in \cite{chen2023uniform} to deal with the nonlinear term $\nabla W \ast \rho_t $ of Eq.(\ref{Eq:limiting equation}). Now let us compute the time evolution of $\mathcal{E}^M(f_t|\hat{f}_t)$. We omit the footnote ``$t$" for convenience. For free energy part $\mathcal{F}(f)$, we directly use the result in Theorem 2.1 in \cite{chen2023uniform},
\begin{equation}\label{1.5-v1}
\frac{d}{dt} \mathcal{F}(f) = - \sigma \int f|\nabla_v u|^2 dz.
\end{equation}
For distorted Fisher Information part, we formally write the equation of $\log \hat{f}$ as following
\begin{equation*}
\begin{split}
\partial_t \log \hat{f} + v \cdot \nabla_x \log \hat{f} - [\nabla V + (K \ast \rho)] \cdot \nabla_v \log \hat{f} = & \sigma \frac{\Delta_v \hat{f}}{\hat{f}} + \gamma (d + v \cdot \nabla_v \log \hat{f}) \\ & - \partial_t (\log Z_t + W \ast \rho),
\end{split}
\end{equation*}
by similar computation of Lemma \ref{lemma:Eq-u}, we have the equation of $u = \log \frac{f}{\hat{f}}$,

\begin{equation}
\partial_t u = - B_t u - \sigma A^{\ast} A + \sigma \nabla_v \log f \cdot Au + R, 
\end{equation}
where $B_t = v \cdot \nabla_x - (\nabla V + \nabla W \ast \rho_t) \cdot \nabla_v - \gamma v \cdot \nabla_v$, $A = \nabla_v$, $A^{\ast} = - \nabla_v + v$ which is conjugate operator of $A$ in the sense of $L^2(\hat{f})$, $R = \partial_t (\log Z_t + W \ast \rho)$. Then by Lemma \ref{lemma:Keylemma2} and Corollary \ref{Cor:block matrix}, if we take $
M = 
\left(
\begin{array}{cc}
e  & f \\
f  & g
\end{array}
\right)$, 
where $e,f,g: \Omega \times \R^d \rightarrow \R^{d \times d}$ are smooth matrix-valued functions, then we have
\begin{equation}\label{IEq:evlov-Fisher-f}
\begin{split}
\frac{d}{dt} \left\{\int f \z \nabla u, M \nabla u \y \right\} \leq & - \int f \z \nabla u, S \nabla u \y + 2 \int f \z \nabla R, M \nabla u\y,
\end{split}
\end{equation}
where $S(Z)$ is a matrix reads as
\begin{equation*}
S = 
\left(
\begin{array}{cc}
2e - L^{\ast} e \ \ & 4 \gamma f - L^{\ast} f - 2e \nabla^2 U \\
2c - L^{\ast} f \ \ & 4 \gamma c - L^{\ast} c - 2f \nabla^2 U
\end{array}
\right),
\end{equation*}
and $\nabla^2 U = \nabla^2 V + \nabla^2 W \ast \rho$, and $L^{\ast} = B + \sigma \Delta_v$. Now we follow the similar argument of the Case 2 of Theorem \ref{Thm:theorem1.4}, we select $e(z) = \delta a^3 (H(z))^{-3\theta}, 
f(z) = \delta a^2 (H(z))^{-2\theta}, 
g(z) = 2\delta a(H(z))^{-\theta}$ where $\delta > 0, a > 0$ are to be confirmed. We firstly estimate the upper bound of $L^{\ast}e, L^{\ast}f$ and $L^{\ast}g$. For some $S>0$, we have 
\begin{equation*}
\begin{aligned}
L^{\ast}(H^{-s\theta}(z)) = -s\theta \left\{ \frac{\sigma d}{H} + \frac{\sigma (-3\theta -1)v^2}{H^2} - \frac{\gamma v^2 + v \cdot \nabla W \ast \rho(x)}{H} \right\}H^{-s\theta}(z),
\end{aligned}
\end{equation*}
then we have
\begin{equation*}
|L^{\ast}(H^{-s\theta}(z))| \leq s\theta \big(\frac{\sigma d + \sigma(s\theta + 1)}{H_0} + \gamma + 1 \big) H^{-s \theta}(z). 
\end{equation*}
Now we can replace the norm $\|\nabla W\|_{L^{\infty}}$ by $1$ in the Case 2 in Theorem \ref{Thm:theorem1.4} and use completely the same argument of \eqref{1.2-case2-step1-XX}-\eqref{1.2-case2-step1-VX}, 
we have
\begin{equation}\label{1.5-xv}
\begin{split}
- \int f\z \nabla u, S \nabla u\y \leq - \frac{3}{4} \delta a^2 \int f H^{-2\theta}(z)|\nabla_xu|^2 + m_2 \int f |\nabla_vu|^2,  \end{split}
\end{equation}
where $m_2 = \delta [4 + 6\gamma a + 4a\theta(2\gamma + 1)]^2 + \delta a[6\gamma + \theta(2\gamma + 1)]$ and $a, H_0$ satisfy
\begin{equation*}
H_0 \geq \max \Big \{\frac{\sigma(d + 3 \theta + 1)}{\gamma}, [3\theta a (2\gamma + 1)]^{\frac{1}{\theta}},1 \Big \}, \ \ a \leq \frac{\gamma}{C^{\theta}_V + C_K}.
\end{equation*}
For second term of (\ref{IEq:evlov-Fisher-f}), we have $\nabla R = \nabla W \ast (\partial_t \rho)$, but 
\begin{equation}
\partial_t \rho + \nabla_x \cdot (v_t^x \rho_t) = 0
\end{equation}
where
\begin{equation*}
v_t^x(x) = \frac{\int vf(x,v)dv}{\int f(x,v)dv} = \frac{\int (\nabla_v u) f(x,v)dv}{\int f(x,v)dv}. 
\end{equation*}
By (4.14) of \cite{chen2023uniform}, we then have
\begin{equation*}
|\nabla W \ast (\partial_t \rho)| = |\nabla^2 W \ast (v_t^x \rho_t)| = \left| \int \nabla^2 W(x-y) \cdot v_t^{y} \rho(y)dy \right| \leq C_K \| \nabla_v u \|_{L^2(f)},
\end{equation*}
therefore
\begin{equation}\label{IEq:nonlinear error}
\begin{split}
\int f \z \nabla R, M \nabla u\y \leq & C_K \| \nabla_v u \|_{L^2(f)} \| M \nabla u \|_{L^2(f)} \\ \leq & \frac{\sigma}{4} \int f |\nabla_v u|^2 + \frac{C^2_K}{\sigma} \int f \z \nabla u, M^2 \nabla u \y \\ \leq & \frac{\sigma}{4} \int f |\nabla_v u|^2 + \frac{\delta^2 a^2 C^2_K}{\sigma}[(a^2 + 4)+2(a^2 + 2)^2]\int f |\nabla_v u|^2 \\ & + \frac{\delta^2 a^4 C^2_K}{\sigma}(a^2 + 3)\int f H^{-2\theta}(z)|\nabla_x u|^2.
\end{split} 
\end{equation}
Combining \eqref{1.5-v1}, \eqref{1.5-xv} and \eqref{IEq:nonlinear error}, we have
\begin{equation*}
\begin{aligned}
\frac{d}{dt} \mathcal{E}^M(f|\hat{f}) \leq & (-\frac{3\sigma}{4} + m_2)\int f |\nabla_v u|^2 + \frac{\delta^2 a^2 C^2_K}{\sigma}[(a^2 + 4)+2(a^2 + 2)^2]\int f |\nabla_v u|^2 \\ & + [-\frac{3}{4}\delta a^2 + \frac{\delta^2 a^4 C^2_K}{\sigma}(a^2 + 3)]\int f H^{-2\theta}(z)|\nabla_x u|^2.
\end{aligned}
\end{equation*}
Now we choose $\delta$ such that
\begin{equation}
\left\{
\begin{aligned}
& \frac{\delta^2 a^2 C^2_K}{\sigma}[(a^2 + 4)+2(a^2 + 2)^2] \leq \frac{\sigma}{4}, \\ & \frac{\delta^2 a^4 C^2_K}{\sigma}(a^2 + 3) \leq \frac{1}{4} \delta a^2, \\ & m_2 \leq \frac{\sigma}{4},
\end{aligned}
\right.
\end{equation}
i.e.
\begin{equation}
\delta \leq \min \left\{\frac{\sigma}{2aC_K \sqrt{[(a^2 + 4)+2(a^2 + 2)^2]}}, \frac{\sigma}{4a^2C^2_K(a^2 + 3)}, \frac{\sigma}{4m_2'}\right\},
\end{equation}
where $m_2' = [4 + 6\gamma a + 4a\theta(2\gamma + 1)]^2 + a[6\gamma + \theta(2\gamma + 1)]$,
then we have
\begin{equation}
\frac{d}{dt}\mathcal{E}^M(f|\hat{f}) \leq -\frac{\sigma}{4}\int f |\nabla_v u|^2 - \frac{1}{4}\delta a^2 \int f H^{-2\theta}(z)|\nabla_x u|^2. 
\end{equation}
By weighted Log Sobolev inequality of $\hat{f}$ and $\delta a^2 < \sigma$, we have 
\begin{equation*}
\frac{d}{dt} \mathcal{E}^M(f_t|\hat{f}_t) \leq - \frac{\delta a^2}{8} I^M(f|\hat{f}) - \frac{\delta a^2}{8 \rho_{wls}}H(f|\hat{f}),
\end{equation*}
furthermore, by convexity of functional $F$, we have entropy sandwich inequality $\mathcal{F}(f_t) - \mathcal{F}({f_{\infty}}) \leq H(f_t|\hat{f}_t)$, we have
\begin{equation*}
\frac{d}{dt} \mathcal{E}^M(f_t|\hat{f}_t) \leq  - \frac{\delta a^2}{16 + 16 \rho_{wls}} \mathcal{E}^M(f_t|\hat{f}_t),
\end{equation*}
then we finish the proof.

\subsection*{Acknowledgements}
Zhenfu Wang was partially supported by  the National Key R\&D Program of China, Project Number 2021YFA1002800, NSFC grant No.12171009 and  Young
Elite Scientist Sponsorship Program by China Association for Science and Technology (CAST) No. YESS20200028.

\bibliographystyle{plain}
\bibliography{Reference}

\end{document}